\documentclass[final,onefignum,onetabnum]{siamart171218}

\usepackage{graphicx}%
\usepackage{multirow}%
\usepackage{amsmath,amssymb,amsfonts}%

\usepackage{mathrsfs}%
\usepackage{xcolor}%
\usepackage{textcomp}%
\usepackage{manyfoot}%
\usepackage{booktabs}%
\usepackage{algorithm}%

\allowdisplaybreaks

\let\a=\alpha
\let\e=\varepsilon

\let\d=\delta

\let\pt=\partial

\let\G=\Gamma

\let\g=\gamma
\let\th=\theta

\let\b=\beta
\let\k=\kappa

\newcommand\normmm[1]{\left\vert\kern-0.25ex \left\vert\kern-0.25ex \left\vert #1 \right\vert\kern-0.25ex \right\vert\kern-0.25ex \right\vert}

\newcommand\normm[1]{\left\lVert#1\right\rVert}
\newcommand\norm[1]{\left\lvert#1\right\rvert}

\newcommand\normmb[1]{\left\lbrack\kern-0.25ex\left\lbrack\ #1 \kern0.75ex\right\rbrack\kern-0.25ex\right\rbrack}

\newcommand\inn[1]{\left\langle#1\right\rangle}

\newcommand{\Hquad}{\hspace{0.5em}}

\newcommand{\R}{\mathbb{R}}

\renewcommand{\P}{\mathbf{P}}
\renewcommand{\S}{\mathbf{S}}
\newcommand{\ip}{(\mathbf{I}-\mathbf{P})}

\newcommand{\be}{\begin{equation}}
\newcommand{\bee}{\begin{equation*}}

\newcommand{\bm}{\begin{multline}}
\newcommand{\ee}{\end{equation}}
\newcommand{\eee}{\end{equation*}}

\newcommand{\dd}{\mathrm{d}}

\newcommand{\1}{\mathbf{1}}

\usepackage{lipsum}
\usepackage{amsfonts}
\usepackage{graphicx}
\usepackage{epstopdf}
\usepackage{algorithmic}
\ifpdf
  \DeclareGraphicsExtensions{.eps,.pdf,.png,.jpg}
\else
  \DeclareGraphicsExtensions{.eps}
\fi

\newsiamremark{remark}{Remark}
\newsiamremark{hypothesis}{Hypothesis}
\crefname{hypothesis}{Hypothesis}{Hypotheses}
\newsiamthm{claim}{Claim}

\headers{Global Diffusive expansion of Boltzmann Equation in exterior Domain}{Junhwa Jung}

\title{Global Diffusive expansion of Boltzmann Equation in exterior Domain\thanks{Submitted to the editors DATE.
\funding{This research is funded in part
by a Kwanjeong foundation and NSF Grant DMS-2106650.}}}

\author{Junhwa Jung \thanks{Division of Applied Mathematics, Brown University,
(\email{junhwa\_jung@brown.edu}).}}

\usepackage{amsopn}

\makeatletter
\newcommand*{\addFileDependency}[1]{
  \typeout{(#1)}
  \@addtofilelist{#1}
  \IfFileExists{#1}{}{\typeout{No file #1.}}
}
\makeatother

\ifpdf
\hypersetup{
  pdftitle={Global Diffusive expansion of Boltzmann Equation in exterior Domain},
  pdfauthor={Junhwa Jung}
}
\fi

\begin{document}

\maketitle

\begin{abstract}
The study of flows over an obstacle is one of the fundamental problems in fluids. In this paper we establish the global validity of the diffusive expansion for the Boltzmann equations to the Navier-Stokes-Fourier system in an exterior domain. To overcome the well-known difficulty of the lack of Poincare's inequality in the unbounded domain, we develop a new $L^2-L^3-L^6$ splitting to extend the $L^2-L^\infty$ framework into the unbounded domain.
\end{abstract}



\section{Introduction}

\subsection{Background}
The purpose of this paper is to establish the incompressible Navier-Stokes-Fourier expansion from the unsteady Boltzmann equation with diffusive boundary conditions in the exterior domain. The rescaled Boltzmann equation with small Knudsen and Mach numbers is given by the following equation:
\begin{eqnarray}
&& \label{Boltzeqfull} \e \pt_t F + v \cdot \nabla_x F = \frac{1}{\e}Q(F,F), \quad \text{in } \R_{+} \times \Omega^c \times \R^{3} \\
&& \label{Boltzbdd} F|_{\g_{-}} = \mathscr{P}^{w}_{\gamma}(F) \\
&& F(t,x,v)|_{t=0} = F_{0}(x,v) \ge 0, \quad \text{on } \Omega^{c} \times \R^3.
\end{eqnarray}
where $F(t,x,v)$ is the distribution of particles at time $t \ge 0$, position $x$ and velocity $v$. The Boltzmann collision operator $Q$ can be written as
\begin{eqnarray} \label{collisionoperator}
Q(f,g)(v) &=& Q^{+}(f,g) - Q^{-}(f,g) \\
Q^{+}(f,g)(v) &=& \iint_{\R^3 \times \S^2} B(\omega,{v-u}) f(v')g(u') \dd \omega \dd u\\
Q^{-}(f,g)(v) &=& f(v) \iint_{\R^3 \times \S^2} B(\omega,{v-u}) g(u) \dd \omega \dd u,
\end{eqnarray}
where $v'$ and $u'$ are defined as the velocities of particles after collision
\begin{eqnarray} \label{u'v'defn}
v' = v - \omega(v-u) \cdot \omega, \quad u' = u + \omega(v-u) \cdot \omega,
\end{eqnarray}
and $B(\omega,{v-u}) $ is the cross section for hard potentials with Grad's angular cut off, so that $\int_{\{\norm{\omega}=1\}} B(\omega, V) \dd \omega \lesssim \norm{V}^{\th}$ for $0\le \th \le 1$, depending on the interaction potential.

Throughout this paper, $\Omega = \{x : \xi(x) <0 \}$ is a general bounded domain in $\R^3$. Denote its boundary as $\pt \Omega = \{x : \xi(x) = 0 \}$, where $\xi(x)$ is a $C^2$ function. We assume $\nabla \xi(x) \neq 0$ at $\pt \Omega$. The inner normal vector at $\pt \Omega$ is given by
\begin{eqnarray*}
n(x) = -\frac{\nabla \xi(x)}{\norm{\nabla \xi(x)}},
\end{eqnarray*}
and it can be smoothly extended near $\pt \Omega$. Let $\gamma = \pt \Omega \times \R^3 = \gamma_{+} \cup \gamma_{0} \cup \gamma_{-}$ with
\begin{eqnarray}
\label{defboundary}
\gamma_{\pm} = \left\{(x,v) \in \pt \Omega \times \R^3 : n(x) \cdot v \gtrless 0 \right\},\\
\gamma_{0} = \left\{(x,v) \in \pt \Omega \times \R^3 : n(x) \cdot v = 0 \right\}.
\end{eqnarray}

The interaction of the gas with the boundary $\pt \Omega$ is given by the diffusive boundary condition, defined as follows
\[
M_{\rho, u, T} := \frac{\rho}{(2 \pi T)^{\frac{3}{2}}} exp\left[-\frac{\norm{v-u}^2}{2T}\right],
\]
be the local Maxwellian with density $\rho$, mean velocity $u$, and temperature $T$
and define global (or absolute) Maxwellian as
\[
\mu = M_{1,0,1} = \frac{1}{(2 \pi )^{\frac{3}{2}}} exp\left[-\frac{\norm{v}^2}{2}\right]. 
\]

On the boundary, $F$ satisfies the diffusive reflection condition, defined as 
\[
F(t,x,v) = \mathscr{P}^{w}_{\gamma}(F)(t,x,v) \quad \text{on } \gamma_{-},
\]
where
\[
\mathscr{P}^{w}_{\gamma}(F)(t,x,v) := M^{w}(x,v) \int_{\g_{+}} F(t,x,w)(n(x) \cdot w) \dd w,
\]
with wall Maxwellian defined as
\[
M^{w} = \sqrt{2 \pi} \mu, \quad \int_{ \{v \cdot n \gtrless 0\} } M^{w}(v) \norm{n \cdot v} \dd v =1.
\]
Since the domain is a complement of compact domain, we need to specify the condition at infinity. In this paper, we assume in a suitable sense
\[
\lim_{\norm{x} \to \infty} F(x,v) = \mu(v).
\]

In the study of hydrodynamic limit, Hilbert proposed Hilbert expansion as follows
\begin{eqnarray}
\label{hilbertexpansion} F = \mu + \sqrt{\mu}\sum_{n=1}^{\infty} \e^{n} f_{n}.
\end{eqnarray}

To determine the coefficient of $f_1, \cdots f_n, \cdots $ in such expansion, we plug the formal expansion \eqref{hilbertexpansion} into \eqref{Boltzeqfull}
\begin{eqnarray} \label{Boltzepower}
(\e \pt_t + v \cdot \nabla)\sum_{n=1}^{\infty} \e^n f_n = \frac{1}{\e \sqrt{\mu}}Q(\mu + \sqrt{\mu}\sum_{n=1}^{\infty} \e^n f_n, \mu + \sqrt{\mu} \sum_{n=1}^{\infty} \e^n f_n) .
\end{eqnarray}

To expand right hand side, define $L$, linearized Boltzmann operator as
\begin{eqnarray}
L f := - \mu^{\frac{1}{2}} [Q(\mu, \mu^{\frac{1}{2}}f) + Q( \mu^{\frac{1}{2}}f, \mu)] : = \nu f - K f,
\end{eqnarray}
where $\nu(v) = \int_{\R^3 \times \{ \norm{\omega} = 1 \}}q(v-v',\omega) \dd v' \dd \omega$. There exist constants $\nu_0$ and $\nu_1$ such that $0 \le \nu_0 \norm{v}^{\th} \le \nu(v) \le \nu_1 \norm{v}^{\th}$ and $K$ is a compact operator on $L^{2}(\R^3_{v})$. 
Define the nonlinear collision operator $\G$ as
\begin{eqnarray} \label{Gdefn}
\G(f,g) = \frac{1}{\sqrt{\mu}}Q(\sqrt{\mu}f,\sqrt{\mu}g).
\end{eqnarray}
Now comparing the $\e$ power of the both side in the equation \eqref{Boltzepower}, we obtain the following equations
\begin{eqnarray*}
0 &=& - L f_1, \\
v \cdot \nabla_x f_1 &=& - L f_2 + \G(f_1,f_1),\\
\pt_t f_1 + v \cdot \nabla_x f_2  &=& -L f_3 + \G(f_1,f_2) + \G(f_2,f_1), \\
&\vdots& \\
\pt_t f_{n} + v \cdot \nabla_x f_{n+1}  &=& -L f_{n+2} + \sum_{i+j=n+2, i,j\ge 1} \G(f_{i},f_{j}).
\end{eqnarray*}

From the above relation we can obtain the fluid equations. We can easily check that  $L$ is an operator on $L^{2}(\R^3_{v})$ whose kernel is 
\begin{eqnarray}
\ker L = \text{span} \{1, v, \norm{v}^2 \} \sqrt{\mu}.
\end{eqnarray}
Let $\P$ be the orthogonal projection on $\ker L$. Then, we can decompose any function $g(t,x,v)$ as follows
\begin{eqnarray}
g = \P g + \ip g,
\end{eqnarray}
where $\P g$ is the fluid parts and we denote
\begin{eqnarray}
\P g(t,x,v) = \left\{ \rho_{g}(t,x) + v \cdot u_{g}(t,x) + \frac{\norm{v}^2-3}{2}\th_{g}(t,x)\right\} \sqrt{\mu}.
\end{eqnarray}

In this paper, we will construct the solution in the exterior domain. Define $F$ as a small perturbation of $\mu$ which satisfies the following relation
\begin{eqnarray}
F = \mu + \sqrt{\mu}(\e f_1 + \e^2 f_2 + \e^{\frac{3}{2}}R).
\end{eqnarray}
According to the several references (\cite{Bardos1991-1}, \cite{Bardos1993}, \cite{Cercignani1994},\cite{Guo2006}), $f_1$ satisfies the following relation:
\begin{eqnarray} \label{deff1}
f_1 = \P f_1 = \left(\rho(x) + u(x) \cdot v + \th (x)\frac{\norm{v}^2-3}{2} \right)\sqrt{\mu},
\end{eqnarray}
where $\rho, u, \th$ satisfies the coupled Navier-Stokes equation in $\Omega^c$. Since our domain is exterior domain the following relation contains boundary condition and infinity condition. Thus, we can get the following equations.
\begin{eqnarray}
\pt_t u + u \cdot \nabla u  + \nabla p =  \eta \Delta u, \quad  \nabla \cdot u = 0  &\text{ in }& \Omega^{c}, \label{nseq}\\
\pt_{t} \th + u \cdot \nabla \th = \kappa \Delta \th &\text{ in }& \Omega^{c} \label{foureq}, \\
u(x,0) = u_0 (x), \quad \th(x,0) = \th_{0} (x) &\text{ in }& \Omega^{c}, \nonumber\\
u(x) = 0, \quad \th(x) = 0 &\text{ on }& \pt \Omega,  \nonumber \\
u \to 0, \quad \th \to 0 &\text{ as }& \norm{x} \to \infty, \nonumber \\
\rho + \th = 0&\text{ in }& \Omega^{c}.
\nonumber
\end{eqnarray}

Solution to this equation do exist in $L^p$, for any $p>2$(\cite{Kato1984}, \cite{Wigner2000}). 
In addition, define $f_2$ as
\begin{eqnarray} \label{deff2}
f_2 &=& \frac{1}{2} \sum_{i,j=1}^{3} \mathscr{A}_{ij} \left[\pt_{x_i}u_{j} + \pt_{x_j}u_{i} \right] + \sum_{i,j=1}^{3} \mathscr{B}_{i} \pt_{x_i} \th\\
\nonumber && + L^{-1}\left[\G\left(f_1,f_1\right) \right] + v \cdot u_{2} \sqrt{\mu},
\end{eqnarray}
where, $\mathscr{A}_{ij}$ and $\mathscr{B}_{i}$ are given by
\begin{eqnarray} \label{ABdefine}
\mathscr{A}_{ij} = -L^{-1} \left(\sqrt{\mu} (v_i v_j -\frac{\norm{v}^2}{3} \delta_{ij} ) \right), \quad \mathscr{B}_{i} = -L^{-1} \left(\sqrt{\mu} v_i \frac{\norm{v}^2-5 }{2} \right)
\end{eqnarray}
and $u_2 =  - \nabla \Delta^{-1} \rho_t  $.

Then, the equation for $R$ can be written as follows
\begin{align*} \label{BoltzR}
\e \pt_t R + v \cdot \nabla_x R + \e^{-1}L(R) = h + \tilde{L}(R) + \e^{1/2} \G(R,R),
\end{align*}
where $\G(f,g) = \frac{1}{\sqrt{\mu}}Q(\sqrt{\mu}f,\sqrt{\mu}g)$ ,$L(f) = -\G(\sqrt{\mu},f)$, $\tilde{L}(R) = \G(f_1+\e f_2,R) +\G(R,f_1+\e f_2)$ and

\begin{align} \label{eqofh}
h=-\e^{1/2} \pt_t (f_1+\e f_2) -\e^{-1/2} v \cdot \nabla_x (f_1+\e f_2) - \e^{-3/2}L(f_1+\e f_2) + \e^{-1/2} \G(f_1+\e f_2,f_1+\e f_2).
\end{align}
Since $u \to 0$ at $\infty$, $f_1$ and $f_2$ also go to 0 at $\infty$. Thus, we have to impose
\begin{eqnarray}
\lim_{\norm{x} \to \infty} R = 0.
\end{eqnarray}

\subsection{Basic Notation}

\begin{definition}
In this paper we define
\begin{eqnarray*}
\normm{f}_{L^{p}_{t}L^{q}_{x}L^{r}_{v}} &:=& \left(\int_{t}\left( {\int_{x}
\left({\int_{v} \norm{f}^{r} \dd v}\right)^{q/r} \dd x ^{} }\right)^{p/q} \dd t \right)^{1/p}, \\
\normm{f(t)}_{L^{q}_{x}L^{r}_{v}} &:=& \left( {\int_{x}
\left({\int_{v} \norm{f(t)}^{r} \dd v}\right)^{q/r} \dd x ^{} }\right)^{p/q}, \\
\nu (v) &:=& \int_{\R^3}B(\omega, \norm{v-u}) \mu(u) \dd u, \\
\normm{f}_{\nu}&:=&\normm{\nu(v)^{1/2}f}_{L^{2}_{x,v}}, \\
\normm{f}_{L^{p}_{t}\nu}&:=&\normm{\nu(v)^{1/2}f}_{L^{p}_{t}L^{2}_{x,v}}, \\
\norm{f}_{2,+}&:=& \left( \int_{\gamma_+} f^2 \dd \gamma \right)^2 := \left( \int_{\gamma_+} f^2 \left\{ n(x) \cdot v \right\} \dd s \right)^2, \\
\norm{f}_{L^p_{t}2,+}&:=& \left(\int_{t} \norm{f}_{2,+}^p \dd t \right)^{1/p}, \\
\inn{f,g} &:=& \int_{\R^3} fg \dd v.
\end{eqnarray*}
If $p$ or $q$ is infinity, then integral should changed to $\sup$ norm. To simplify the notation, we denote $\normm{\cdot}_{L^{p}}$ or $\normm{\cdot}_{p}$ as the standard $L^p$ norm with respect to the $x$ and $v$ variables. We will indicate the variable $t$ if time integration is also considered.
\end{definition}

\subsection{Boundary condition}

In this paper, we will focus on diffusive boundary condition. The boundary condition is given as follows.

For $f \in L^1(\g_{\pm})$ we define
\begin{eqnarray}
\mathscr{P}_{\gamma} f = \mu^{-\frac{1}{2}} \mathscr{P}^{w}_{\gamma}[\mu^{\frac{1}{2}} f] .
\end{eqnarray}

The boundary condition for $R$ is
\begin{eqnarray}
R = \mathscr{P}_{\gamma} R + \e^{1/2} r,
\end{eqnarray}
where 
\begin{eqnarray}
r = \mathscr{P}_{\gamma}[f_2] - f_2.
\end{eqnarray}

From the definition of $r$ it follows that 
\begin{eqnarray}
    \norm{r}_{2,+} +\norm{r}_{\infty} \lesssim \normm{f_2}_{L^2} + \normm{f_2}_{L^{\infty}}.
\end{eqnarray}

\subsection{Main result}

The following theorem is the main result of this paper.
\begin{theorem} \label{mainthm}
Let $\Omega$ be a $C^2$ bounded open set of $\R^3$ and $\Omega^c = \R^3 \backslash \bar{\Omega}$. For any $0< \e \ll 1$ consider the boundary value problem
\begin{equation} \label{thm1eq}
\left\{ \begin{array}{l l}
\e \pt_t F + v \cdot \nabla_x F = \frac{1}{\e}Q(F,F) & \quad \text{in } \R_{+} \times \Omega^c \times \R^{3}, \\
F|_{\g_{-}} = \mathscr{P}^{w}_{\gamma}(F) & \quad \text{on } \g_{-},\\
\lim_{\norm{x} \to \infty} F(t,x,v) = \mu(v), & \\
F(t,x,v)|_{t=0} = F_{0}(x,v) \ge 0 & \quad \text{on } \Omega^{c} \times \R^3.
\end{array} \right.
\end{equation}
Suppose the initial datum takes the form of $F_0 = \mu + \e \sqrt{\mu} f_0 \ge 0$ such that
\begin{equation}
f_0 = f_{1,0} + \e f_{2,0} + \e^{1/2}R_0,
\end{equation}
where
\begin{eqnarray}
f_{1,0} &=& \left(\rho_0(x) + u_0(x) \cdot v + \th_0 (x)\frac{\norm{v}^2-3}{2} \right)\sqrt{\mu},\\
f_{2,0} &=& \frac{1}{2} \sum_{i,j=1}^{3} \mathscr{A}_{ij} \left[\pt_{x_i}u_{0,j} + \pt_{x_j}u_{0,i} \right] + \sum_{i,j=1}^{3} \mathscr{B}_{i} \pt_{x_i} \th\\
\nonumber && + L^{-1}\left[\G\left(f_{1,0},f_{1,0}\right) \right] + v \cdot u_{0,2} \sqrt{\mu},
\end{eqnarray}
satisfies
\begin{eqnarray}
\nabla \cdot u_0 &=& 0, \\
\normm{\rho_0}_{L^{p}} + \normm{u_0}_{L^{p}} &\ll& 1,\\
\normm{R_0}_{L^2_{x,v}} + \normm{\pt_t R_0}_{L^2_{x,v}} + \e^{1/2} \normm{\omega R_0}_{L^\infty_{x,v}} + \e^{3/2} \normm{\omega \pt_t R_0}_{L^\infty_{x,v}} &\ll& 1,
\end{eqnarray}
where $\omega(v) = e^{\b \norm{v}^2}$ with $0 < \b \ll 1$ and $p \in [1,\infty]$.

Then, the problem \eqref{thm1eq} has a unique global solution, which can be represented as
\begin{eqnarray}
F= \mu + \e \sqrt{\mu}(f_1 +\e f_2+ \e^{1/2}R),
\end{eqnarray}
with $f_1$ and $f_2$ given by \eqref{deff1} and \eqref{deff2}.
\end{theorem}
\begin{remark}
$\pt_t R_0$ doesn't make sense in general. However, it can be obtain from the equation by $\pt_t R_0 = -v \cdot \nabla_{x} R_0 - \e^{-1}L(R_0) + h + \tilde{L}(R_0) + \e^{1/2} \G(R_0,R_0)$.
\end{remark}
\begin{remark}
We remark that $\e^{1/2} R$ is of higher order in $L^p$ for $2 \le p \le 6$. On the other hand, $\e^{1/2} R$ is small, but not infinitesimal in $\e$ in $L^{\infty}$, making it possible that $(F-(\mu+\e \sqrt{\mu}f_{1})) / \e = O(1)$ in $L^{\infty}$.
\end{remark}
\begin{remark}
Domain $\Omega$ does not need to be connected.
If $\Omega = \emptyset$, the theorem holds for whole domain($\R^3$).
\end{remark}

\subsection{Mathematical difficulty}
One of the major difficulties in the boundary value problem of the Boltzmann equation is controlling the grazing sets, denoted as $\g_0$ from \eqref{defboundary}. To handle this problem, Guo introduced $L^2-L^\infty$ frameworks.

Based on the energy estimate and the positivity of $L$, we obtain
\begin{eqnarray}
\normm{R(t)}^{2}_{2} + \e^{-2} \int_{0}^{t} \normm{\ip R}_{\nu}^{2} \lesssim \e \int_{0}^{t} \normm{\G(R,\pt_t R)}^2_{2} +1.
\end{eqnarray}
The main nonlinear term is
\begin{eqnarray}
\e \int_{0}^{t} \normm{\G(\P R, \pt_t \P R)}^2_{2}.
\end{eqnarray}

One of the major difficulties in the $L^2-L^\infty$ estimate is controlling $\P f$. In the bounded domain, Esposito et al. (\cite{Esposito2018-apde}) controlled $\P f$ as follows:
\begin{align*}
\int_{s}^{t} \normm{\P f (\tau)}_{2} \dd \tau \lesssim& G(t) - G(s) + \int_{s}^{t}\normm{\frac{g(\tau)}{\sqrt{\nu}}}_{2}^{2} \dd \tau + \int_{s}^{t} \norm{r(\tau)}_{2,-}^{2} \dd \tau \\
&+ \e^{-2} \int_{s}^{t} \normm{\ip f(\tau)}_{\nu}^2 \dd \tau + \int_{s}^{t} \norm{(1-P_{\g}^{ })f}_{2,-}^2 \dd \tau.
\end{align*}
What they used were a duality argument and the Poincare inequality to control $\P R$.
However, Poincare inequality does not hold in an unbounded domain.
In this paper, we develop a new splitting of $\P f$ to extend $L^2-L^\infty$ framework to the unbounded domain. The main contribution of this paper is to suggest an $L^2 - L^3 - L^6$ splitting of the $\P f$ terms. Using this splitting, we can derive a dissipative estimate without relying on the Poincare inequality.

The linearlized equation is
\begin{eqnarray}
\label{Boltzwithg}
\e \pt_t f + v \cdot \nabla f + \e^{-1} L f = g,  &\quad& \text{ in } \R_{+} \times \Omega^{c} \times \R^3\\
f = \mathscr{P}_{\gamma} f + \e^{1/2} r, &\quad& \text{ on } \g_{-} 
\end{eqnarray}
where 
\begin{eqnarray}
r = \mathscr{P}_{\gamma}[f_2] - f_2.
\end{eqnarray}
\begin{theorem}
\label{l2l6estimate}
($L^2$-$L^3$-$L^6$ splitting)
Let us assume $f$ satisfies \eqref{Boltzwithg} and $\P f(t,x,v) = (a(t,x) + b(t,x) \cdot v + c(t,x) \frac{\norm{v}^2-3}{2})\sqrt{\mu}$. Then, there exist splitting $\P_{1}f, \P_{2}f$ and $\P_{3}f$ of $\P f$, where
\begin{align*}
\P_{1} f &:= (a_1 + b_1 \cdot v + c_1 \frac{\norm{v}^2-3}{2})\sqrt{\mu}, \\
\P_{2} f &:= (a_2 + b_2 \cdot v + c_2 \frac{\norm{v}^2-3}{2})\sqrt{\mu}, \\
\P_{3} f &:= (a_3 + b_3 \cdot v + c_3 \frac{\norm{v}^2-3}{2})\sqrt{\mu},
\end{align*}
and
\begin{align*}
a = a_1 +a_2 + a_3, \quad
b = b_1 +b_2 + b_3, \quad
c = c_1 +c_2 + c_3.
\end{align*}
Definition of the $a_{i}, b_{i}$ and $c_{i}$($i=1,2,3$) can be seen in \eqref{atildedefn}, \eqref{abardefn}, \eqref{agammadefn}, \eqref{btildedefn}, \eqref{bbardefn}, \eqref{bgammadefn}, \eqref{ctildedefn}, \eqref{cbardefn} and \eqref{cgammadefn}.
Then the following estimate holds
\begin{align}
\label{a16estimate}
&\normm{a_{1}(t)}_{L^6} \lesssim \e^{-1} \normm{\ip f(t)}_{\nu} +\e\normm{\pt_t f (t)}_{2} + \normm{g (t)\nu^{-1/2}}_{2},\\ \label{bc16estimate}
&\normm{b_{{1}}(t)}_{L^6} +\normm{c_{1}(t)}_{L^6} \lesssim \e^{-1} \normm{\ip f(t)}_{\nu} +\e\normm{\pt_t \ip f (t)}_{2} + \normm{g(t) \nu^{-1/2}}_{2},\\ \label{abc22estimate}
&\normm{a_{2}(t)}_{L^2} + \normm{b_{{2}}(t)}_{L^2} + \normm{c_{2}(t)}_{L^2} \lesssim  \normm{\ip f(t)}_{\nu},\\ \label{abc33estimate}
&\normm{a_{3}(t)}_{L^3} + \normm{b_{{3}}(t)}_{L^3} + \normm{c_{3}(t)}_{L^3} \lesssim  \norm{(1-P_{\g}^{ })f(t)}_{2,+} + \e^{1/2} \norm{r(t)}.
\end{align}
\end{theorem}
\begin{remark}
Note that \eqref{bc16estimate}, \eqref{abc22estimate} and \eqref{abc33estimate} are included in the dissipation, but \eqref{a16estimate} is not due to $\normm{\pt_t f}_{2}$ term. However, thanks to the structure of the nonlinear term, we do not need to control the $a$ terms to manage the nonlinear term. Therefore, we define $\P_{1}'$, which excludes $a_1$.
\begin{definition}
\begin{align}
\P_{1}' f := (b_1 \cdot v + c_1 \frac{\norm{v}^2-3}{2})\sqrt{\mu}.
\end{align}
\end{definition}
With $L^2 - L^3 - L^6$ splitting we can bound nonlinear term as
\begin{align*}
&\e^{1/2}\normm{\G (\P \pt_t f, \P f)}_{L^2_{t,x,v}} \\ 
&\lesssim \e^{1/2} \normm{\P \pt_t f}_{L^{\infty}_{t}L^{3}_{x,v}} \normm{\P_{1}' f}_{L^{2}_{t} L^{6}_{x,v}} + \e^{3/2} \normm{\P \pt_t f}_{L^{\infty}_{t}L^{\infty}_{x,v}}\e^{-1} \normm{\P_{2}f}_{L^{2}_{t} L^{2}_{x,v}} \\
& +\e \normm{\P \pt_t f}_{L^{\infty}_{t}L^{6}_{x,v}} \e^{-1/2} \normm{\P_{3} f}_{L^{2}_{t} L^{3}_{x,v}} \\
&\lesssim \normm{\P \pt_t f}_{L^{\infty}_{t}L^{2}_{x,v}}^{2/3} \normm{\e^{3/2} \P \pt_t f}_{L^{\infty}_{t}L^{\infty}_{x,v}}^{1/3} \normm{\P_{1}' f}_{L^{2}_{t} L^{6}_{x,v}} + \normm{\e^{3/2} \P \pt_t f}_{L^{\infty}_{t}L^{\infty}_{x,v}} \normm{\e^{-1}\P_{2}f}_{L^{2}_{t} L^{2}_{x,v}} \\
& +\normm{\P \pt_t f}_{L^{\infty}_{t}L^{2}_{x,v}}^{1/3} \normm{\e^{3/2} \P \pt_t f}_{L^{\infty}_{t}L^{\infty}_{x,v}}^{2/3} \normm{\e^{-1/2}\P_{3} f}_{L^{2}_{t} L^{3}_{x,v}}.
\end{align*}
\end{remark}
\begin{remark}
Due to the structure of $\P_1$, $\P_2$, $\P_3$, and $\P$, the $L^p$ norms with respect to $v$ are equivalent for $1 \le p \le \infty$.
\end{remark}

To illustrate the idea of such $L^2-L^3-L^6$ splitting, we rewrite equation \eqref{Boltzwithg} as follows:
\begin{eqnarray*}
\e \pt_t f + v \cdot \nabla \P f + v \cdot \nabla \ip f + \e^{-1} Lf = g.
\end{eqnarray*}
We split $\P f$ into three parts as follows:
\begin{align*}
v \cdot \nabla \P_{1} f &\approx -\e \pt_t f - \e^{-1} Lf +g, \\
v \cdot \nabla \P_{2} f &\approx -v \cdot \nabla \ip f, \\
v \cdot \nabla {\P}_{3} f &\approx f|_{\g}.
\end{align*}
The precise definitions of $a_{2}$, $b_{{2}}$, $c_{2}$, $a_{3}$, $b_{{3}}$, and $c_{3}$ are given in Section 5. Intuitively, these terms represent the weak solutions of the following equation:
\begin{eqnarray*}
\int_{\Omega^c} {\cdot}_{2} \Delta \varphi \dd x &=& \int_{\Omega^c \times \R^3} \ip f \sum_{i=1}^{3} \sum_{j=1}^{3} v_j \phi \pt_i \pt_j \varphi_{i,j} \dd x \dd v,\\
\int_{\Omega^c} {\cdot}_{3} \Delta \varphi \dd x &=& 
\int_{\g_{-}} ((P_{\g}^{ }-1) f + \e^{\frac{1}{2}}r)\sum_{i=1}^{3} \sum_{j=1}^{3} v_j \phi_{i,j} \pt_i \pt_j \varphi \dd \g.
\end{eqnarray*}
We may choose $\phi_{i,j}$ for each $a,b$ and $c$.

Formally speaking, by the $H^{-1}$ estimate, we can obtain:
\begin{align*}
\normm{\P_{2} f}_{L^2_x} \lesssim \normm{\ip f}_{L^2_x}.
\end{align*}
Because of the trace inequality we can obtain:
\begin{align*}
\normm{{\P}_{3} f}_{L^3_x} \lesssim \norm{(1-P_{\g}^{ })f}_{2,+} + \e^{1/2} \norm{r}.
\end{align*}
On the other hands, we note that:
\begin{align*}
-\e \pt_t f - \e^{-1} Lf +g \in L^2.
\end{align*}
Then, we establish an $L^6$ estimate (with the same scaling as $H^1$ in $\R^3$) as follows:
\begin{align*}
\normm{\P_{1} f}_{L^6} &\lesssim \e^{-1} \normm{\ip f}_{\nu} +\e\normm{\pt_t f }_{2} + \normm{g \nu^{-1/2}}_{2}.
\end{align*}

\subsection{Historical background}
In Hilbert's sixth question, he asked for the mathematical justification of the relationship between the Boltzmann equation and fluid equations in the context of small Knudsen numbers(\cite{Hilbert1900}, \cite{Hilbert1916}, \cite{Hilbert_}).
The present work partially answer the hydrodynamic limit of Boltzmann equation. In 1971, Sone discussed the hydrodynamic limit of the Boltzmann equation in steady case(\cite{Sone1971}). Caflisch and Lachowicz obtained the Hilbert expansion of the Boltzmann equation, which relating it to the solution of the compressible Euler equations for small Knudsen numbers(\cite{Caflisch1980}, \cite{Lachowicz1987}). On the other hand, Nishida, Asano, and Ukai proved it using different methods(\cite{Nishida1978},  \cite{Asano1983}).

We can obtain the incompressible Navier-Stokes equation from the Boltzmann equation, when both the Mach number and the Knudsen number go to zero. For the small Mach number regime, it is natural to consider the perturbations near the absolute Maxwellian. This fact was observed by Sone during the formal expansion for the steady solution(\cite{Sone1971}, \cite{Sone1987}). Mathematical results were given among the others. Masi, Esposito, and Lebowitz gives short time existence in $\R^d$(\cite{Masi1989}). Bardos and Ukai proved incompressible Navier-Stokes limit in a distributional sense using a semigroup theory(\cite{Bardos1991}). Golse and Saint-Raymond derived the convergence to the DiPerna-Lions renormalized solutions(\cite{Masmoudi2003}, \cite{Lions2001}). Recently Guo introduced the $L^2-L^\infty$ framework to prove incompressible Navier-Stokes limit for the periodic domain and the bounded domain(\cite{Guo2006}, \cite{Esposito2018-apde}).

On the other hand, in the context of exterior boundary value problem for the Boltzmann equation, Ukai and Asano studied this problem with fixed Knudsen number (\cite{Ukai1983}, \cite{Ukai1986}, \cite{Ukai2009}). Xia also demonstrated a similar result with relaxed assumptions using the spectral method(\cite{Xia2017}). For hydrodynamic limit, Esposito, Guo and Marra proved the existence of a steady solution in exterior domain(\cite{Esposito2018-cmp}).

Our main result is the justification of the global stability of the equation \eqref{thm1eq}. The main contribution of this work is extending Guo's $L^2-L^\infty$ framework(\cite{Esposito2013}) to a unbounded domain by a new splitting argument of the macroscopic part of the particle distribution.

This paper is organized as follows: Section \ref{sec2} discusses the time integrability of fluid parts, which purely relies on the time decay of the solution of the fluid equations. In Section \ref{sec3}, we discuss the energy estimate of the linearized problem. Section \ref{sec4} discusses the $L^\infty$ estimate of the equation. We propose $L^2-L^3-L^6$ splitting method in Section \ref{l2l6split}. In Section \ref{sec6}, we will control the nonlinear collision term. Then, we use iteration method to construct the solution of the Boltzmann equation in Section \ref{sec7}.

\section{Estimation of the Fluid Equations} \label{sec2}

In this section, we will discuss the time decay of remainder terms in the Boltzmann equation. 

\subsection{Estimation of the fluid terms in the remainder of the Boltzmann equation}
In this subsection, we will discuss how $h$ in \eqref{eqofh} can be represented as the solution of the fluid equations.

Recall the definition of $h$ from \eqref{eqofh}. $h$ can be written as follow by the order of $\e$ power, and we define $h_{-3/2}, h_{-1/2}, h_{1/2} $ and $h_{3/2}$ as
\begin{eqnarray}
\nonumber h&=& -\e^{-3/2} L(f_1)\\
\nonumber &&+\e^{-1/2}( - v  \cdot \nabla f_1 -L(f_2) + \G(f_1,f_1))\\
\nonumber &&+\e^{1/2}(-\pt_t f_1 - v\cdot \nabla f_2 + \G(f_1,f_2) + \G(f_2,f_1)) \\
\nonumber && + \e^{3/2}(-\pt_t f_2 +\G(f_2,f_2)) \\
\label{hdecomposedef} &=:& \e^{-3/2}h_{-3/2} +
\e^{-1/2}h_{-1/2} +
\e^{1/2}h_{1/2} +
\e^{3/2}h_{3/2}.
\end{eqnarray}
To prove the following proposition is the main goal of this subsection.
\begin{proposition} \label{hdecompose}
\begin{align*}
h_{-3/2} = h_{-1/2} = 0, h_{1/2},(h_{3/2} + v \cdot \pt_t u_2 \sqrt{\mu}) \perp \ker L, v \cdot \pt_t u_2 \sqrt{\mu} \perp \sqrt{\mu}
\end{align*}
\end{proposition}
Before we start to prove the proposition, let us prove the following lemma.
\begin{lemma} \label{gammasquarecal}
If $\sqrt{\mu}h \in \ker L$ Then,
$$
2\G (\sqrt{\mu}h,\sqrt{\mu}h) = L\left[\sqrt{\mu}h^2\right]
$$
\end{lemma}
\begin{proof}
By the definition of the $\G$ in \eqref{Gdefn}
\begin{eqnarray*}
\G(\sqrt{\mu}h,\sqrt{\mu}h)
&=& \iint_{\R^3 \times \S^2} B(\omega,{v-u}) \exp{\left(-\frac{\norm{u}^2}{2}\right)} \left[  h(u')h(v') -  h(u)h(v) \right] \dd \omega \dd u
\end{eqnarray*}
and
\begin{eqnarray*}
-L\left[\sqrt{\mu}h^2\right] &=& \G\left(\sqrt{\mu}h^2,\sqrt{\mu}\right) + \G\left(\sqrt{\mu},\sqrt{\mu}h^2\right)\\
&=& \iint_{\R^3 \times \S^2} B(\omega,{v-u}) \exp{\left(-\frac{\norm{u}^2}{2}\right)} \\ 
&&\left[ h(u')^2 + h(v')^2 - h(u)^2 - h(v)^2 \right] \dd \omega \dd u
\end{eqnarray*}
According to the definition of the $u', v'$ in \eqref{u'v'defn} and by simple calculation
\begin{eqnarray*}
&&\left[ h(u')^2 + h(v')^2 - h(u)^2 - h(v)^2 \right] + 2\left[  h(u')h(v') -  h(u)h(v) \right]\\
&&=\left(h(u')+h(v') \right)^2-\left(h(u)+h(v) \right)^2\\
&&=\left(\left(h(u')+h(v') \right)-\left(h(u)+h(v) \right)\right)\left(\left(h(u')+h(v') \right)+\left(h(u)+h(v) \right)\right)=0.
\end{eqnarray*}
Thus,
\[
2\G (\sqrt{\mu}h,\sqrt{\mu}h) = L\left[\sqrt{\mu}h^2\right]
\].
\end{proof}

Now prove the Proposition \ref{hdecompose}.
\begin{proof} (proof of the Proposition \ref{hdecompose})

Recall the equation \eqref{hdecomposedef}. By the definition of the $f_1$ in \eqref{deff1}
\begin{eqnarray*}
h_{-3/2}=-L(f_1) = 0.
\end{eqnarray*}
By the definition of the $f_2$ in \eqref{deff2}, $h_{-1/2}$ can be simplified as follows:
\begin{eqnarray*}
&&h_{-1/2}\\
&&=- v \cdot \nabla_{x} f_1 -L( f_2) + \G(f_1,f_1) \\
&&= -L\Bigg(\frac{1}{2} \sum_{i,j=1}^{3} \mathscr{A}_{ij} \left[\pt_{x_i}u_{j} + \pt_{x_j}u_{i} \right] + \sum_{i,j=1}^{3} \mathscr{B}_{i} \pt_{x_i} \th
 + L^{-1}\left[\G\left(f_1,f_2\right) \right]  \Bigg) \\
&&- L(v \cdot u_2 \sqrt{\mu}) - v \cdot \nabla_{x} f_1 + \G(f_1,f_1) \\
&& = -\frac{1}{2} \sum_{i,j=1}^{3} L\mathscr{A}_{ij} \left[\pt_{x_i}u_{j} + \pt_{x_j}u_{i} \right] - \sum_{i=1}^{3} L\mathscr{B}_{i} \pt_{x_i} \th  - v \cdot \nabla_{x} f_1 \\
&& = \frac{1}{2} \sum_{i,j=1}^{3}  \left(\sqrt{\mu} (v_i v_j -\frac{\norm{v}^2}{3} \delta_{ij} ) \right) \left[\pt_{x_i}u_{j} + \pt_{x_j}u_{i} \right]
+\sum_{i=1}^{3} \left(\sqrt{\mu} v_i \frac{\norm{v}^2-5 }{2} \right) \pt_{x_i} \th  \\
&& \hspace{5mm} - \sum_{i,j=1}^{3}  \sqrt{\mu} v_i v_j \pt_{x_i} u_j - \sum_{i=1}^{3} \sqrt{\mu} v_i \frac{\norm{v}^2-5 }{2}  \pt_{x_i} \th \\
&&=0.
\end{eqnarray*}
Next we need to calculate $h_{1/2}$,
\begin{eqnarray*}
&&h_{1/2}\\
&&= -\pt_t f_1 -v \cdot \nabla_{x} f_2 + \G(f_1,f_2)+\G(f_2,f_1) \\
&&=-\sqrt{\mu}\left(\pt_t \rho + \pt_t u \cdot v + \pt_t \th \frac{\norm{v}^2-3}{2} \right) \\
&&-\sqrt{\mu}\left( \frac{\norm{v}^2}{3} \nabla \cdot u_2  \right) - v\cdot \nabla_x \ip f_2\\
&&-\sqrt{\mu}\left( \sum_{i=1}^{3}\sum_{j=1}^{3}(v_i v_j -\frac{\norm{v}^2}{3} \delta_{ij} )\pt_i (u_2)_j \right) \\
&& + \G(f_1,f_2)+\G(f_2,f_1).
\end{eqnarray*}
\newline
\textbf{Step 1)} Calculation of $P(v\cdot \nabla_x \ip f_2)$.

Calculate $P(v\mathscr{A}_{ij})$ and $P(v\mathscr{B}_{i})$ from \eqref{ABdefine}.
Because of the oddness and perpendicularity to kernel implies
\begin{eqnarray*}
\inn{v_k \mathscr{A}_{ij},\sqrt{\mu}} = 0, \quad& \inn{v_k \mathscr{A}_{ij}, v_k \sqrt{\mu}} = 0,  &\quad \inn{v_k \mathscr{A}_{ij},\norm{v}^2\sqrt{\mu}} = 0,\\
\inn{v_k \mathscr{B}_{i},\sqrt{\mu}} = 0, \quad &\inn{v_k \mathscr{B}_{i},v_l\sqrt{\mu}} = 0.&
\end{eqnarray*}
Define
\begin{eqnarray*}
\inn{\mathscr{A}_{ij}, v_i v_j \sqrt{\mu}} = -\kappa, \quad \inn{\mathscr{B}_{i}, v_i \left(\frac{\norm{v}^2-3}{2} \right) \sqrt{\mu}} = -\eta.
\end{eqnarray*}
Then, 
\begin{eqnarray*}
\inn{v_k \mathscr{A}_{ij}, v_l \sqrt{\mu}} =- (\delta_{ik}\delta_{jl}+\delta_{il}\delta_{jk})\kappa, \quad \inn{v_k\mathscr{B}_{i}, \left(\frac{\norm{v}^2-3}{2} \right) \sqrt{\mu}} = -\delta_{ik}\eta,
\end{eqnarray*}
\begin{eqnarray*}
&&P\left(v \cdot \nabla_x \frac{1}{2} \sum_{i,j=1}^{3} \mathscr{A}_{ij} \left[\pt_{x_i}u_{j} + \pt_{x_j}u_{i} \right] \right)\\
&&= -\frac{1}{2} \kappa \sum_{i,j=1}^{3} \left[(\pt_i \pt_i u_j + \pt_j \pt_i u_i)v_j + (\pt_j \pt_i u_j + \pt_j \pt_j u_i)v_i\right]\\
&=&-\kappa \sum_{j=1}^{3} \left[(\Delta u_j + \pt_j \nabla \cdot u)v_j \right]\\
&=& -\kappa \sum_{j=1}^{3}\Delta u_j v_j \sqrt{\mu},
\end{eqnarray*}
\begin{eqnarray*}
P\left(v \cdot \nabla_x\sum_{i=1}^{3} \mathscr{B}_{i} \pt_{x_i} \th \right)&=& -\eta \Delta \th \frac{\norm{v}^2-3}{2} \sqrt{\mu}.
\end{eqnarray*}
\newline
\textbf{Step2)} Calculate of
$
P\left( v \cdot \nabla_x L^{-1}\G(f_1,f_1)\right).
$

Since
$$
\inn{ v \cdot \nabla_x \G(f_1,f_1), \psi} = \sum_{i=1}^3\inn{  \pt_i \G(f_1,f_1), v_i \psi},
$$
 we only need to focus on 
\begin{align} \label{Pf1f1}
\sum_{i=1}^3\inn{  \pt_i L^{-1} \G(f_1,f_1), v_i \frac{\norm{v}^2-5}{2} \sqrt{\mu}},
\sum_{i=1}^3\inn{  \pt_i L^{-1} \G(f_1,f_1), (v_i v_j -\frac{\norm{v}^2}{3} \delta_{ij} ) \sqrt{\mu}}.
\end{align}

By the lemma \ref{gammasquarecal}, \eqref{Pf1f1} become
$$ 
\sum_{i=1}^3\inn{  \pt_i \left(\rho + u \cdot v + \th \frac{\norm{v}^2-3}{2} \right)^2\sqrt{\mu}, v_i \frac{\norm{v}^2-5}{2} \sqrt{\mu}}
$$ and
$$
\sum_{i=1}^3\inn{  \pt_i \left(\rho + u \cdot v + \th \frac{\norm{v}^2-3}{2} \right)^2\sqrt{\mu}, (v_i v_j -\frac{\norm{v}^2}{3} \delta_{ij} ) \sqrt{\mu}}.
$$
By direct calculations
\begin{eqnarray*}
\sum_{i=1}^3\inn{  \pt_i \left(\rho + u \cdot v + \th \frac{\norm{v}^2-3}{2} \right)^2\sqrt{\mu}, v_i \frac{\norm{v}^2-5}{2} \sqrt{\mu}} &=& \frac{5}{2} \nabla \cdot (u \th)\\
&=&\frac{5}{2} (u \cdot \nabla)  \th,
\end{eqnarray*}
\begin{eqnarray*}
&&\sum_{i=1}^3\inn{  \pt_i \left(\rho + u \cdot v + \th \frac{\norm{v}^2-3}{2} \right)^2\sqrt{\mu}, (v_i v_j -\frac{\norm{v}^2}{3} \delta_{ij} ) \sqrt{\mu}} \\
&& =\sum_{i=1}^3 \pt_i (u_{i}u_j - \frac{1}{3}\norm{u}^2 \delta_{i,j})\\
&&= (u \cdot \nabla) u_j - \pt_{j} \frac{1}{3}\norm{u}^2.
\end{eqnarray*}
Then the coefficient of $(1,v,\norm{v}^2)\sqrt{\mu}$ for the $\P h_{1/2}$ can be calculated by the definition of the $f_1$ and $f_2$.

The coefficient of the $\sqrt{\mu}$ is
\begin{eqnarray*}
\rho_t + \nabla \cdot u_2 = 0.
\end{eqnarray*}
The coefficient of the $v\sqrt{\mu}$ is
\begin{eqnarray*}
u_t + (u\cdot \nabla)u +\nabla \norm{u}^2 - \kappa \Delta u  =0.
\end{eqnarray*}
The coefficient of the $\frac{\norm{v}^2-3}{2}\sqrt{\mu}$ is
\begin{eqnarray*}
\th_t + \nabla \cdot u_2 + (u\cdot \nabla)\th - \eta \Delta \th  =0.
\end{eqnarray*}
Thus, $h_{1/2} \perp \ker L$.
Lastly, $h_{3/2} = -\pt_t f_2 +\G(f_2,f_2)$ implies that $h_{3/2} + v \cdot \pt_t u_2 \sqrt{\mu} \perp \ker L$ and $v \cdot \pt_t u_2 \sqrt{\mu}\perp \sqrt{\mu}$.
\end{proof}

\subsection{Time decay of the fluid equations}

In this subsection, we will discuss time integrability of the $f_1, f_2$  and $h$. $f_1, f_2$  and $h$ can be written as the solution of the fluid equation. The main propose of this subsection is to prove the following proposition.

\begin{proposition} \label{prop23}
Recall $f_1, f_2$ and $h$ in \eqref{deff1}, \eqref{deff2} and \eqref{eqofh}. We have
\begin{align*}
&\normm{f_1+ \e f_2}_{L^{2}_{t} L^{\infty}_{x,v}},
\normm{\pt_t f_1+ \e \pt_t f_2}_{L^{2}_{t}L^{\infty}_{x,v}}, \\
&\normm{f_1+ \e f_2}_{L^{\infty}_{t}L^{\infty}_{x,v}},
\normm{\pt_t f_1+ \e \pt_t f_2}_{L^{\infty}_{t} L^{\infty}_{x,v}},\\
&\normm{h}_{L^2_{t}L^{2}_{x,v}}, \normm{\pt_t h}_{L^2_{t}L^{2}_{x,v}}, \normm{\pt_t u_2}_{L^{2}_{t}L^{2}_{x,v}}, \normm{\pt_t^2 u_2}_{L^{\infty}_{t}L^{2}_{x,v}}, \normm{\pt_t^3 u_2}_{L^{2}_{t}L^{2}_{x,v}}, \\
&\normm{h}_{L^\infty_{t}L^{2}_{x,v}}, \normm{\pt_t h}_{L^\infty_{t}L^{2}_{x,v}},  \normm{\pt_t u_2}_{L^{2}_{t}L^{\infty}_{x,v}}, \normm{\pt_t^2 u_2}_{L^{\infty}_{t}L^{\infty}_{x,v}}, \normm{\pt_t^3 u_2}_{L^{2}_{t}L^{\infty}_{x,v}},
\end{align*}
are bounded by $C(u_0, \th_0)$.
\end{proposition}

To understand $f_1, f_2$  and $h$, we only need to understand the solutions of the fluid equations : $u, \th$ and $\rho$.
Consider the following linear equation
\begin{eqnarray}  \label{stokeseq}
\pt_t u + \nabla p = \Delta u, \quad  \nabla \cdot u = 0  &\text{ in }& \Omega^{c}, \\
\pt_{t} \th = \Delta \th &\text{ in }& \Omega^{c} \label{heateq}, \\
u(x,0) = u_0 (x), \quad \th(x,0) = \th_{0} (x) &\text{ in }& \Omega^{c}, \nonumber\\
u(x) = 0, \quad \th(x) = 0 &\text{ on }& \pt \Omega.  \nonumber 
\end{eqnarray}
The following lemmas explain time decay of linear fluid problems.
\begin{lemma} \label{stokestimedecay}
(Time decay of stokes operator)

Let $A$ be a Stokes operator, and $e^{-At}$ forms Stokes semi-group. The solution of \eqref{stokeseq} exist, denoted as $u(t) = e^{-At} u_0$. In addition, the following decay estimate holds for $1 < p \le q \le \infty$
\begin{eqnarray*}
\normm{u(t)}_{L^q} &\lesssim& \e_{p}(t) (1+t)^{-\frac{3}{2}(\frac{1}{p}-\frac{1}{q})}, \\
t^{1/2} \normm{\nabla u(t)}_{L^q} &\lesssim& \e_{p}(t) (1+t)^{-\frac{3}{2}(\frac{1}{p}-\frac{1}{q})}, \\
t \normm{A u(t)}_{L^q} &\lesssim& \e_{p}(t) (1+t)^{-\frac{3}{2}(\frac{1}{p}-\frac{1}{q})}, \\
\end{eqnarray*}
with $\e_p(t) \lesssim \normm{u_0}_{p}$ and (for $p>1$) $\lim_{t \to \infty} \e_{p}(t) = 0$.
\end{lemma}

\begin{proof}
This is a result of the Theorem 1 of \cite{Wigner2000}.
\end{proof}

\begin{lemma}\label{heattimedecay}
(Time decay of heat operator)

$e^{-\Delta t}$ forms a semi-group. The solution of \eqref{heateq} exist, denoted as $\th(t) = e^{-\Delta t} \th_0$. In addition, the following decay estimate holds for $1 < p \le q \le \infty$
\begin{eqnarray*}
\normm{\th(t)}_{L^q} &\lesssim& \e_{p}(t) (1+t)^{-\frac{3}{2}(\frac{1}{p}-\frac{1}{q})}, \\
t^{1/2} \normm{\nabla \th(t)}_{L^q} &\lesssim& \e_{p}(t) (1+t)^{-\frac{3}{2}(\frac{1}{p}-\frac{1}{q})}, \\
t \normm{\Delta \th(t)}_{L^q} &\lesssim& \e_{p}(t) (1+t)^{-\frac{3}{2}(\frac{1}{p}-\frac{1}{q})}, \\
\end{eqnarray*}
with $\e_p(t) \lesssim \normm{\th_0}_{p}$ and (for $p>1$) $\lim_{t \to \infty} \e_{p}(t) = 0$.
\end{lemma}

\begin{proof}
This is a result of the Theorem 1.1 of \cite{Ishige2009}.
\end{proof}

Now we are ready to prove the time decay of nonlinear problem. Recall the fluid problem \eqref{nseq} and \eqref{foureq}.

\begin{lemma} \label{fluiddecay}
(Time decay of Navier-Stokes-Fourier system) 

Define $\e_{p} = \max\{\normm{u_0}_{p}, \normm{\pt_t u_0}_{p}, \normm{\th_0}_{p}, \normm{\pt_t \th_0}_{p} \}.$
If $\e_{p} \ll 1$, then the solution of \eqref{nseq} and \eqref{foureq} exist and the following decaying estimate holds, if $1 < p \le 3$, and $1< p < q \le \infty$ or $1< p = q < \infty$
\begin{eqnarray*}
\normm{u(t)}_{L^q} &\lesssim& \e_{p} (1+t)^{-\frac{3}{2}(\frac{1}{p}-\frac{1}{q})}, \\
t^{1/2} \normm{\nabla u(t)}_{L^q} &\lesssim& \e_{p} (1+t)^{-\frac{3}{2}(\frac{1}{p}-\frac{1}{q})}, \\
\normm{\th(t)}_{L^q} &\lesssim& \e_{p} (1+t)^{-\frac{3}{2}(\frac{1}{p}-\frac{1}{q})}, \\
t^{1/2} \normm{\nabla \th(t)}_{L^q} &\lesssim& \e_{p} (1+t)^{-\frac{3}{2}(\frac{1}{p}-\frac{1}{q}) }, \\
\normm{\pt_t u(t)}_{L^q} &\lesssim& \e_{p} (1+t)^{-\frac{3}{2}(\frac{1}{p}-\frac{1}{q})}, \\
t^{1/2} \normm{\pt_t  \nabla u(t)}_{L^q} &\lesssim& \e_{p} (1+t)^{-\frac{3}{2}(\frac{1}{p}-\frac{1}{q})}, \\
\normm{\pt_t \th(t)}_{L^q} &\lesssim& \e_{p} (1+t)^{-\frac{3}{2}(\frac{1}{p}-\frac{1}{q})}, \\
t^{1/2} \normm{\pt_t \nabla \th(t)}_{L^q} &\lesssim& \e_{p} (1+t)^{-\frac{3}{2}(\frac{1}{p}-\frac{1}{q}) }.
\end{eqnarray*}
\end{lemma}

\begin{proof}
According to the Duhamel's principle, the solution of the \eqref{nseq} and \eqref{foureq} can be written as
\begin{eqnarray*}
u(t) =  e^{-tA} u_0 - \int_{0}^{t} e^{-(t-s)A} P(u \cdot \nabla) u (s) \dd s 
\end{eqnarray*}
and
\begin{eqnarray*}
\th(t) =  e^{t \Delta} \th_0 - \int_{0}^{t} e^{(t-s)\Delta} (u \cdot \nabla) \th (s) \dd s .
\end{eqnarray*}
$A = - P \Delta = -\Delta P$, $P$ is a projection from $L^p$ to divergence free space.

To simplify the notation, define
\begin{eqnarray*}
G_1(u,v)(t) &=& - \int_{0}^{t} e^{-(t-s)A} P(u \cdot \nabla) v (s) \dd s, \\
G_2(u,v)(t) &=& - \int_{0}^{t} e^{(t-s)\Delta} (u \cdot \nabla) v (s) \dd s, \\
F_1(u,v)(t) &=& P(u \cdot \nabla) v (t), \\
F_2(u,v)(t) &=& (u \cdot \nabla) v (t).
\end{eqnarray*}
According to the estimate of the heat and Stokes operator (Lemma \ref{stokestimedecay}, Lemma \ref{heattimedecay}), 
\begin{eqnarray*}
\normm{e^{-tA} u_0}_{q} &\le& C t^{-(3/p - 3/q)/2}\normm{u_0}_{p}, \\
\normm{e^{t \Delta} \th_0}_{q} &\le& C t^{-(3/p - 3/q)/2}\normm{\th_0}_{p}, \\
\normm{\nabla e^{-tA} u_0}_{q} &\le& C t^{-(1+3/p - 3/q)/2}\normm{u_0}_{p}, \\
\normm{\nabla  e^{t \Delta} \th_0}_{q} &\le& C t^{-(1+ 3/p - 3/q)/2}\normm{\th_0}_{p}.
\end{eqnarray*}
By the Holder inequality,
\begin{eqnarray*}
\normm{F_{i}(u,v)}_{p} \le C \normm{u}_{r} \normm{\nabla v}_{s}, \quad 1/p = 1/r +1/s, \quad i=1,2.
\end{eqnarray*}
Combining the above two inequalities we can get
\begin{eqnarray}
\label{gammaG}
\normm{G(u,v)(t)}_{3/\g} &\le&  C \int_{0}^{t} (t-s)^{-(\a +\b - \g)/2} \normm{u(s)}_{3/\a} \normm{\nabla v(s)}_{3/\b} \dd s, \\
\label{gammanablaG}
\normm{\nabla G(u,v)(t)}_{3/\g} &\le&  C \int_{0}^{t} (t-s)^{-(1+ \a +\b - \g)/2} \normm{u(s)}_{3/\a} \normm{\nabla v(s)}_{3/\b} \dd s.
\end{eqnarray}

First, set $\g = \a = \delta$, $\b =1$ for \eqref{gammaG} and set $\a = \delta$, $\b= \g =1$ for \eqref{gammanablaG}, Define the norm as $\normm{t^{(1-\delta)/2} u}_{3/\delta} = K(u)$, $\normm{t^{1/2} \nabla u}_{3} = K'(u)$. Then, 
\begin{eqnarray*}
K(u) &\le& \normm{u_0}_{3} + c K(u) K'(u), \\ 
K'(u) &\le& \normm{\nabla u_0}_{3} + c K(u) K'(u), \\
K(\pt_t u) &\le& \normm{\pt_t u_0}_{3} + c K(\pt_t u) K'(u) +  K(u) K'(\pt_t  u),\\ 
K'(\pt_t u) &\le& \normm{\nabla \pt_t u_0}_{3} + c K(\pt_t u) K'(u) +  K(u) K'(\pt_t  u), \\
K(\th) &\le& \normm{\th_0}_{3} + c K(u) K'(\th), \\ 
K'(\th) &\le& \normm{\nabla u_0}_{3} + c K(u) K'(\th), \\
K(\pt_t \th) &\le& \normm{\pt_t \th_0}_{3} + c K(\pt_t u) K'(\th) +  K(u) K'(\pt_t  \th),\\ 
K'(\pt_t \th) &\le& \normm{\nabla \pt_t \th_0}_{3} + c K(\pt_t u) K'(\th) +  K(u) K'(\pt_t  \th).
\end{eqnarray*}
Since $\int_{0}^{t} (t-s)^{-1/2}s^{-1+\delta/2} \dd s < \infty$, this leads to the closure of the estimate.
So, we can get decay of $u$(or $\th$) as
\begin{eqnarray*}
\normm{u}_{3/\delta} &\lesssim& t^{-(1-\d)/2} \normm{u_0}_{3}, \\
\normm{\nabla u}_{3} &\lesssim& t^{-1/2} \normm{u_0}_{3}. \\
\end{eqnarray*}
Next, we take $\g = 3/q$, $\a = \d$ and $\b =1$ for \eqref{gammaG}. Then we can obtain,
\begin{eqnarray*}
\normm{u(t)}_{q} &\le& \normm{e^{-t A} u_0} + c K(u) K'(u) \int_{0}^{t} (t-s)^{-(1+\d -3/q)/2} s^{-(1-\d/2)} \dd s \\
&\le& C \normm{u_0}_{3} t^{-(1-3/q)/2}, \\
\normm{\nabla u(t)}_{q} &\le& \normm{\nabla e^{-t A} u_0} + c K(u) K'(u) \int_{0}^{t} (t-s)^{-(2+\d -3/q)/2} s^{-(1-\d/2)} \dd s \\
&\le& C \normm{u_0}_{3} t^{-(2-3/q)/2}.
\end{eqnarray*}
Thus, we can prove the lemma \ref{fluiddecay} for $p=3 \le q$. (For the estimate of $\nabla u$ we need $\d < 3/q$ condition.)

Next we think about the case when $1 < p=q < \infty$. Let us take $\a = \g = 3/p$, $0<\b<1$ for \eqref{gammaG} and we know that $\nabla u$(or $\nabla \th, \pt_t u, \pt_t \th$) then we can get 
\begin{eqnarray*}
\normm{G(u,v)(t)}_{p} &\le&  C \int_{0}^{t} (t-s)^{-\b/2} \normm{u(s)}_{p} \normm{\nabla v(s)}_{3/\b} \dd s \\
& \le & C \int_{0}^{t} (t-s)^{-\b/2} s^{-(1-\b/2)} \normm{u(s)}_{p} \normm{s^{1-\b/2}\nabla v(s)}_{n/\b} \dd s,\\
\normm{\nabla G(u,v)(t)}_{p} &\le&  C \int_{0}^{t} (t-s)^{-(1+\b)/2}  s^{-(1-\b/2)} \normm{u(s)}_{p} \normm{ s^{1-\b/2}\nabla v(s)}_{3/\b} \dd s.
\end{eqnarray*}
By the simple calculation
\begin{eqnarray*}
\int_{0}^{t} (t-s)^{-\b/2} s^{-(1-\b/2)} \dd s &\sim& 1, \\
\int_{0}^{t} (t-s)^{-(1+\b)/2} s^{-(1-\b/2)} \dd s &\sim& t^{-1/2}.
\end{eqnarray*}
Thus,
\begin{eqnarray*}
\normm{u(t)}_{p} &\le& \normm{u_0}_{p} + C \normm{u_0}_{p} \sup_{s\le t}\normm{u(s)}_{p}, \\
\normm{\nabla u(t)}_{p} &\le& t^{-1/2} (\normm{u_0}_{p} + C \normm{u_0}_{p} \sup_{s\le t}\normm{u(s)}_{p}).
\end{eqnarray*}
We can get same inequality for $\pt_t u, \th$ and $\pt_t \th$ too.
So, we proved the theorem for $p=q$.
Next we move on to general $p$ and $q$.

\begin{lemma} \label{lem27}
For each $p \le q \le \infty$, there is some $0<\g<1$ such that
\begin{eqnarray*}
\normm{u(t)}_{q}, \normm{\pt_t u(t)}_{q}, \normm{\th(t)}_{q}, \normm{\pt_t \th(t)}_{q}  &\le& c \normm{u(\g t)}_{p} t^{-\frac{3}{2}(\frac{1}{p}-\frac{1}{q})}
\end{eqnarray*}
\end{lemma}
\begin{proof}
If $p \ge 3/2$ let $k_j = p (3/2)^j$ and $k \ge k_0$. By the Duhamel's principle,
\begin{eqnarray*}
u(t) &=&  e^{-t/2A} u(t/2) - \int_{t/2}^{t} e^{-(t-s)A} P(u \cdot \nabla) u (s) \dd s \\
\th(t) &=&  e^{t/2 \Delta} \th(t/2) - \int_{t/2}^{t} e^{(t-s)\Delta} (u \cdot \nabla) \th (s) \dd s 
\end{eqnarray*}
Thus,
\begin{eqnarray*}
\normm{u(t)}_{3k/2} &\le&  c \normm{u(t/2)}_{k} t^{-1/2 k} \\
&&+ c \int_{t/2}^{t} (t-s)^{-3/2((3+k)/3k-2/3k)} \normm{u \nabla u (s)}_{3k/(3+k)} \dd s \\
&\le& c \normm{u(t/2)}_{k} t^{-1/2 k} + c \int_{t/2}^{t} (t-s)^{-(k+1)/2k} t^{-1/2} \normm{u (s)}_{k} \dd s \\
&\le & c \sup_{t/2\le s \le t}\normm{u(s)}_{k} t^{-1/2 k}
\end{eqnarray*}
We obtain the following inequality by iteratively using the previous argument
\begin{eqnarray*}
\normm{u(t)}_{k_j} \le c_j \sup_{2^{-j} t \le s \le t}\normm{u(s)}_{p} t^{-3/2 (1/p-1/k_j)}.
\end{eqnarray*}
As $k_j>q$ for some $j$, interpolation with $p=q$ case then we have,
\begin{eqnarray*}
\normm{u(t)}_{q}  &\le& c \normm{u(\g t)}_{p} t^{-\frac{3}{2}(\frac{1}{p}-\frac{1}{q})}
\end{eqnarray*}
for $3/2 \le p \le q < \infty $. Since we know that the Lemma \ref{fluiddecay} holds for $p=3 \le q \le \infty$, we also can conclude $u$ estimate of Lemma \ref{fluiddecay} for $3/2 \le p \le q \le \infty $.
We can get same results for $\normm{\pt_t u(t)}_{q}, \normm{\th(t)}_{q}, \normm{\pt_t \th(t)}_{q}$.
Now assume $p< 3/2$, define 
\begin{eqnarray*}
L := \sup_{r \le t/2} \normm{u(t/2+r)}_{3/2} r^{3/2(1/p- 2/3)}.
\end{eqnarray*}
Then,
\begin{eqnarray*}
L &\le&  c \normm{u(t/2)}_{p} r^{-3/2(1/p- 2/3)} \\
&& + c \int_{t/2}^{t/2+r} (t/2+r-s)^{-1/2} \normm{\nabla u (s)}_{3} \normm{ u (s)}_{3/2} \dd s \\
&\le& c \normm{u(t/2)}_{p} r^{-3/2(1/p- 2/3)} \\
&& + c t^{-1/2} \normm{u_0}_{3}\int_{0}^{r} (r-x)^{-1/2} L x^{-3/2(1/p- 2/3)} \dd s.
\end{eqnarray*}
Thus, $L \le c \normm{u(t/2)}_{p} + c_0 \normm{u_0}_{3} L$, with $c_0$ is not depend on $p$. If $c_0 \normm{u_0}_{3} <1$, we can have, 
\begin{eqnarray*}
\normm{u(t)}_{3/2} \le c t^{-3/2(1/p- 2/3)} \normm{u(t/2)}_{p}.
\end{eqnarray*}
We can apply the Lemma \ref{lem27} with $q=3/2$ then we can complete the Lemma \ref{lem27}.
\end{proof}

So, we have a result for $1<p < q \le \infty$ and $1< p=q< \infty$.
The following lemmas prove the similar result for $\normm{\nabla u(s)}_{q}$.
\begin{lemma}
$1<q\le 3$ then,
\begin{eqnarray}
t^{1/2} \normm{\nabla u(t)}_{q} &\le c \normm{u(t/2)}_{q}, \\
t^{1/2} \normm{\nabla \th(t)}_{q} &\le c \normm{\th(t/2)}_{q}, \\
t^{1/2} \normm{\nabla \pt_t u(t)}_{q} &\le c \normm{\pt_tu(t/2)}_{q}, \\
t^{1/2} \normm{\nabla \pt_t \th(t)}_{q} &\le c \normm{\pt_t \th(t/2)}_{q}.
\end{eqnarray}
\end{lemma}
\begin{proof}
\begin{eqnarray*}
u(t) &=&  e^{-t/2A} u(t/2) - \int_{t/2}^{t} e^{-(t-s)A} P(u \cdot \nabla) u (s) \dd s \\
\th(t) &=&  e^{t/2 \Delta} \th(t/2) - \int_{t/2}^{t} e^{(t-s)\Delta} (u \cdot \nabla) \th (s) \dd s 
\end{eqnarray*}
Thus,
\begin{eqnarray*}
&& \normm{\nabla u(t/2+r)}_{q} \\
&& \le  c r^{-1/2} \normm{u(t/2)}_{q} + \int_{t/2}^{t/2+r} (t/2+r-s)^{-1/2}  \normm{P(u \cdot \nabla) u (s)}_{q} \dd s \\
&& \le  c r^{-1/2} \normm{u(t/2)}_{q} + \int_{t/2}^{t/2+r} (t/2+r-s)^{-1/2}  \normm{\nabla u(s)}_{q} \normm{u(s)}_{\infty} \dd s \\
\\
&& \le  c r^{-1/2} \normm{u(t/2)}_{q} + \int_{t/2}^{t/2+r} (t/2+r-s)^{-1/2} s^{-1/2}\normm{u_0}_{3} \normm{\nabla u(s)}_{q}  \dd s
\end{eqnarray*}
Let us define
\begin{eqnarray*}
Q : = \sup_{0 \le r \le t/2} r^{1/2} \normm{\nabla u(t/2+r)}_{q}
\end{eqnarray*}
Then,
\begin{eqnarray*}
Q &\le& c \normm{u(t/2)}_{q} + c\normm{u_0}_{3} \sup_{0 \le r \le t/2} \left\{ r^{1/2} \int_{0}^{r} (r-s)^{-1/2} s^{-1/2} (s+t/2)^{-1/2} Q  \dd s \right\} \\
&\le& c \normm{u(t/2)}_{q} + c\normm{u_0}_{3} Q \\
\end{eqnarray*}
So, we complete the proof for small initial data.
$\th , \pt_t u, \pt_t \th$ are the exactly same.
\end{proof}
Thus, the Lemma \ref{fluiddecay} holds.
\end{proof}

Now let us prove the main proposition of this subsection.
\begin{proof} (proof of the Proposition \ref{prop23})

Recall $f_1, f_2$ and $h$ in \eqref{deff1}, \eqref{deff2} and \eqref{eqofh} and Proposition \ref{hdecompose}
\begin{eqnarray*}
f_1 &=& \left(\rho(x) + u(x) \cdot v + \th (x)\frac{\norm{v}^2-3}{2} \right)\sqrt{\mu}, \\
f_2 &=& \frac{1}{2} \sum_{i,j=1}^{3} \mathscr{A}_{ij} \left[\pt_{x_i}u_{j} + \pt_{x_j}u_{i} \right] + \sum_{i,j=1}^{3} \mathscr{B}_{i} \pt_{x_i} \th\\
\nonumber && + L^{-1}\left[\G\left(f_1,f_1\right) \right] + v \cdot u_{2} \sqrt{\mu},
\end{eqnarray*}
\begin{eqnarray}
h = \e^{1/2}h_1 + \e^{3/2} h_2,
\end{eqnarray}
where
\begin{eqnarray*}
h_1 &=& \ip (v \cdot \nabla f_2) + \G(f_1, f_2) +\G(f_2, f_1) \\
h_2 &=& \G(f_2,f_2).
\end{eqnarray*}
By Lemma \ref{fluiddecay}, the following inequality holds for small $\d$
\begin{eqnarray*}
\normm{\rho (t)}_{L^{\infty}_{x}} &\lesssim& (1+t)^{-\frac{3}{2}+\d}, \\
\normm{u (t)}_{L^{\infty}_{x}} &\lesssim& (1+t)^{-\frac{3}{2}+\d}, \\
\normm{\pt_t \rho (t)}_{L^{\infty}_{x}} &\lesssim& (1+t)^{-\frac{3}{2}+\d}, \\
\normm{\pt_t u (t)}_{L^{\infty}_{x}} &\lesssim& (1+t)^{-\frac{3}{2}+\d},
\end{eqnarray*}
where $\sim$ is depend on the initial data $u_0, \th_0$. Thus, 
$\normm{f_1}_{L^2_{t} L^{\infty}_x}, \normm{\pt_t f_1}_{L^2_{t} L^{\infty}_x}, \normm{f_1}_{L^\infty_{t} L^{\infty}_x}$ and
$ \normm{\pt_t f_1}_{L^\infty_{t} L^{\infty}_x}$ are bounded by $C(u_0, \th_0)$.

To estimate the decay of $f_2$, we have
\begin{eqnarray*}
\normm{\nabla \rho (t)}_{L^{\infty}_{x}} &\lesssim& (1+t)^{-2+\d} \\
\normm{\nabla u (t)}_{L^{\infty}_{x}} &\lesssim& (1+t)^{-2+\d} \\
\normm{\pt_t \nabla \rho (t)}_{L^{\infty}_{x}} &\lesssim& (1+t)^{-2+\d} \\
\normm{\pt_t  \nabla u (t)}_{L^{\infty}_{x}} &\lesssim& (1+t)^{-2+\d} .
\end{eqnarray*}

In addition, $u_2 = - \nabla \Delta^{-1} \rho_t =  - \nabla \Delta^{-1}(\nabla \cdot (u \th) - \eta \Delta \th )$.
Thus, 
$\normm{f_2}_{L^2_{t} L^{\infty}_x}, \normm{\pt_t f_2}_{L^2_{t} L^{\infty}_x}, \normm{f_2}_{L^\infty_{t} L^{\infty}_x} $, $\normm{\pt_t f_2}_{L^\infty_{t} L^{\infty}_x}$, $\normm{\pt_t u_2}_{L^2_{t} L^{2}_{x,v}}$, $\normm{\pt_t^2 u_2}_{L^\infty_{t} L^{2}_{x,v}}$, $\normm{\pt_t^3 u_2}_{L^2_{t} L^{2}_{x,v}}$, $\normm{\pt_t u_2}_{L^2_{t} L^{\infty}_{x,v}}$, $\normm{\pt_t^2 u_2}_{L^\infty_{t} L^{\infty}_{x,v}}$, and $\normm{\pt_t^3 u_2}_{L^2_{t} L^{\infty}_{x,v}}$ are bounded by $C(u_0, \th_0)$.

Now, move on $h_1$ and $h_2$
We know that $\normm{f_1(t)}_{L^{\infty}_{x}} \lesssim (1+t)^{-\frac{3}{2}+\d}$. By 
Lemma \ref{fluiddecay},
\begin{eqnarray*}
\normm{\nabla \rho (t)}_{L^{2}_{x}} &\lesssim& (1+t)^{-\frac{5}{4}+\d}, \\
\normm{\nabla u (t)}_{L^{2}_{x}} &\lesssim& (1+t)^{-\frac{5}{4}+\d}, \\
\normm{\pt_t \nabla \rho (t)}_{L^{2}_{x}} &\lesssim& (1+t)^{-\frac{5}{4}+\d}, \\
\normm{\pt_t  \nabla u (t)}_{L^{2}_{x}} &\lesssim& (1+t)^{-\frac{5}{4}+\d},
\end{eqnarray*}
which means $\normm{h}_{L^2_{t} L^{2}_{x}}, \normm{\pt_t h}_{L^2_{t} L^{2}_{x}}, \normm{h}_{L^\infty_{t} L^{2}_{x}}$ and $\normm{\pt_t h}_{L^\infty_{t} L^{2}_{x}}$ are bounded by $C(u_0, \th_0)$.
\end{proof}

\section{Energy estimate} \label{sec3}
In this section, we will prove the existence and the energy estimate of the following linear problem:

\begin{align}
\label{Boltzwwithg}
\e \pt_t f + v \cdot \nabla_x f + \e^{-1} L f = g & \quad \text{ in } \R_{+} \times \Omega^c \times \R^3,
\\ \nonumber
f|_{t=0} = f_0, \quad f_{\g_{-}} = P_{\gamma} f +r. &
\end{align}

\begin{proposition}
\label{energyestimate} Suppose $g \in L^2(\R_+ \times \Omega^{c} \times \R^3)$, and $r \in L^2(\R_+ \times \g_{-})$ such that, for all $t>0$,
\begin{align*}
\iint_{\Omega^c \times \R^3} g(t,x,v) \sqrt{\mu} \dd v \dd x = \int_{\g_{-}} r(t,x,v) \sqrt{\mu} \dd \g = 0.
\end{align*}
Then, for any sufficiently small $\e$, there exists a unique solution to the problem \eqref{Boltzwwithg} such that
\begin{align*}
\iint_{\Omega^c \times \R^3} f(t,x,v) \sqrt{\mu} \dd v \dd x = 0,
\end{align*}
for all $t \ge 0$.
Moreover, the following energy estimate holds:
\begin{align*}
& \pt_t \normm{f}^2_{L^2_x L^2_v} + \e^{-1}\int_{\gamma_{+}} (1-P_{\gamma}) f^2 d\gamma + \e^{-2} \kappa \normm{\ip f}^2_{{\nu}} \\
& \lesssim \e^{-1}\int_{\gamma_{+}} r^2 d\gamma + \e^{-1}\int_{\Omega^c \times \R^3} fg \dd x \dd v.
\end{align*}
\end{proposition} 
\begin{proof}
First, we prove the existence of the solution to \eqref{Boltzwwithg}. Define the approximating sequence with $f^0 := f_0$:
\begin{align} \label{approximatel}
\pt_t f^{l+1} + \e^{-1} v \cdot \nabla f^{l+1} + \e^{-2} \nu f^{l+1} - \e^{-2} K f^{l} = \e^{-1}g, \\ \nonumber
f^{l+1}|_{t=0} = f_0, \quad f^{l+1}_{\g_{-}} = \left(1-\frac{\e}{j} \right)P_{\gamma} f^l +r. &
\end{align}
\newline
\textbf{STEP 1)} Fix $j$, $f^l \to f_j$ as $l \to \infty$. Notice that,
\begin{align*}
\norm{\int_{\R^3 } K f^l f^{l+1} \dd v} \le \iint_{\R^3 \times \R^3} \norm{k(v,u)}^{1/2} \norm{f^l(u)} \norm{k(v,u)}^{1/2} \norm{f^{l+1}(v)} \dd u \dd v \\
\le \sqrt{\int_{u} \norm{f^l(u)}^2 \int_{v} \norm{k(v,u)} } \sqrt{\int_{v} \norm{f^{l+1}(v)}^2 \int_{u} \norm{k(v,u)} } \lesssim \normm{f^l}_2^2 + \normm{f^{l+1}}_2^2,
\end{align*}
where we used $\sup_u \int_{v} \norm{k(v,u)} + \sup_v \int_{u} \norm{k(v,u)} < \infty$.

Note that, since $\int_{\g_{-}} r \sqrt{\mu} \dd \g = 0$, we have
\begin{align} \label{flboundary}
\norm{\left(1-\frac{\e}{j} \right)P_{\gamma} f^l + r}^2_{2,-} = \norm{\left(1-\frac{\e}{j} \right)P_{\gamma} f^l}^2_{2,-} + \norm{r}^2_{2}.
\end{align}
Multiply $f^{l+1}$ by \eqref{approximatel} and integrate by space and velocity, we can get:
\begin{align*}
\pt_t \frac{1}{2} \int_{x,v} {f^{l+1}}^2  \dd x \dd v + \e^{-1} \frac{1}{2} \int_{\gamma_{+}} {f^{l+1}}^2 \dd \gamma + \e^{-2}  \int_{x,v} \nu {f^{l+1}}^2 \dd x \dd v \\
= \e^{-1} \frac{1}{2} \int_{\gamma_{-}} {f^{l+1}}^2 \dd \gamma + \e^{-2} \int_{x,v} f^{l+1}  K f^{l}\dd x \dd v + \e^{-1} \int_{x,v} g  f^{l+1}\dd x \dd v.
\end{align*}
Apply \eqref{flboundary}, take time integration we can get:
\begin{align*}
&\frac{1}{2} \normm{f^{l+1}(t)}_2^2 + \e^{-2} \int_{0}^t \normm{f^{l+1}}_{\nu}^2 + \e^{-1} \frac{1}{2} \int_{0}^t \norm{f^{l+1}}_{2,+}^2\\
&\le \e^{-1} \frac{1}{2} \left[\left(1-\frac{\e}{j} \right)^2 \right] \int_{0}^{t} \norm{P_{\gamma} f^l}_{2,-}^2 + \e^{-1} \frac{1}{2} \int_{0}^t \norm{r}^2_2 + C \e^{-1}\int_{0}^t \max_{1 \le i \le l+1} \normm{f^i}^2_2 \\
& + \int_{0}^t \normm{g}^2_2 + \frac{1}{2}\normm{f_0}^2_2 \\
&\le \e^{-1} \frac{1}{2} \left[\left(1-\frac{\e}{j} \right)^2 \right] \int_{0}^{t} \norm{f^l}_{2,+}^2 + \e^{-1} \frac{1}{2} \int_{0}^t \norm{r}^2_2 + C \e^{-1}\int_{0}^t \max_{1 \le i \le l+1} \normm{f^i}^2_2 \\
& + \int_{0}^t \normm{g}^2_2 + \frac{1}{2}\normm{f_0}^2_2.
\end{align*}

Set $\eta = \left(1-\frac{\e}{j} \right)^2 <1$. Now use this inequality to bound for $\e^{-1} \int_{0}^t \norm{f^l}_{2,-}^2$ and iterate:
\begin{align*}
&\frac{1}{2} \normm{f^{l+1}(t)}_2^2 + \e^{-2} \int_{0}^t \normm{f^{l+1}}_{\nu}^2 + \e^{-1} \frac{1}{2} \int_{0}^t \norm{f^{l+1}}_{2,+}^2\\
& \le \eta \bigg\{ \e^{-1} \eta \frac{1}{2} \int_{0}^{t} \norm{f^{l-1}}_{2,+}^2 + \e^{-1} \frac{1}{2} \int_{0}^t \norm{r}^2_2 + C \e^{-1}\int_{0}^t \max_{1 \le i \le l+1} \normm{f^i}^2_2\\
& + \int_{0}^t \normm{g}^2_2 + \frac{1}{2}\normm{f_0}^2_2  \bigg\} \\
&+ \e^{-1} \frac{1}{2} \int_{0}^t \norm{r}^2_2 + C \e^{-1}\int_{0}^t \max_{1 \le i \le l+1} \normm{f^i}^2_2 + \int_{0}^t \normm{g}^2_2 + \frac{1}{2}\normm{f_0}^2_2 \\
&= \eta^2 \e^{-1} \frac{1}{2}\int_{0}^{t} \norm{f^{l-1}}_{2,+}^2 \\
&+ (1+\eta) \bigg\{ \e^{-1} \frac{1}{2} \int_{0}^t \norm{r}^2_2 + C \e^{-1}\int_{0}^t \max_{1 \le i \le l+1} \normm{f^i}^2_2 + \int_{0}^t \normm{g}^2_2 + \frac{1}{2}\normm{f_0}^2_2 \bigg\}\\
& \vdots \\
&\le \eta^{l+1} \e^{-1} \frac{1}{2} \int_{0}^{t} \norm{f^{0}}_{2,+}^2 \\
&+ \frac{(1-\eta)^{l+1}}{(1-\eta)} \bigg\{ \e^{-1} \frac{1}{2} \int_{0}^t \norm{r}^2_2 + C \e^{-1}\int_{0}^t \max_{1 \le i \le l+1} \normm{f^i}^2_2 + \int_{0}^t \normm{g}^2_2 + \frac{1}{2}\normm{f_0}^2_2 \bigg\}.
\end{align*}
Since $f^0 = f_0$,
\begin{align*}
\max_{1 \le i \le l+1} \normm{f^i}^2_2 \lesssim_{\eta, j}& \e^{-1} \int_{0}^t \norm{r}^2_2 + \e^{-1}\int_{0}^t \max_{1 \le i \le l+1} \normm{f^i}^2_2 + \int_{0}^t \normm{g}^2_2 + \normm{f_0}^2_2 \\
& + t \norm{f_0}_{2,+}^2.
\end{align*}
By Gronwall's lemma, we have, for fixed $t >0$,
\begin{align*}
\max_{1 \le i \le l+1} \normm{f^i}^2_2 \lesssim_{\eta, j, \e, t}& \left\{ \e^{-1} \int_{0}^t \norm{r}^2_2 +  \int_{0}^t \normm{g}^2_2 + \normm{f_0}^2_2 + \norm{f_0}_{2,+}^2 \right\}.
\end{align*}
This leads to
\begin{align*}
\max_{1 \le i \le l+1} \left\{\normm{f^i}^2_2 + \e^{-2} \int_{0}^t \normm{f^{i}}_{\nu}^2 + \e^{-1} \int_{0}^t \norm{f^{i}}_{2,+}^2 \right\} \\
\lesssim_{\eta, j, \e, t} \left\{ \e^{-1} \int_{0}^t \norm{r}^2_2 +  \int_{0}^t \normm{g}^2_2 + \normm{f_0}^2_2 + \norm{f_0}_{2,+}^2 \right\}.
\end{align*}
Upon taking the difference, we have
\begin{align*}
\pt_t [f^{l+1} -f^l] + \e^{-1} v \cdot \nabla [f^{l+1} -f^l] + \e^{-2} \nu [f^{l+1} -f^l] = \e^{-2} K [f^{l} -f^{l-1}],
\end{align*}
with $[f^{l+1} -f^l](0) = 0$ and $[f^{l+1} -f^l]|_{\g_-} = \left(1-\frac{\e}{j} \right)
P_{\gamma}[f^{l+1} -f^l]$.

Applying previous iteration to $[f^{l+1} -f^l]$ yields 
\begin{align*}
&\frac{1}{2} \normm{f^{l+1}(t) -f^l(t)}^2_2 + \e^{-2} \int_{0}^t \normm{f^{l+1}(s) -f^l(s)}_{\nu}^2 \dd s \\
&+ \e^{-1} \frac{1}{2} \int_0^{t} \norm{f^{l+1}(s) -f^l(s)}^2_{2,+} \dd s \\
&\le \eta \e^{-1} \frac{1}{2} \int_{0}^t \norm{f^{l}(s) -f^{l-1}(s)}_{2,+}^2 \dd s + C \int_{0}^t \normm{f^{l}(s) -f^{l-1}(s)}^2_2 \dd s \\
&\le \eta \left\{\e^{-1} \frac{1}{2} \int_{0}^t \norm{f^{l}(s) -f^{l-1}(s)}_{2,+}^2 \dd s + \sup_{0 \le s \le T} \normm{f^{l}(s) -f^{l-1}(s)}^2_2 \right\},
\end{align*}
for $TC < \eta < 1$. This implies that $f^l$ is a Cauchy sequence with respect to the norm
\begin{align*}
\e^{-1} \int_{0}^t \norm{f^{l}(s)}_{2,+}^2 \dd s + \sup_{0 \le s \le T} \normm{f^{l}(s) }^2_2, 
\end{align*}
in $[0,T]$. Repeating the argument for $[0,T], [T,2T]....$ we deduce that for time $t$, there exist a unique limit function $f^l \to f_j$ such that
\begin{align} \label{eqforfj}
\pt_t f_j + \e^{-1} v \cdot \nabla f_j + \e^{-2} L f_j = \e^{-1}g
\end{align}
\newline
\textbf{STEP 2)} Let $j \to \infty$. Multiply $f_j$ by \eqref{eqforfj} and integrate by space, velocity and time and use integration by parts, we deduce
\begin{align*}
&\frac{1}{2} \normm{f_j(t)}^2_2 + \e^{-2}  \int_{0}^t \normm{\ip f_j(s)}^2_{\nu} \dd  s + \e^{-1} \frac{1}{2} \int_{0}^t \norm{(1-P_{\gamma}) f_j(s)}^2_{2,+} \dd  s\\
& \le \e^{-1} \frac{1}{2} \int_0^t \norm{r(s)}_{2,+}^2 \dd s + \frac{\e}{2j} \e^{-1} \int_0^t \norm{P_{\gamma} f_j(s)}^2_{2,+} \dd s + \int_{0}^t \normm{\nu^{-1/2}\ip g(s)}_2^2 \dd s \\
& + \e^{-2} \int_0^t \normm{\P g(s)}^2_2 \dd s + \frac{1}{2} \normm{f_0}^2_2.
\end{align*}
By trace theorem(lemma 2.3 of \cite{Esposito2018-apde}),
\begin{align*}
&\int_0^t \norm{P_{\gamma} f_j(s)}^2_{2,+} \dd s  \\
&\lesssim \int_0^t \norm{\1_{\g_{+}^\d} f_j(s)}_{2,+}^2 \dd s + \e^{-1} \int_0^t \norm{(1-P_{\gamma}) f_j(s)}_{2,+}^2 \dd s \\
&\lesssim \e \normm{f_j(0)}^2_2 + \e \int_{0}^t \normm{f_j(s)}^2_2 \dd s  + \e \int_0^t \iint_{\Omega^c \times \R^3} \norm{[\pt_t + \e^{-1} v \cdot \nabla ]f_j f_j} \dd x \dd v \dd s \\
&+ \e^{-1} \int_0^t \norm{(1-P_{\gamma}) f_j(s)}_{2,+}^2 \dd s\\
& \lesssim \e \normm{f_j(0)}^2_2 + \e \int_{0}^t \normm{f_j(s)}^2_2 \dd s + \int_0^t \e^{-2} \norm{\ip f_j(s)}_{\nu}^2 \dd s \\
& + \int_0^t \iint_{\Omega^c \times \R^3} \norm{g(s) f_j(s)} \dd x \dd v \dd s  + \e^{-1} \int_0^t \norm{(1-P_{\gamma}) f_j(s)}_{2,+}^2 \dd s.
\end{align*}
From the boundary condition, $\int_{0}^t \norm{f_j}^2_{2,-} \le \int_{0}^t \norm{f_j}^2_{2,+} + \int_{0}^t \norm{r}^2_{2,+}$. From
\begin{align*}
&\int_0^t \iint_{\Omega^c \times \R^3} \norm{g(s) f_j(s)} \dd x \dd v \dd s \\
&\lesssim \int_0^t \normm{g(s)}^2_2 \dd s\\
& + o(1) \left[ \int_0^t \normm{\P f_j(s)}^2_2 \dd s + \int_0^t \normm{\ip f_j(s)}^2_2 \dd s \right],
\end{align*}
we get
\begin{align} \label{eq357}
&\normm{f_j(t)}_2^2 + \e^{-2} \int_0^t \normm{\ip f_j (s)}^2_{\nu} \dd s + \int_0^t \norm{f_j(s)}_{2,+}^2 \dd s \\ \nonumber
& \lesssim \e^{-1} \int_0^t \norm{r}^2_2 \dd s + \int_0^t \normm{ g(s)}_2^2 \dd s+ \normm{f_0}^2_2 + o(1) \int_0^t \normm{\P f_j(s)}^2_2 \dd s.
\end{align}
Since, $\normm{\P f_j(s)}^2_2 \le \normm{f_j(s)}^2_2$, integrating \eqref{eq357} from $0$ to $t$, we have
\begin{align*}
\int_0^t \normm{\P f_j(s)}^2_2 \dd s \lesssim_{t} \e^{-1} \int_0^t \norm{r}^2_2 \dd s + \int_0^t \normm{ g(s)}_2^2 \dd s+ \normm{f_0}^2_2.
\end{align*}
Thus, we conclude that, for $j \gg 1$ and $0 < \e \ll 1$,
\begin{align} \label{eq358}
&\normm{f_j(t)}_2^2 + \e^{-2} \int_0^t \normm{\ip f_j (s)}^2_{\nu} \dd s + \int_0^t \norm{f_j(s)}_{2,+}^2 \dd s \\ \nonumber
& \lesssim_{\e, t} \e^{-1} \int_0^t \norm{r}^2_2 \dd s + \int_0^t \normm{ g(s)}_2^2 \dd s+ \normm{f_0}^2_2.
\end{align}
By taking a weak limit, we obtain a weak solution $f$ to \eqref{Boltzwwithg} with the same bound \eqref{eq358}. Taking difference, we have
\begin{align*}
&\pt_t [f_j - f] + \e^{-1} v \cdot \nabla [f_j - f] + \e^{-2} L[f_j - f] = 0, \\
& [f_j - f]|_{\g_{-}} = P_{\gamma} [f_j - f] + \frac{\e}{j} P_{\gamma} f_j, \quad [f_j - f](0) = 0.
\end{align*}
Apply \eqref{eq358} with $r = \frac{\e}{j} P_{\gamma} f_j$ we obtain
\begin{align*}
&\normm{f_j(t) - f(t)}_2^2 + \e^{-1} \int^t_0 \normm{f_j(s) - f(s)}_{\nu}^2 \dd s + \e^{-1} \int^t_0 \norm{f_j(s) - f(s)}_{2,+}^2 \dd s \\
&\lesssim_{t} \frac{1}{j} \int^t_0 \normm{P_{\gamma}f_j}_{2,+}^2 \dd s \to 0.
\end{align*}
Thus, we construct $L^2$ solution to \eqref{Boltzwwithg}.
The energy estimate can be easily obtained by multiplying $f$ by \eqref{Boltzwwithg} and integrating by parts.
\end{proof}

\section{$L^{\infty}$ 
Estimate} \label{sec4}
The main goal of this section is to prove the following Proposition.

\begin{proposition}\label{linftyestimate}
($L^\infty$ estimate)

Let $f$ satisfy,
$$
[\e \pt_t + v \cdot \nabla_x + \e^{-1}C_0 \inn{v}^{\th} ]\norm{f} \le \e^{-1} K_{\b}\norm{f} + \norm{g},
$$
$$
\norm{f}_{\gamma_{-}} \le P_{\gamma} \norm{f} +\norm{r}, \norm{f|_{t=0}} \le \norm{f_0},
$$
for $0 < \b <\frac{1}{4}$, $K_{\b}\norm{f} = \int_{\R^3} k_{\b}(v,u) \norm{f(u)} \dd u$ and
\begin{align*}
k_{\b}(v,u) := \left\{ \norm{u-v} + \norm{u-v}^{-1}\right\} \exp \left[-\b \norm{u-v}^2- \b \frac{[\norm{u}^2 -\norm{v^2}]^2}{\norm{u-v}^2} \right].
\end{align*}
Then, for $\omega (v) = e^{\b' \norm{v}^2}$ with $0 < \b' \ll \b$,
\begin{eqnarray*}
&&\normm{\e^{1/2} \omega f(t)}_{\infty}
\\
&&\lesssim \normm{\e^{1/2} \omega f_0}_{\infty} + \sup_{0\le s\le t } \normm{\e^{1/2} \omega r(s)}_{\infty} +\e^{3/2} \sup_{0\le s\le t } \normm{ \inn{v}^{-\th} \omega g(s)}_{\infty} \\
&&+\sup_{0\le s\le t } \normm{ \P_{1} f(s)}_{L^6(\Omega^{c}\times \R^3)} + \e^{-1}\sup_{0\le s\le t } \normm{ \P_{2} f(s)}_{L^{2}(\Omega^c\times \R^3)} \\
&& + \e^{-1/2}\sup_{0\le s\le t } \normm{ {\P}_{3} f(s)}_{L^{3}(\Omega^c\times \R^3)} + \e^{-1}\sup_{0\le s\le t } \normm{ \ip f(s)}_{L^{2}(\Omega^c \times \R^3)},\\
&&\normm{\e^{1/2} \omega f(t)}_{\infty} \\
&&\lesssim \normm{\e^{1/2} \omega f_0}_{\infty} + \sup_{0\le s\le t } \normm{\e^{1/2} \omega r(s)}_{\infty} +\e^{3/2} \sup_{0\le s\le t } \normm{ \inn{v}^{-\th} \omega g(s)}_{\infty} \\
&&+ \e^{-1}\sup_{0\le s\le t } \normm{ f(s)}_{L^{2}(\Omega^c \times \R^3)},
\end{eqnarray*}
holds, where $\P_{1}$, $\P_{2}$ and $\P_{3}$ are defined in theorem \ref{l2l6estimate}.
\end{proposition}

\begin{definition}
Consider the following linear problem
\begin{eqnarray} \label{transporteq}
\e \pt_t f + v \cdot \nabla f =g.
\end{eqnarray}
The characteristic of the equation \eqref{transporteq} are
\begin{eqnarray}
\dot{X} = \e^{-1} V, \quad \dot{V} = 0, \quad
X(t;t,x,v) = x, \quad
V(t;t,x,v) = v
\end{eqnarray}
Define backward and forward time, position and velocity as follows
\begin{eqnarray}
t_{b} (x,v) &=& \sup\{t>0 : X(-s,0,x,v) \in \Omega^c \text{ for all } 0< s<t \}, \\
t_{f} (x,v) &=& \sup\{t>0 : X(s,0,x,v) \in \Omega^c \text{ for all } 0< s<t \}, \\
x_{b}(x,v) &=& X(-t_b(x,v);0,x,v), \\ 
v_{b}(x,v) &=& V(-t_b(x,v);0,x,v), \\ 
x_{f}(x,v) &=& X(t_f(x,v);0,x,v), \\ 
v_{f}(x,v) &=& V(t_f(x,v);0,x,v)  .
\end{eqnarray}

\end{definition}
We define the stochastic cycles. 
\begin{definition}
Define, for free variable $v_k \in \R^3$
\begin{eqnarray*}
t_0 &=& t, \\
x_0 &=& x, \\
t_1 &=& t- t_b(x,v), \\
x_1 &=& x_b(x,v), \\
t_2 &=& t_1- t_b(x_1,v_1), \\
x_2 &=& x_b(x_1,v_1), \\
&\vdots& \\
t_{k+1} &=& t_k- t_b(x_k,v_k), \\
x_{k+1} &=& x_b(x_k,v_k).
\end{eqnarray*}
Set
\begin{eqnarray*}
X_{cl}(s;t,x,v) &:=& \sum_{k} \mathbf{1}_{[t_{k+1}, t+k)}(s) X(s;t_{k},x_k,v_k), \\
V_{cl}(s;t,x,v) &:=& \sum_{k} \mathbf{1}_{[t_{k+1}, t+k)}(s) V(s;t_{k},x_k,v_k).
\end{eqnarray*}
For $x \in \pt \Omega$, define
\begin{eqnarray*}
\mathscr{V} := \{ v\in \R^3 : n(x) \cdot v >0 \}, \quad \dd \sigma = \sqrt{2 \pi} \mu (v) n(x) \cdot v \dd v.
\end{eqnarray*}
For $j \in \mathbb{N}$, denote
\begin{eqnarray*}
\mathscr{V}_{j} := \{ v_{j} \in \R^3 : n(x_{j}) \cdot v_{j} >0 \}, \quad \dd \sigma_j = \sqrt{2 \pi} \mu (v_{j}) n(x_{j}) \cdot v_{j} \dd v_{j}.
\end{eqnarray*}
\end{definition}

\begin{lemma} \label{finitecollision}
For sufficient large $T_0 >0$, there exist constant $C_1, C_2 >0$, independent of $T_0$, such that for $k = C_1 T_0^{5/4}$,
\begin{eqnarray*}
\sup_{(t,x,v) \in [0,T_{0}] \times \Omega^c \times \R^3} \int_{\Pi_{l=1}^{k-1} \mathscr{V}_l} \mathbf{1}_{t_k(t,x,v_1,v_2, \cdot, v_{k-1})>0}  \Pi_{l=1}^{k-1} \dd \sigma_{l} < \left\{ \frac{1}{2} \right\}^{C_2 T_0^{5/4}}
\end{eqnarray*}
\end{lemma}
\begin{proof}
This lemma is lemma 3.12 from \cite{Esposito2018-cmp}
\end{proof}

Now we are ready to prove the main result of this section.

\begin{proof}
(proof of Proposition \ref{linftyestimate})

For the weight $w(v) = e^{\b' \norm{v}^2}$ define $h$ as
\begin{eqnarray*}
h(t,x,v) := w(v) f(t,x,v).
\end{eqnarray*}
According to lemma 3 of \cite{Guo2010-arma}, there exist $\tilde{\b} = \tilde{\b} ({\b},\b') >0 $ such that $k_{\b}(v,u)\frac{w(v)}{w(u)} \lesssim k_{\tilde{\b}} (v,u)$.  Then,
\begin{eqnarray*}
[\pt_t + \e^{-1} v \cdot \nabla_x + \e^{-2}C_0\inn{v}^{\th} ]\norm{\e^{1/2} h} \le \e^{-2} \int_{\R^3} k_{\tilde{\b}}(v,u)\norm{\e^{1/2}h} \dd u + \e^{-1/2}\norm{wg}.
\end{eqnarray*}
\newline
\textbf{Step 1)} For $t \in [n \e T_0, (n+1) \e T_0]$ with $n \in \mathbb{N} \cup \{0\}$ and $T_0$ from Lemma \ref{finitecollision} we will prove
\begin{eqnarray} \nonumber
&&\norm{\e^{1/2} h(t,x,v)} \\ \nonumber
&\le& C T_0^{5/2} e^{-\frac{C_0 (t-n \e T_0)}{\e^2}} \normm{\e^{1/2} h(n\e T_0,\cdot,\cdot)}_{\infty} + C_{T_0} \e^{1/2} \sup_{0\le s\le t } \normm{ \omega r(s)}_{\infty} \\
\nonumber&&
+C_{T_0} \e^{3/2} \sup_{0\le s\le t } \normm{ \inn{v}^{-\th} \omega g(s)}_{\infty} \\
\label{inftystep1}&&+C T_0^{5/2} \sup_{n \e T_0\le s\le (n+1) \e T_0 } \normm{ \P_{1} f(s)}_{L^6(\Omega^{c}\times \R^3)}\\
\nonumber&& + C T_0^{5/2} \e^{-1}\sup_{n \e T_0\le s\le (n+1) \e T_0 } \normm{ \P_{2} f(s)}_{L^{2}(\Omega^c \times \R^3)} \\
\nonumber &&+C T_0^{5/2} \e^{-1/2}\sup_{n \e T_0\le s\le (n+1) \e T_0 } \normm{ {\P}_{3} f(s)}_{L^3(\Omega^{c}\times \R^3)}\\
\nonumber&& + C T_0^{5/2} \e^{-1}\sup_{n \e T_0\le s\le (n+1) \e T_0 } \normm{ {\ip} f(s)}_{L^{2}(\Omega^c \times \R^3)} \\
\nonumber&& + [C T_0^{5/4} \frac{1}{2}^{C_2 T_0^{5/4}} + o(1) C T_0^{5/2}] \sup_{n \e T_0\le s\le (n+1) \e T_0 } \normm{ \e^{1/2} h}_{\infty},
\end{eqnarray}
and
\begin{eqnarray} \nonumber
&&\norm{\e^{1/2} h(t,x,v)} \\ \nonumber
&\le& C T_0^{5/2} e^{-\frac{C_0 (t-n \e T_0)}{\e^2}} \normm{\e^{1/2} h(n\e T_0, \cdot, \cdot)}_{\infty} + C_{T_0} \e^{1/2} \sup_{0\le s\le t } \normm{ \omega r(s)}_{\infty} \\
\nonumber&&
+C_{T_0} \e^{3/2} \sup_{0\le s\le t } \normm{ \inn{v}^{-\th} \omega g(s)}_{\infty} \\
\label{inftystep1-1}&&+ C T_0^{5/2} \e^{-1}\sup_{n \e T_0\le s\le (n+1) \e T_0 } \normm{ f(s)}_{L^{2}(\Omega^c \times \R^3)} \\
\nonumber&& + [C T_0^{5/4} \frac{1}{2}^{C_2 T_0^{5/4}} + o(1) C T_0^{5/2}] \sup_{n \e T_0\le s\le (n+1) \e T_0 } \normm{ \e^{1/2} h}_{\infty}.
\end{eqnarray}
To simplify the notation define $(X_{cl}(\cdot), V_{cl}(\cdot)) = (X_{cl}(\cdot;t,x,v), V_{cl}(\cdot;t,x,v))$. Since $(X_{cl},V_{cl})$ is a characteristic of the equation.
\begin{eqnarray*}
&&\frac{d}{d s} \left[ e^{-\int_{s}^{t} \frac{C_0}{\e^2} \inn{V_{cl}(\tau)}^{\th} \dd \tau} \e^{1/2} h^{l+1} (s,X_{cl}(\tau),V_{cl}(\tau)) \right] \\
&& \le e^{-\int_{s}^{t} \frac{C_0}{\e^2} \inn{V_{cl}(\tau)}^{\th} \dd \tau} \frac{1}{\e^2}
\int_{\R^3} k_{\tilde{\b}}\left( V_{cl}(\tau), v'\right) \norm{\e^{1/2}h\left( s, X_{cl}(\tau), v' \right)} \dd v' \\
&& + e^{-\int_{s}^{t} \frac{C_0}{\e^2} \inn{V_{cl}(\tau)}^{\th} \dd \tau} \e^{-1/2} \norm{w g(s,X_{cl}(\tau),V_{cl}(\tau))}.
\end{eqnarray*}
holds for $t_1 \le s \le t$.
Along the stochastic cycles, for $k = C_1 T_0^{5/4}$, we deduce the following estimate:
\begin{eqnarray}
\nonumber &&\norm{\e^{1/2}h(t,x,v)} \\
\label{h0term}&& \le \mathbf{1}_{t_1 < 0} e^{-\int_{0}^{t}\frac{C_0 \inn{V_{cl}(\tau)}^{\th}}{\e^2} \dd \tau} \norm{\e^{1/2} h(0,X_{cl}(0),V_{cl}(0))}\\
\label{hwithk}&&+\int_{\max \{0,t_1\}}^{t} \frac{e^{-\int_{s}^{t}\frac{C_0 \inn{V_{cl}(\tau)}^{\th}}{\e^2} \dd \tau} }{\e^2}
\int_{\R^3} k_{\tilde{\b}} (V_{cl}(s)), v')
\norm{\e^{1/2} h(s,X_{cl}(s),v')} \dd v' \dd s \\
\label{gterm}&& + \int_{\max \{0,t_1\}}^{t} \frac{e^{-\int_{s}^{t}\frac{C_0 \inn{V_{cl}(\tau)}^{\th}}{\e^2} \dd \tau} }{\e^2}
\e^{3/2}\norm{w g(s,X_{cl}(s),V_{cl}(s))} \dd s \\
\label{rterm}&& + \mathbf{1}_{t_1 \ge 0} {e^{-\int_{t_1}^{t}\frac{C_0 \inn{V_{cl}(\tau)}^{\th}}{\e^2} \dd \tau} } \e^{1/2} \norm{wr(t_1, x_1, v_1)} \\
\label{Hterm}&& +\mathbf{1}_{t_1 \ge 0} \frac{e^{-\int_{t_1}^{t}\frac{C_0 \inn{V_{cl}(\tau)}^{\th}}{\e^2} \dd \tau} }{\tilde{w}(v_1)}
\int_{\Pi_{j=1}^{k-1}\mathscr{V}_j} H ,
\end{eqnarray}
where $\tilde{w}(v) := \frac{1}{w(v) \sqrt{\mu(v)}}$ and $H$ is given by,
\begin{eqnarray}
\label{H0term} &&\sum_{l=1}^{k-1} \mathbf{1}_{t_{l+1} \le 0 < t_l} \norm{\e^{1/2}h(0,X_{cl}(0),V_{cl}(0))} \dd \Sigma_{l}(0) \\
\label{Hwithk} &&+\sum_{l=1}^{k-1} \int_{\max \{0,t_{l+1}\}}^{t_l} \mathbf{1}_{ t_l>0}
\frac{1}{\e^2}\int_{\R^3} k_{\tilde{\b}} (V_{cl}(\tau)), v')
\norm{\e^{1/2} h(\tau,X_{cl}(\tau),v')} \dd v' \dd \Sigma_{l}(\tau) \dd \tau \\
\label{Hgterm}&& + \sum_{l=1}^{k-1} \int_{\max \{0,t_{l+1}\}}^{t_l} \mathbf{1}_{ t_l>0}
\e^{-1/2}\norm{w g(\tau,X_{cl}(\tau),V_{cl}(\tau))} \dd \Sigma_{l}(\tau) \dd \tau \\
\label{Hrterm}&& + \sum_{l=1}^{k-1} \int_{\max \{0,t_{l+1}\}}^{t_l} \mathbf{1}_{ t_l>0}
\e^{1/2}\norm{w r(t_l,x_{l+1},v_l)} \dd \Sigma_{l}(t_{l+1})  \\
\label{Hremain}
&&+\mathbf{1}_{ t_k>0}
\e^{1/2}\norm{h(t_k,x_k,v_{k-1})} \dd \Sigma_{k-1}(t_k) \dd \tau,
\end{eqnarray}
$\dd \Sigma_{l}(s)$ is defined by
\begin{eqnarray*}
&& \dd \Sigma_{l}(s) \\
&& = \left\{ \Pi_{j=l+1}^{k-1} \dd \sigma_j  \right\} 
\left\{ e^{-\int_{s}^{t_{l}}\frac{C_0 \inn{v_l}^{\th}}{\e^2} \dd \tau} \tilde{w}(v_{l}) \dd \sigma_{l} \right\}\Pi_{j=1}^{l-1} \left\{ e^{-\int_{i_{j+1}}^{t_{j}}\frac{C_0 \inn{v_l}^{\th}}{\e^2} \dd \tau} \tilde{w}(v_{j}) \dd \sigma_{j} \right\}.
\end{eqnarray*}
From our choice of $k = C_1 T_0^{5/4}$
\begin{eqnarray*}
\eqref{h0term} + \eqref{H0term} &\lesssim& C_1 T_0^{5/4}e^{-\frac{C_0}{\e^2}t}\normm{\e^{1/2}h_0}_{\infty},\\
\eqref{rterm} + \eqref{Hrterm} &\lesssim& C_1 T_0^{5/4} \sup_{0\le s\le t} \normm{\e^{1/2}wr(s)}_{\infty},
\end{eqnarray*}
and
\begin{eqnarray*}
&&\eqref{gterm} +\eqref{Hgterm} \\
&&\lesssim \e^{3/2} \sup_{0\le s\le t} \normm{\inn{v}^{-\th}wg(s)}_{\infty} \biggl \{\int_{0}^{t}\frac{\inn{V_{cl}(s)}^{\th}}{\e^2} e^{-\int_{s}^{t}\frac{C_0\inn{V_{cl}(\tau)}^{\th}}{\e^2} \dd \tau} \dd s  \\
&&+ C_1 T_0^{5/4} \sup_{l} \int_{0}^{t_l}\frac{\inn{V_{cl}(s)}^{\th}}{\e^2} e^{-\int_{s}^{t_l}\frac{C_0\inn{V_{cl}(\tau)}^{\th}}{\e^2} \dd \tau} \dd s   \biggr \} \\
&& \lesssim C_1 T_0^{5/4} \e^{3/2} \sup_{0\le s\le t} \normm{\inn{v}^{-\th}wg(s)}_{\infty} \int_{0}^{t}\frac{d}{d s}e^{-\int_{s}^{t}\frac{C_0\inn{V_{cl}(\tau)}^{\th}}{\e^2} \dd \tau} \dd s \\&& \lesssim C_1 T_0^{5/4} \e^{3/2} \sup_{0\le s\le t} \normm{\inn{v}^{-\th}wg(s)}_{\infty},
\end{eqnarray*}
where we used the fact that $\sigma_j$ is a probability measure of $\mathscr{V}_j$. For \eqref{Hremain}, by Lemma \ref{finitecollision}
\begin{eqnarray*}
&&\eqref{Hremain} \\
&& \lesssim \sup_{0 \le s \le t} \normm{\e^{1/2} h(s)}_{\infty} \sup_{(t,x,v)\in [0, \e T_0]\times \Omega^c \times \R^3}\int_{\Pi_{j=1}^{k-1} \mathscr{V}_j} \mathbf{1}_{t_k(t,x,v_1,v_2, \cdot, v_{k-1})>0}  \Pi_{l=1}^{k-1} \dd \sigma_{l}\\
&& \lesssim \left\{ \frac{1}{2} \right\}^{C_2 T_{0}^{5/4}} \sup_{0 \le s \le t} \normm{\e^{1/2} h(s)}_{\infty}.
\end{eqnarray*}
Now focus on \eqref{hwithk} and \eqref{Hwithk}. For $N>1$, we choose $m=m(N) \ll 1$ such that
\begin{eqnarray*}
&& k_{m}(v,u) :=  \1_{\norm{v-u} \ge \frac{1}{m}} \1_{\norm{u} \le m} \1_{\norm{v} \le m} k_{\tilde{\b}}(u,v),\\
&& \sup_{v} \int_{\R^3} \norm{k_m(v,u) -k_{\tilde{\b}}}(v,u) \dd u \le \frac{1}{N}.
\end{eqnarray*}
Split $k_{\tilde{\b}}$ into $k_{\tilde{\b}} = [k_{\tilde{\b}}-k_{m}] + k_m$ and put it into \eqref{hwithk} and \eqref{Hwithk}. For $N \ll_{T_0} 1$, first difference would lead to a small contribution bounded by,
\begin{eqnarray*}
\frac{k}{N} \normm{\e^{1/2} h}_{\infty} = \frac{C_1 T^{5/4}}{N} \normm{\e^{1/2} h}_{\infty}.
\end{eqnarray*}
For the remaining part split time integration into $[ \max\{0,t_{l+1}\},t_l ] = [\max\{0,t_{l+1}\},$ $
t_l-\k \e^2] \cup (t_l - \k \e^2 ,t_l]$ and $[\max\{0,t_{1}\},t] = [\max\{0,t_{1}\},t-\k \e^2] \cup (t - \k \e^2 ,t]$. The small time integration term ($(t_l - \k \e^2 ,t_l]$ and $(t - \k \e^2 ,t]$ ) can be bounded by
\begin{eqnarray*}
\k \e^2\sum_{l=1}^{k-1} \frac{1}{\e^2} \sup_{v} \int_{\norm{v'} \le m} k_m(v,v') \sup_{0 \le s \le t} \normm{\e^{1/2} h(s)}_{\infty}
\lesssim k \k \sup_{0 \le s \le t} \normm{\e^{1/2} h(s)}_{\infty},
\end{eqnarray*}
which will be a small contribution if we set $k \k \ll 1$.
Overall, for $(t,x,v) \in [0, \e T_0] \times \Omega^c \times \R^3$,
\begin{eqnarray*}
&&\norm{\e^{1/2}h(t,x,v)}\\
&&
\lesssim \int_{\max \{0,t_1  \}}^{t-k\e^2} \frac{e^{-\frac{C_0(t-s)}{\e^2}}  }{\e^2}
\int_{\norm{v'} \le m} k_{m} (V_{cl}(s)), v')
\underbrace{\norm{\e^{1/2} h(s,X_{cl}(s),v')}} \dd v' \dd s\\
&&+ \frac{e^{-\frac{C_0(t-t_1)}{\e^2}}  }{\tilde{w}(v)}\int_{\Pi_{j=1}^{k-1}\mathscr{V}_j}\sum_{l=1}^{k-1} \int_{\max \{0,t_{l+1}\}}^{t_l -k\e^2 } \mathbf{1}_{ t_l>0}\\
&& \times
\frac{1}{\e^2}\int_{\norm{v'} \le m} k_{m} (V_{cl}(\tau)), v')
\underbrace{\norm{\e^{1/2} h(\tau,X_{cl}(\tau),v')}} \dd v' \dd \Sigma_{l}(\tau) \dd \tau\\
&& + k\biggl\{ e^{-\frac{C_0}{\e^2}t}\e^{1/2}\normm{h_0}_{\infty} + \e^{1/2} \sup_{0 \le s \le t}\normm{wr(s)}_{\infty} + \e^{3/2}\sup_{0 \le s \le t} \normm{\inn{v}wg(s)}_{\infty}\biggr\}\\
&&+o(1)C_1T_0^{5/4} \sup_{0\le s \le t} \normm{\e^{1/2} h(s)}_{\infty} + \frac{1}{2}^{C_2 T_0} \sup_{0\le s \le t} \normm{\e^{1/2} h(s)}_{\infty}.
\end{eqnarray*}

Note that similar estimate holds for the underbraced terms. We plug these estimate into the underbraced terms to conclude
\begin{eqnarray*}
\norm{\e^{1/2} h(t,x,v) } \lesssim I_1 + I_2 + I_3.
\end{eqnarray*}
Here, using $w(v) \lesssim_m 1$ and $k_{m}(u,v) \lesssim_m 1$.
Estimate of $I_1$ and $I_2$ can be written as
\begin{eqnarray*}
I_1 &\lesssim_{m}&  \int_{\max \{0,t_1  \}}^{t-k\e^2} \frac{e^{-\frac{C_0(t-s)}{\e^2}}  }{\e^2}
\int_{\norm{v'} \le m} \int_{\max \{0,t_1  \}}^{s-k\e^2} \frac{e^{-\frac{C_0(s-s')}{\e^2}}  }{\e^2}
\int_{\norm{v''} \le m} \\
&&
\times {\norm{\e^{1/2} h(s',X_{cl}(s';s,X_{cl}(s;t,x,v),v'),v'')}} \dd v'' \dd s' \dd v' \dd s \\
&&
+ \int_{\max \{0,t_1  \}}^{t-k\e^2} \frac{e^{-\frac{C_0(t-s)}{\e^2}}  }{\e^2}
\int_{\norm{v'} \le m} \1_{t'_1 \ge 0} \frac{e^{-\frac{C_0(s-t'_1)}{\e^2}}}{\tilde{w}(v)}
\\
&&
\times \int_{\Pi_{j=1}^{k-1}\mathscr{V}_j}\sum_{l=1}^{k-1} \int_{\max \{0,t'_{l+1}\}}^{t'_l -k\e^2 } \mathbf{1}_{ t'_l>0} \frac{1}{\e^2} \\
&&
\times \int_{\norm{v''} \le m} \norm{\e^{1/2} h(s', X_{cl}(s'; t'_l, x'_l, v'_l), v'')} \dd v'' \dd s' \dd \Sigma_{l}(s') \dd v' \dd s,
\end{eqnarray*}
\begin{eqnarray*}
&& I_2 \\
&& \lesssim_{m} \1_{t_1 \ge 0} \frac{e^{-\frac{C_0(t-t_1)}{\e^2}}}{\tilde{w}(v)} \int_{\Pi_{j=1}^{k-1}\mathscr{V}_j}\sum_{l=1}^{k-1} \int_{\max \{0,t_{l+1}\}}^{t_l -k\e^2 } \1_{t_l \ge 0} \frac{1}{\e^2} \int_{\norm{v'} \le m} \int_{\max \{0,t'_{1}\}}^{s -k\e^2 } \frac{e^{-\frac{C_0(s-s')}{\e^2}}}{\e^2} \\
&& \times \int_{\norm{v''} \le m} \norm{\e^{1/2} h(s,X_{cl}(s';s,X_{cl}(s;t_l,x_l,v_l),v'),v'')} \dd v'' \dd  s' \dd v' \dd s \dd \Sigma_{l}(s)  \\
&&+ \1_{t_1 \ge 0} \frac{e^{-\frac{C_0(t-t_1)}{\e^2}}}{\tilde{w}(v)} \int_{\Pi_{j=1}^{k-1}\mathscr{V}_j}\sum_{l=1}^{k-1} \int_{\max \{0,t_{l+1}\}}^{t_l -k\e^2 } \1_{t_l \ge 0} \frac{1}{\e^2} \int_{\norm{v'} \le m} \\
&& \times \1_{t'_1 \ge 0} \frac{e^{-\frac{C_0(s-t'_1)}{\e^2}}}{\tilde{w}(v')}
\int_{\Pi_{j=1}^{k-1}\mathscr{V}'_j}\sum_{l'=1}^{k-1} \int_{\max \{0,t'_{l+1}\}}^{t'_l -k\e^2 } \1_{t'_l \ge 0} \frac{1}{\e^2} \int_{\norm{v''} \le m} \\
&&\times \norm{\e^{1/2} h(s',X_{cl}(s';t'_l,x'_l,v'_l),v'')} \dd v'' \dd s' \dd \Sigma_{l'}(s') \dd v' \dd s \dd \Sigma_{l}(s),
\end{eqnarray*}

where
\begin{eqnarray*}
t'_l  &=& t_l(s, X_{cl}(s;t,x,v),v'),\\
x'_l  &=& x_l(s, X_{cl}(s;t,x,v),v'),\\
v'_l  &=& v_l(s, X_{cl}(s;t,x,v),v').
\end{eqnarray*}

$I_3$ is estimated as
\begin{eqnarray*}
I_3 &\lesssim& CT_0^{5/2} \biggl \{ e^{-\frac{C_0}{\e^2}t} \normm{\e^{1/2}h_0}_{\infty} + \e^{1/2} \sup_{0 \le s \le t} \normm{wr(s)}_{\infty} + \e^{3/2} \sup_{0 \le s \le t} \normm{\inn{v}^{-\th} wg(s)}_{\infty} \biggr \} \\
&&+ o(1) CT_0^{5/2} \sup_{0\le s\le t} \normm{\e^{1/2}h(s)}_{\infty} + T_0^{5/4} \frac{1}{2}^{C_2 T_0^{5/4}} \sup_{0\le s\le t} \normm{\e^{1/2}h(s)}_{\infty},
\end{eqnarray*}
which is already included in RHS of \eqref{inftystep1}.

Now we focus on $I_1$ and $I_2$. Consider the change of variables
\begin{eqnarray*}
v'_{l'} \to X_{cl}(s;t'_{l'},x'_{l'},v'_{l'}).
\end{eqnarray*}
For $0 \le t'_{l'} \le s- \k \e^2 \le s \le \e T_0$,
\begin{eqnarray*}
\frac{\pt X_{i}(s;t'_{l'})}{\pt v'_j} = -\frac{t'_{l'}-s}{\e} \delta_{ij}
\end{eqnarray*}
and therefore
\begin{eqnarray*}
\det \nabla_{v'_{l'}} X_{cl}(s;t'_{l'},x'_{l'},v'_{l'}) = \left(\frac{-t'_{l'} - s}{\e} \right) ^3 \ge \k^3 \e^3.
\end{eqnarray*}
Applying the change of variables, we obtain:
\begin{eqnarray*}
&&\int_{\norm{v'_{l'}} \le m} \int_{\norm{v''} \le m} \norm{h(s,X_{cl}(s';t'_{l'},x'_{l'},v'_{l'}),v'')} \dd v'' \dd v'_{l'} \\
&&\lesssim_{m} \int_{\norm{v'_{l'}} \le m} \int_{\norm{v''} \le m} (\norm{\P_{1} f}+\norm{\P_{2} f} + \norm{{\P}_{3} f} + \norm{{\ip} f})(s,X_{cl}(s';t'_{l'},x'_{l'},v'_{l'}),v'') \dd v'' \dd v'_{l'} \\
&&\lesssim_{m} [\int_{\norm{v'_{l'}} \le m} \int_{\norm{v''} \le m} \norm{\P_{1} f(s,X_{cl}(s';t'_{l'},x'_{l'},v'_{l'}),v'')}^6 \dd v'' \dd v'_{l'} ]^{1/6}\\
&&+ [\int_{\norm{v'_{l'}} \le m} \int_{\norm{v''} \le m} \norm{\P_{2} f(s,X_{cl}(s';t'_{l'},x'_{l'},v'_{l'}),v'')}^2 \dd v'' \dd v'_{l'} ]^{1/2}\\
&&+ [\int_{\norm{v'_{l'}} \le m} \int_{\norm{v''} \le m} \norm{{\P}_{3} f(s,X_{cl}(s';t'_{l'},x'_{l'},v'_{l'}),v'')}^3 \dd v'' \dd v'_{l'} ]^{1/3}\\
&&+ [\int_{\norm{v'_{l'}} \le m} \int_{\norm{v''} \le m} \norm{{\ip} f(s,X_{cl}(s';t'_{l'},x'_{l'},v'_{l'}),v'')}^2 \dd v'' \dd v'_{l'} ]^{1/2}\\
&&\lesssim_{m} [\int_{\Omega^c} \int_{\norm{v''} \le m} \norm{\P_{1} f(s)}^6 \frac{1
}{k^3 \e^3} \dd v'' \dd x ]^{1/6}\\
&&+ [\int_{\Omega^c} \int_{\norm{v''} \le m} \norm{\P_{2} f(s)}^2 \frac{1
}{k^3 \e^3} \dd v'' \dd x ]^{1/2}\\
&&+ [\int_{\Omega^c} \int_{\norm{v''} \le m} \norm{{\P}_{3} f(s)}^3 \frac{1
}{k^3 \e^3} \dd v'' \dd x ]^{1/3}\\
&&+ [\int_{\Omega^c} \int_{\norm{v''} \le m} \norm{{\ip} f(s)}^2 \frac{1
}{k^3 \e^3} \dd v'' \dd x ]^{1/2}\\
&& \lesssim \e^{-1/2} \normm{\P_{1} f}_{L^6} +\e^{-3/2} \normm{\P_{2} f}_{L^2} + \e^{-1} \normm{\P_{3} f}_{L^3} +\e^{-3/2} \normm{{\ip} f}_{L^2}.
\end{eqnarray*}
Similarly,
\begin{eqnarray*}
&&\int_{\norm{v'_{l'}} \le m} \int_{\norm{v''} \le m} \norm{h(s,X_{cl}(s';t'_{l'},x'_{l'},v'_{l'}),v'')} \dd v'' \dd v'_{l'} \\
&&\lesssim_{m} \e^{-3/2} \normm{ f}_{L^2}.
\end{eqnarray*}
Hence,
\begin{eqnarray*}
&& I_1 + I_2 \\
&& \lesssim T_0^{5/2} (\sup_{n \e T_0\le s\le (n+1) \e T_0 } \normm{ \P_{1} f(s)}_{L^6(\Omega^{c})}+ \frac{1}{\e} \sup_{n \e T_0\le s\le (n+1) \e T_0 } \normm{ \P_{2} f(s)}_{L^{2}(\Omega^c \times R^3)}) 
\end{eqnarray*}
and
\begin{eqnarray*}
I_1 + I_2 \lesssim T_0^{5/2} \frac{1}{\e} \sup_{n \e T_0\le s\le (n+1) \e T_0 } \normm{ f(s)}_{L^{2}(\Omega^c \times R^3)}.
\end{eqnarray*}
\newline
\textbf{Step 2)} Applying \eqref{inftystep1} successively we obtain, 
\begin{eqnarray*}
&&\normm{\e h(n\e T_0,\cdot,\cdot)}_{\infty}\\
&&\le CT_0^{5/2} e^{-\frac{C_0 T_0}{\e}} \normm{\e^{1/2} h((n-1)\e T_0,\cdot,\cdot)}_{\infty} + \sup_{(n-1)\e T_0 \le s \le n \e T_0} D(s) \\
&&\le \left[ CT_0^{5/2} e^{-\frac{C_0 T_0}{\e}} \right]^{n} \normm{\e^{1/2} h_0}_{\infty}
+ \sum_{j=1}^{n-1} \left[ CT_0^{5/2} e^{-\frac{C_0 T_0}{\e}} \right]^{j} \sup_{0 \le s \le n\e T_0} D(s),
\end{eqnarray*}
where
\begin{eqnarray*}
D(s) &:=& \normm{\e^{1/2} wr(s)}_{\infty} + C T_0^{5/2} \normm{\inn{v}^{-1} \e^{3/2} wg(s)}_{\infty} + CT_0^{5/2} \normm{\P_{1} f(s)}_{L^6_{x,v}} \\ 
&&+ CT_0^{5/2} \frac{1}{\e} \normm{\P_{2} f(s)}_{L^2_{x,v}}+ CT_0^{5/2} \e^{-1/2} \normm{{\P}_{3} f(s)}_{L^3_{x,v}} \\
&&+ CT_0^{5/2} \frac{1}{\e} \normm{{\ip f(s)}}_{L^2_{x,v}} \\
&&+\left[ CT_0^{5/4} \frac{1}{2} ^{C_2 T_0^{5/4}} + o(1) CT_0^{5/4}\right] \normm{\e^{1/2} h(s)}_{\infty},
\end{eqnarray*}
or
\begin{eqnarray*}
D(s) &:=& \normm{\e^{1/2} wr(s)}_{\infty} + C T_0^{5/2} \normm{\inn{v}^{-1} \e^{3/2} wg(s)}_{\infty} +  CT_0^{5/2} \frac{1}{\e} \normm{ f(s)}_{L^2(\Omega^c)} \\
&&+\left[ CT_0^{5/4} \frac{1}{2} ^{C_2 T_0^{5/4}} + o(1) CT_0^{5/4}\right] \normm{\e^{1/2} h(s)}_{\infty}.
\end{eqnarray*}

Clearly $\sum_{j=1}^{n-1} \left[ CT_0^{5/2} e^{-\frac{C_0 T_0}{\e}} \right]^{j} < \infty$.
Combining the above estimate with \eqref{inftystep1}, for $t \in [n \e T_0, (n+1)\e T_0]$, and absorbing the last term,
\begin{eqnarray*}
&&CT_0^{5/2} e^{-\frac{C_0 (t-n\e T_0)}{\e^2}} \sum_{j=1}^{n-1} \left[ CT_0^{5/2} e^{-\frac{C_0 T_0}{\e}} \right]^{j} \\
&&\times \left[ CT_0^{5/4} \frac{1}{2} ^{C_2 T_0^{5/4}} + o(1) CT_0^{5/4}\right] \normm{\e^{1/2} h(s)}_{\infty}\\
&& \lesssim \frac{T_0^{5/2}}{1- CT_0^{5/2} e^{-\frac{C_0 T_0}{\e}}}\left[ CT_0^{5/4} \frac{1}{2} ^{C_2 T_0^{5/4}} + o(1) CT_0^{5/4}\right] \normm{\e^{1/2} h(s)}_{\infty} \\
&&\lesssim o(1) \normm{\e^{1/2} h(s)}_{\infty}.
\end{eqnarray*}
Thus, we can conclude Proposition \ref{linftyestimate}.
\end{proof}

\section{{$L^2$- $L^3$ - $L^6$} Splitting} \label{l2l6split}
The main purpose of this section is to prove Theorem \ref{l2l6estimate}.
In bounded case, \cite{Esposito2018-apde} use $L^6$ estimate. One of the key idea to control $L^6$ norm is using Poincare inequality which does not holds for exterior domain. We split $\P f$ into three terms: one that requires the Poincare inequality, another that does not require it, and the last one is a boundary term.

\begin{lemma} \label{65estimate}
There exists a unique solution to $-(\Delta + \pt_i^2) \phi  = h \in L^{6/5}(\R^3)$ (i=1,2,3) or $-\Delta \phi  = h \in L^{6/5}(\R^3)$ such that
\begin{align*}
\normm{\nabla \phi}_{L^2} + \normm{\phi}_{L^6} + 
\normm{\nabla^2 \phi}_{L^{6/5}} 
\le \normm{h}_{L^{6/5}}.
\end{align*}
\end{lemma}

\begin{proof}
We solve $-(\Delta + \pt_i^2) \phi  = h \in L^{6/5}(\R^3)$ by the Lax-Milgram theorem : define a bilinear form
\begin{eqnarray*}
\left(\left( \phi,  \psi \right)\right)_1 &\equiv& \int_{\Omega^c} \nabla \phi \cdot \nabla \psi \dd x \dd v,\\
\left(\left( \phi,  \psi \right)\right)_2 &\equiv& \int_{\Omega^c} \nabla \phi \cdot \nabla \psi + \pt_i \phi \pt_i \psi \dd x \dd v,
\end{eqnarray*}
with the functional $h$ defined by
\begin{align*}
\langle h, \psi \rangle \equiv \int_{\R^3} h \psi \dd x \dd v .
\end{align*}
We have the Sobolev inequality
\begin{align*}
\normm{\xi}_{L^6 (\R^3)} \lesssim  \normm{\nabla \xi}_{L^2 (\R^3)}.
\end{align*}
Therefore $\langle h, \psi \rangle$ defines a bounded linear functional in $\dot{H}^1 (\R^3)$ thanks to the inequality
\begin{eqnarray*}
\langle h, \psi \rangle = \int_{\Omega^c} f \psi \dd x \dd v &\le& \normm{h}_{L^{6/5}} \normm{\psi}_{L^6} \le c_h \normm{\nabla \psi}_{L^2}, \\
\left(\left( \phi,  \psi \right)\right)_{1,2}  &\le& 2 \normm{\nabla \phi}_{L^2} \normm{\nabla \psi}_{L^2}.
\end{eqnarray*}
The existence and uniqueness as well as the first two inequalities then follows from the Lax-Milgram theorem.
\end{proof}

Recall the linearized Boltzmann equation,
\begin{eqnarray}
\label{lineareq}
\e \pt_t f + v \cdot \nabla f + \e^{-1} L f = g  &\quad& \text{ in } \R_{+} \times \Omega^{c} \times \R^3,\\
f = \mathscr{P}_{\gamma} f + \e^{1/2} r &\quad& \text{ on } \g_{-} .
\end{eqnarray}

Multiply test function $\psi$ and integration by parts then we can get
\begin{eqnarray*}
&&\int_{\Omega^c \times \R^3} \e \pt_t f  \psi \dd x \dd v + \int_{\g_+} f \psi \dd \g - \int_{\Omega^c \times \R^3} f v \cdot \nabla \psi \dd x \dd v + \e^{-1} \int_{\Omega^c \times \R^3} \psi L_{ } f \dd x \dd v \\
&& = \int_{\Omega^c \times \R^3} g \psi \dd x \dd v + \int_{\g_{-}} (P_{\g}^{ } f + \e^{\frac{1}{2}}r)\psi \dd \g.
\end{eqnarray*}

\begin{definition}
(Space of 13 moments) Define $\phi_i$ as follows.
\begin{eqnarray*}
&\phi_1 = \mu^{1/2}, \Hquad \phi_{j+1} = v_j \mu^{1/2} \Hquad (j=1,2,3),\Hquad \phi_{j+4} = (v_j^2-1) \mu^{1/2} \Hquad (j=1,2,3),\Hquad \\
& \phi_{8} = v_1 v_2 \mu^{1/2}, \Hquad \phi_{9} = v_2 v_3 \mu^{1/2}, \Hquad \phi_{10} = v_3 v_1 \mu^{1/2}, \Hquad \\
&\phi_{j+10} = (\norm{v}^2-5) v_j \mu^{1/2} \Hquad (j=1,2,3).
\end{eqnarray*}
Then, we define $W$ as
\begin{eqnarray}
W := \text{linear span of } \{\phi_i\}_{i=1}^{13}.
\end{eqnarray}
\end{definition}

\begin{lemma}
$\{\phi_i\}_{i=1}^{13}$ are orthogonal bases of $W$.
In addition, $\normm{\phi_{5}} =\normm{\phi_{6}} = \normm{\phi_{7}}$,  $\normm{\phi_{8}} =\normm{\phi_{9}} = \normm{\phi_{10}}$, $\normm{\phi_{11}} =\normm{\phi_{12}} = \normm{\phi_{13}}$ and $\normm{\phi_i}_{\infty} < \infty$.
\end{lemma}
\begin{proof}
This come from chapter 3.9 of \cite{Glassey1996}. 
\end{proof}
By direct calculation, we can derive the following corollary.
\begin{corollary} \label{basestructure}
The map $v_i \times : \ker L \to W$ satisfies the following relation.
\begin{equation*}
\begin{array} {cccc}
        & v_1             & v_2             & v_3             \\
\phi_1  & \phi_2          & \phi_3          & \phi_4          \\
\phi_2  & \phi_5   & \phi_8   & \phi_{10}  \\
\phi_3  & \phi_8   & \phi_6   & \phi_9   \\
\phi_4  & \phi_{10}  & \phi_9   & \phi_7   \\
\phi_5+\phi_6+\phi_7  & \phi_{11}-2\phi_2  & \phi_{12}-2\phi_3   & \phi_{13}-2\phi_4
\end{array}
\end{equation*}
\end{corollary}

\begin{definition}
Let $C_c^{\infty} (\R^3)$ as smooth function with compact support and $\dot{W}^{2,p}(\R^3)$ as a completion of $C_c^{\infty} (\R^3)$ with a norm $\normm{\phi}_{\dot{W}^{2,p}} := \normm{\nabla^2 \phi}_{L^p} $. If $p=2$ we called this space as $H^2$ with $\normm{\cdot}_{H^2}$ norm.
\end{definition}

The lemma \ref{ellipticlemma} is key to constructing the $\P_{2}$ and $\P_{3}$ parts for the $L^2-L^3-L^6$ splitting.
According to the lemma, we can split $\P f$ into an $L^2$ term, an $L^3$ term and an $L^6$ term. For the kinetic regime, we can gain a factor of $\e$, while for the boundary parts, we can gain a $\e^{1/2}$.  Before proving the lemma, we will introduce generalized Lax-Milgram theorem.
\begin{lemma} \label{generlalaxmilgram}
(Generalized Lax-Milgram Theorem \cite{Kozono2013}) Let $X$ is a Banach space with the norm $\normm{\cdot}_{X}$ and $Y$ is a reflexive Banach space with the norm $\normm{\cdot}_{Y}$. Suppose that $a: X \times Y \to \mathbb{C}$ is a bilinear form satisfying the following assumption.
\begin{enumerate}
    \item []\textbf{Assumption}
    \item There is a constant $M>0$ such that 
    \begin{align*}
    \norm{a(u,\varphi)} \le M \normm{u}_{X} \normm{\varphi}_{Y} \text{ for all } u \in X \text{ and } \varphi \in Y.
    \end{align*}
    \item Let $N_{X}:=\{u \in X : a(u,\varphi) = 0, \forall \varphi \in Y \}$ and let $N_{Y}:=\{\varphi \in X : a(u,\varphi) = 0, \forall u \in X \}$. There are closed subspace $R_{X}$ in $X$ and $R_{Y}$ in $Y$ such that
    \begin{align*}
    X =& N_{X} \oplus R_{X} \text{(direct sum)},\\
    Y =& N_{Y} \oplus R_{Y} \text{(direct sum)}.
    \end{align*}
    \item There is a constant $C>0$ such that
    \begin{align*}
    \normm{u}_{X} \le & C\left( \sup_{\varphi \in Y} \frac{\norm{a(u,\varphi)}}{\normm{\varphi}_{Y}} + \normm{P_{X}u}_{X} \right) \text{ for all } u \in X,\\
    \normm{\varphi}_{Y} \le & C\left( \sup_{\varphi \in X} \frac{\norm{a(u,\varphi)}}{\normm{u}_{Y}} + \normm{P_{Y}\varphi}_{Y} \right) \text{ for all } \varphi \in Y,
    \end{align*}
    where $P_{x}$ and $P_{Y}$ are the projection from $X$ onto $N_{X}$ and from $Y$ onto $N_{Y}$ respectively.
\end{enumerate}
Then for every $F \in N_{Y}^{\perp}$, i.e., $F \in Y^{*}$ with $F(\phi) =0 $ for all $\phi \in N_{Y}$ there exist $w \in X$ such that 
\begin{align*}
a(w,\varphi) = F(\varphi) \text{ for all } \varphi \in Y.
\end{align*}
Such $w$ is subject to the estimate
\begin{align*}
\normm{w}_{X} \le C \normm{F}_{Y^{*}},
\end{align*}
where $C$ is a constant independent of $w$ and $F$.
\end{lemma}
\begin{proof}
Theorem 1.1 of \cite{Kozono2013}
\end{proof}

\begin{lemma} \label{ellipticlemma}
Let $A = -\Delta$ or $A = -\Delta - \pt_{i}^2$ $(i=1,2,3)$. Then, there exists $u_f \in L^2(\R^3)$ and $u_{\g} \in L^3(\R^3)$ that satisfy the following equation for all test functions $\varphi \in C_{c}^{\infty}(\R^3)$ and a fixed $\{ \alpha_{i}\}_{i=1}^{13} \in \R$.
\begin{align*}
\int_{\R^3} u_f A \varphi \dd x &= \int_{\Omega^c \times \R^3} \ip f \sum_{i=1}^{13} \a_i v \phi_i \nabla^2 \varphi \dd x \dd v,\\
\int_{\R^3} u_{\g} A \varphi \dd x &= \int_{\g_{-}} ((P_{\g}^{ }-1) f + \e^{\frac{1}{2}}r)\sum_{i=1}^{13} \a_i v \phi_i \nabla \varphi \dd \g.
\end{align*}
In addition 
\begin{align*}
\normm{u_f}_{L^2} & \lesssim \normm{\ip f}_{L^2}, \\
\normm{u_{\g}}_{L^3} & \lesssim (\norm{(1-P_{\g}^{ })f}_{2,+} + \e^{1/2} \norm{r}).
\end{align*}
\end{lemma}
\begin{proof}
Define $V = \dot{H}^{2}(\R^3)$, $V_{1} = L^{3}(\R^3)$ and $V_{2} = \dot{W}^{2,\frac{3}{2}}(\R^3)$ with a norm $\normm{\nabla^2 \cdot}_{L^2}$, $\normm{ \cdot}_{L^3}$ and $\normm{\nabla^2 \cdot}_{L^{\frac{3}{2}}}$ respectively. Define bilinear form $B_{1} : V \times V \to \R$, $B_{2} : V_{1} \times V_{2} \to \R$ and linear functional $g_{1} \in V^{*}, g_{2} \in V_{2}^{*}$ as follows:
\begin{align*}
B_{1}(u,v) &= \int_{\R^3} A u A v \dd x \\
B_{2}(u,v) &= \int_{\R^3} u A v \dd x
\end{align*}
\begin{align*}
\inn{g_{1},\varphi} = \int_{\R^3 \times \R^3} \ip f \sum_{i=1}^{13} \a_i v \phi_i \nabla^2 \varphi \dd x \dd v, \\
\inn{g_{2}, \varphi} = \int_{\g_{-}} ((P_{\g}^{ }-1) f + \e^{\frac{1}{2}}r)\sum_{i=1}^{13} \a_i v \phi_i \nabla \varphi \dd \g.
\end{align*} 
Note that $\ip f$ is a function defined on $\Omega^c$, but we can naturally extend it to $\R^3$ by defining it as 0 in $\Omega$.

First, let us check $g_{1} \in V^{*}$. It is clear that $g_{1}$ is linear. In addition,
\begin{eqnarray*}
\int_{\R^3 \times \R^3} \ip f \sum_{i=1}^{13} \a_i v \phi_i \nabla^2 \varphi \dd x \dd v &=& \int_{\R^3} \sum_{i=1}^{13} \nabla^2 \varphi \int_{\R^3} \ip f \a_i v \phi_i \dd v  \dd x \\
&\lesssim & \int_{\R^3} \nabla^2 \varphi \int_{\R^3} \ip f \dd v  \dd x \\
&\lesssim & \normm{\nabla^2 \varphi}_{L^2(\R^3)} \normm{\ip f}_{L^2(\Omega^c)}.
\end{eqnarray*}

So, $g_{1}$ is a bounded linear form on $V$ which means $g_{1} \in V^{*}$.

Next, let us check $g_{2} \in V_{2}^{*}$. It is clear that $g_{2}$ is linear.
By the assumption that $\Omega$ is a $C^1$ domain in $\R^N$ with $N=3$, we can use the following trace estimate (see \cite{Leoni2009},page 466):
\begin{align*}
\left(\int_{\pt \Omega}\norm{u}^{\frac{p(N-1)}{N-p}} \dd S(x)\right)^{\frac{N-p}{p(N-1)}} \le C(N,p) \left(\int_{ \Omega}\norm{u}^{p}\dd x + \int_{ \Omega}\norm{\nabla u}^{{p}}\dd x\right)^{\frac{1}{p}}.
\end{align*}
This is a consequence of the trace theorem $W^{1,p}(\Omega) \to W^{1-\frac{1}{p},p} (\pt \Omega)$, and the Soboelv embedding in $N-1$ dimensional sub-manifold ($W^{1-\frac{1}{p},p} (\pt \Omega) \subset L^{\frac{p(N-1)}{N-p}} (\pt \Omega)$ for $\frac{N-p}{p(N-1)} = \frac{1}{p} - \frac{1-\frac{1}{p}}{N-1}$). In particular, with $p = \frac{3}{2}$ and $N=3$ we have $\frac{p(N-1)}{N-p} = 2$. 
Thus we can conclude
\begin{align*}
\normm{u}_{L^{2}(\pt \Omega)} \lesssim \normm{\nabla u}_{L^{\frac{3}{2}}(\R^3)}.
\end{align*}
With $u = \nabla \varphi$, we have
\begin{align*}
\normm{\nabla \varphi}_{L^{2}(\pt \Omega)} \lesssim \normm{\nabla^2 \varphi}_{L^{\frac{3}{2}} (\R^3)}.
\end{align*}
Which leads
\begin{align*}
\int_{\g_{-}} ((P_{\g}^{ }-1) f + \e^{\frac{1}{2}}r)\sum_{i=1}^{13} \a_i v \phi_i \nabla \varphi \dd \g \lesssim (\norm{(1-P_{\g}^{ })f}_{2,+} + \e^{1/2} \norm{r}) \normm{\nabla^2 \varphi}_{L^{\frac{3}{2}} (\R^3)}.
\end{align*}
Thus, $g_{2}$ is a bounded linear form on $V_{2}$ which means $g_{2} \in V_{2}^{*}$

Next, we need to check that $B_{1}$ and $B_{2}$ satisfy the conditions for the Generalized Lax-Milgram theorem. The boundedness and linearity conditions for $B_{1}$ and $B_{2}$ are straightforward by Holder's inequality. We only need to check the coercive condition.
According to the standard elliptic estimate (see \cite{Gilbarg2001} Theorem 9.9)
\begin{align*}
\normm{\nabla^2 u}_{p} \lesssim \normm{A u}_{p},
\end{align*}
for both differential operator $A$ and $1< p <\infty$.
For the operator $B_{1}$ we choose $u=v$ then,
\begin{align*}
\normm{\nabla^2 u}_{2}^{2} \lesssim \normm{A u}_{2}^2 = B_{1}(u,u).
\end{align*}
Thus, $B_{1}$ is coercive.

It is clear that $V_{1}$ and $V_{2}$ are Banach spaces. In addition, since $L^{p}$ is uniformly convex with respect to the $L^{p}$ norm for $1 <p < \infty $, $V_{2} = \dot{W}^{2,p}$ is also a uniform convex Banach space. Thus, $V_{2}$ is reflexive by the Milman–Pettis theorem. By the definition of the bilinear form $B_{2}$ the null set of $B_{2}$ is the zero set (see Assumption 2 of lemma \ref{generlalaxmilgram}). Next, we will show
\begin{align*}
\sup_{\normm{v}_{\dot{W}^{2,\frac{3}{2}}}=1} \norm{B_{2}(u,v)} \gtrsim& \normm{u}_{L^{3}},\\
\sup_{\normm{u}_{L^{3}}=1} \norm{B_{2}(u,v)} \gtrsim& \normm{v}_{\dot{W}^{2,\frac{3}{2}}}.
\end{align*}
According to the standard duality argument
\begin{align*}
\sup_{\normm{u}=1} \norm{B_{2}(u,v)} = \normm{Av}_{\frac{3}{2}} \gtrsim \normm{v}_{\dot{W}^{2,\frac{3}{2}}}.
\end{align*}
If $u =0$ $\sup_{\normm{v}=1} \norm{B_{2}(u,v)} = 0$.
If $u \neq 0$ according to the duality of $L^{3}$ and $L^{\frac{3}{2}}$, there exist $u^{*} \in L^{\frac{3}{2}}$ and $\normm{u^{*}}_{\frac{3}{2}} =1$ such that 
\begin{align*}
\int_{\R^3} u u^{*} \dd x > \frac{1}{2} \normm{u}_{3}.
\end{align*}
We choose $v^{*}$ as a solution of $Av^{*} = u^{*}$ in $\R^3$. Then, for some $C>0$, $\normm{\nabla^2 v^{*}}_{\frac{3}{2}} < C\normm{Av^{*}}_{\frac{3}{2}} = C\normm{u^{*}}_{\frac{3}{2}} =C$. Thus, 
\begin{align*}
\sup_{\normm{v}=1} \norm{B_{2}(u,v)} \ge \norm{B_{2}(u,\frac{v^{*}}{C})} = \frac{1}{C}\int_{\R^3} u u^{*} \dd x > \frac{1}{2C} \normm{u}_{3}.
\end{align*}
Thus, $B_{1}$ and $B_{2}$ satisfies the condition for the Generalized Lax-Milgram theorem. Then we can conclude the lemma \ref{ellipticlemma}.
\end{proof}
\begin{lemma} (Approximate lemma) For any $\phi \in \dot{W}^{2,6/5}$ and fixed $\e$, there exists $\phi_{\e} \in C^\infty_{c}({\R^3})$ which satisfies the following inequality.
\begin{eqnarray*}
\normm{\phi - \phi_{\e}}_{\dot{H}^2(\R^3)} = \normm{\nabla^2 \phi - \nabla^2 \phi_{\e}}_{L^2(\R^3)} \le \e
\end{eqnarray*}
\end{lemma}

\begin{proof}
\cite{Galdi2011} Chapter II.7
\end{proof}

Now we are ready to prove Theorem \ref{l2l6estimate}.

\begin{proof} (proof of the Theorem \ref{l2l6estimate})
\newline
\textbf{Step 1)} Splitting of $c$.

To get the bound of $c$, choose a test function $\psi_c$ as:
\begin{align*}
\psi_c = \sqrt{\mu_{ }} (  \norm{v_{ }}^2 -5) {v_{ }} \cdot \nabla \varphi_c = \phi_{11} (v_{ }) \pt_1 \varphi_c + \phi_{12} (v_{ }) \pt_2 \varphi_c+\phi_{13} (v_{ }) \pt_3 \varphi_c.
\end{align*}
$\varphi_c$ is to be determined.
$\psi_c \perp \ker L_{ }$ because of the definition of $\psi_c$.

\begin{eqnarray} \label{integbypartc}
&&\int_{\Omega^c \times \R^3} \e \pt_t f  \psi_c \dd x \dd v + \int_{\g_+} f \psi_c \dd \g \\ \nonumber
&& - \int_{\Omega^c \times \R^3} f v \cdot \nabla \psi_c \dd x \dd v + \e^{-1} \int_{\Omega^c \times \R^3} \psi_c L_{ } f \dd x \dd v \\ \nonumber
&& = \int_{\Omega^c \times \R^3} g \psi_c \dd x \dd v + \int_{\g_{-}} (P_{\g}^{ } f + \e^{\frac{1}{2}}r)\psi_c \dd \g.
\end{eqnarray}
\newline
\textbf{STEP 1-1)} Construction of $c_{2}$ and $L^2$ estimate of ${c_{2}}$.

According to the Lemma \ref{ellipticlemma} there exist $c_{2}$ satisfies
\begin{eqnarray} \label{cbardefn}
-5\int_{\Omega^c} c_{2} \Delta \varphi \dd x &=& \int_{\Omega^c \times \R^3} \ip f \sum_{i=1}^{3} \sum_{j=1}^{3} v_j \phi_{10+i} \pt_i \pt_j \varphi \dd x \dd v,
\end{eqnarray}
for all $\varphi \in C_c^{\infty} (\R^3) \subset \dot{H}^{2}(\R^3)$, and the following inequality holds
\begin{align} \label{c2l2estimate}
\normm{c_{2}}_{L^2} \lesssim \normm{\ip f}_{L^2}.
\end{align}
\newline
\textbf{STEP 1-2)} Construction of $c_{3}$ and $L^3$ estimate of ${c_{3}}$.

According to the Lemma \ref{ellipticlemma} there exist $c_{3}$ satisfies
\begin{eqnarray} \label{cgammadefn}
-5\int_{\Omega^c} c_{3} \Delta \varphi \dd x &=& 
\int_{\g_{-}} ((P_{\g}^{ }-1) f + \e^{\frac{1}{2}}r)\sum_{i=1}^{3} \sum_{j=1}^{3} v_j \phi_{10+i} \pt_i \pt_j \varphi \dd \g,
\end{eqnarray}
for all $\varphi \in C_c^{\infty} (\R^3) \subset \dot{W}^{2,\frac{3}{2}}(\R^3)$, and the following inequality holds
\begin{align}\label{c3l3estimate}
\normm{c_{3}}_{L^3} \lesssim \norm{(1-P_{\g}^{ })f}_{2,+} + \e^{1/2} \norm{r} .    
\end{align}
\newline
\textbf{STEP 1-3)} Construction of $c_{1}$ and $L^6$ estimate of ${c_{1}}$.

By \eqref{cbardefn}, \eqref{integbypartc} takes the following form:
\begin{eqnarray*}
&&\int_{\Omega^c \times \R^3} \e \pt_t f  \psi_c \dd x \dd v  - \int_{\Omega^c \times \R^3} \P f v \cdot \nabla \psi_c \dd x \dd v + \e^{-1} \int_{\Omega^c \times \R^3} \psi_c L_{ } f \dd x \dd v \\
&& = \int_{\Omega^c \times \R^3} g \psi_c \dd x \dd v -5\int_{\R^3} c_{2} \Delta \varphi_c \dd x -5\int_{\R^3} c_{3} \Delta \varphi_c \dd x,
\end{eqnarray*}
if $\varphi_c \in C_c^{\infty}(\R^3)$.

Note that
\begin{align*}
\int_{\Omega^c \times \R^3} \P f v \cdot \nabla \psi_c \dd x \dd v =& \int_{\Omega^c \times \R^3} a \sqrt{\mu_{ }} v_{ } \cdot \nabla \psi_c \dd x \dd v \\
&+ \int_{\Omega^c \times \R^3} b \cdot v_{ } \sqrt{\mu_{ }} v_{ } \cdot \nabla \psi_c \dd x \dd \\
&+\int_{\Omega^c \times \R^3} c \frac{\norm{v_{ }}^2-3}{2} \sqrt{\mu_{ }} v_{ } \cdot \nabla \psi_c \dd x \dd v.
\end{align*}
By the definition of $\psi_c$, corollary \ref{basestructure} and direct computation,
\begin{align*}
\int_{\Omega^c \times \R^3} a \sqrt{\mu_{ }} v_{ } \cdot \nabla \psi_c \dd x \dd v&=0,\\
\int_{\Omega^c \times \R^3} b \cdot v_{ } \sqrt{\mu_{ }} v_{ } \cdot \nabla \psi_c \dd x \dd v&=0,\\
\int_{\Omega^c \times \R^3} c \frac{\norm{v_{ }}^2-3}{2} \sqrt{\mu_{ }} v_{ } \cdot \nabla \psi_c \dd x \dd v&= 5\int_{\Omega^c } c \Delta \varphi_c \dd x .
\end{align*}
Define $c_{1}$ as
\begin{equation} \label{ctildedefn}
c_{1} := c-c_{2}-c_{3} = \left\{
\begin{aligned}
c_{2}-c_{3} &\text{ in } \Omega \\
c-c_{2}-c_{3} &\text{ in } \Omega^c    
\end{aligned} \right.
\end{equation} 
\eqref{integbypartc} becomes
\begin{eqnarray} \label{tildecestimate}
&&\int_{\Omega^c \times \R^3} \e \ip \pt_t f  \psi_c \dd x \dd v  -5\int_{\R^3 \times \R^3} c_{1} \Delta \varphi_c \dd x \dd v + \e^{-1} \int_{\Omega^c \times \R^3} \psi_c L_{ } f \dd x \dd v\\
\nonumber
&&= \int_{\Omega^c \times \R^3}  \ip  g \psi_c \dd x \dd v.
\end{eqnarray}
By Lemma \ref{65estimate}, there exists a unique $\varphi_c$ that is a solution of
\begin{align*}
- \Delta \varphi_c = c_{1}^5 \text{ in } \R^3
\end{align*}
and satisfies
\begin{eqnarray*}
\normm{\varphi_c}_{\dot{H}^{1}(\R^3)} \le \normm{c_{1}^{5}}_{L^{\frac{6}{5}}(\R^3)} = \normm{c_{1}}_{L^{6}(\R^3)}^5.
\end{eqnarray*}

Since, $C_{c}^{\infty}(\R^3)$ is dense in $\dot{W}^{2,6/5}(\R^3)$ there exist
$\varphi_{c,N} \in C_{c}^{\infty}(\R^3)$ that $$\normm{- \Delta \varphi_{c,N}-c_{1}^5}_{L^{\frac{6}{5}}} <\frac{1}{N}\normm{c_{1}^5}_{L^{\frac{6}{5}}}.$$
Putting $\varphi_{c,N}$ into \eqref{tildecestimate}, we obtain:
\begin{eqnarray*}
-\int_{\R^3 \times \R^3} c_{1} \Delta \varphi_{c,N} \dd x \dd v &=& -\int_{\R^3 \times \R^3} c_{1} \Delta \varphi_{c,N} \dd x \dd v \\
&=& -\int_{\R^3 \times \R^3} c_{1} (\Delta \varphi_{c,N}-c_{1}^5) \dd x \dd v + \int_{\R^3 \times \R^3} c_{1}^6 \dd x \dd v \\
&>& \normm{c_{1}}_{L^6}^6 - \frac{1}{N}\normm{c_{1}}_{L^6} \normm{c_{1}^5}_{L^{\frac{6}{5}}} \\
&=&(1-\frac{1}{N}) \normm{c_{1}}_{L^6}^6.
\end{eqnarray*}

Moreover,
\begin{eqnarray*}
\int_{\Omega^c \times \R^3} \e \pt_t f  \psi_{c,N} \dd x \dd v &=& \int_{\Omega^c \times \R^3} \e \ip \pt_t f  \psi_{c,N} \dd x \dd v \\
& \le & \e \normm{\psi_c}_{2} \normm{\ip \pt_t f}_{2} \\
& \le & \e \normm{\nabla \varphi_{c,N}}_{2} \normm{\ip \pt_t f}_{2} \\
& \lesssim & \e \normm{c_{1}}_{6}^5 \normm{\ip \pt_t f}_{2}.
\end{eqnarray*}
Similarly,
\begin{eqnarray*}
\e^{-1} \int_{\Omega^c \times \R^3} \psi_{c,N} L_{ } f \dd x \dd v &\lesssim&  \e^{-1} \normm{c_{1}}_{6}^5 \normm{\ip  f}_{2},\\
\int_{\Omega^c \times \R^3} \psi_{c,N} g \dd x \dd v &\lesssim&  \normm{c_{1}}_{6}^5 \normm{\ip  g}_{2}.
\end{eqnarray*}
Thus, we can get
\begin{align} \label{c1l6estimate}
\normm{c_{1}}_{L^6} \lesssim \e^{-1} \normm{\ip f}_{\nu} +\e\normm{\pt_t \ip f } + \normm{g \nu^{-1/2}}_{2}.
\end{align}
\newline
\textbf{Step 2)} Splitting of $b$.

Let $b =(b^{1},b^{2},b^{3})$. Without loss of generality we will estimate $b^{1}$. To get the bound of $b^{1}$ choose a test function $\psi_{b}^{1}$ as:
\begin{align*}
\psi_{b}^{1} = \frac{1}{\normm{\phi_{5}}}(2 \phi_{5}- \phi_{6} - \phi_{7}) \pt_{x_1} \varphi_{b}^{1} + \frac{1}{\normm{\phi_{8}}}\phi_8 \pt_{x_2} \varphi_{b}^{1} + \frac{1}{\normm{\phi_{10}}}\phi_{10} \pt_{x_3} \varphi_{b}^{1}.
\end{align*}
$\varphi_{b}^{1}$ is to be determined.
Since $\ker L_{ } = \text{span} \{ \phi_1, \phi_2, \phi_3, \phi_4, \phi_5+\phi_6+\phi_7 \} $ and $\phi_i$ are orthogonal, 
$\psi_{b}^{1} \perp \ker L_{ }$ holds because of the definition of $\psi_{b}^{1}$. Therefore,
\begin{eqnarray} \label{integbypartb1}
&&\int_{\Omega^c \times \R^3} \e \pt_t f  \psi_{b}^{1} \dd x \dd v + \int_{\g_+} f \psi_{b}^{1} \dd \g \\ \nonumber
&&- \int_{\Omega^c \times \R^3} f v \cdot \nabla \psi_{b}^{1} \dd x \dd v + \e^{-1} \int_{\Omega^c \times \R^3} \psi_{b}^{1} L_{ } f \dd x \dd v \\ \nonumber
&& = \int_{\Omega^c \times \R^3} g \psi_{b}^{1} \dd x \dd v + \int_{\g_{-}} (P_{\g}^{ } f + \e^{\frac{1}{2}}r)\psi_{b}^{1} \dd \g.
\end{eqnarray}
\newline
\textbf{Step 2-1)} Construction of $b_{2}$ and $L^2$ estimate of ${b_{2}}$.

According to Lemma \ref{ellipticlemma} there exist $b_{2}^{1}$ satisfies
\begin{eqnarray} \label{bbardefn}
-\int_{\Omega^c} b_{2}^{1} \Delta \varphi \dd x &=& \int_{\Omega^c \times \R^3} \ip )f \psi_{b}^{1} \dd x \dd v,
\end{eqnarray}
for all $\varphi_{b}^{1} \in C_c^{\infty} (\R^3)$, and the following inequality holds
\begin{align} \label{b2l2estimate}
\normm{b_{2}^{1}}_{L^2} \lesssim \normm{\ip f}_{L^2}.
\end{align}
\newline
\textbf{Step 2-2)} Construction of $b_{3}$ and $L^3$ estimate of ${b_{3}}$.

According to Lemma \ref{ellipticlemma} there exist ${b}_{3}^{1}$ satisfies
\begin{eqnarray} \label{bgammadefn}
-\int_{\Omega^c} {b}_{3}^{1} \Delta \varphi \dd x &=& \int_{\g_{-}} ((P_{\g}^{ }-1) f + \e^{\frac{1}{2}}r)\psi_{b}^{1} \dd \g,
\end{eqnarray}
for all $\varphi_{b}^{1} \in C_c^{\infty} (\R^3)$, and the following inequality holds
\begin{align} \label{b3l3estimate}
\normm{{b}_{3}^{1}}_{L^3} \lesssim \norm{(1-P_{\g}^{ })f}_{2,+} + \e^{1/2} \norm{r} .
\end{align}
\newline
\textbf{Step 2-3)} Construction of $b_{1}$ and $L^6$ estimate of ${b_{1}}$.

By \eqref{bbardefn}, \eqref{integbypartb1} takes the following form:
\begin{eqnarray*}
&&\int_{\Omega^c \times \R^3} \e \pt_t f  \psi_{b}^{1} \dd x \dd v  - \int_{\Omega^c \times \R^3} \P f v \cdot \nabla \psi_{b}^{1} \dd x \dd v + \e^{-1} \int_{\Omega^c \times \R^3} \psi_{b}^{1} L_{ } f \dd x \dd v \\
&& = \int_{\Omega^c \times \R^3} g \psi_{b}^{1} \dd x \dd v -\int_{\R^3} b_{2}^{1} \Delta \varphi_{b}^{1} \dd x -\int_{\R^3} b_{3}^{1} \Delta \varphi_{b}^{1} \dd x,
\end{eqnarray*}
if $\varphi_{b}^{1} \in C_c^{\infty}(\R^3)$.


Note that
\begin{align*}
\int_{\Omega^c \times \R^3} \P f v \cdot \nabla \psi_{b}^{1} \dd x \dd v =& \int_{\Omega^c \times \R^3} a \sqrt{\mu_{ }} v_{ } \cdot \nabla \psi_{b}^{1} \dd x \dd v \\
&+ \int_{\Omega^c \times \R^3} b \cdot v_{ } \sqrt{\mu_{ }} v_{ } \cdot \nabla \psi_{b}^{1} \dd x \dd \\
&+\int_{\Omega^c \times \R^3} c \frac{\norm{v_{ }}^2-3}{2} \sqrt{\mu_{ }} v_{ } \cdot \nabla \psi_{b}^{1} \dd x \dd v.
\end{align*}
First compute,
\begin{eqnarray*}
&&\int_{\Omega^c \times \R^3} a \sqrt{\mu_{ }} v_{ } \cdot \nabla \psi_{b}^{1} \dd x \dd v\\
&=& \int_{\Omega^c \times \R^3} \sum_{i=1}^{3} a \sqrt{\mu_{ }} v_{  i} \pt_{i} \psi_{b}^{1} \dd x \dd v \\
&=& \int_{\Omega^c \times \R^3} a ( \phi_{2} \pt_{1} \psi_{b}^{1} + \phi_{3} \pt_{2} \psi_{b}^{1} +\phi_{4} \pt_{3} \psi_{b}^{1}) \dd x \dd v \\
&=& 0.
\end{eqnarray*}
Secondly,
\begin{eqnarray*}
&&\int_{\Omega^c \times \R^3} c \frac{\norm{v_{ }}^2-3}{2} \sqrt{\mu_{ }} v_{ } \cdot \nabla \psi_{b}^{1} \dd x \dd v \\
&=& 
\int_{\Omega^c \times \R^3}  \sum_{i=1}^{3} c \frac{\norm{v_{ }}^2-3}{2} \sqrt{\mu_{ }} v_{  i} \pt_{i} \psi_{b}^{1} \dd x \dd v \\
&=& \int_{\Omega^c \times \R^3} c( \phi_{11} \pt_{1} \psi_{b}^{1} + \phi_{12} \pt_{2} \psi_{b}^{1} +\phi_{13} \pt_{3} \psi_{b}^{1}) \dd x \dd v \\
&=& 0.
\end{eqnarray*}
Lastly,
\begin{eqnarray*}
&&\int_{\Omega^c \times \R^3} b \cdot v_{ } \sqrt{\mu_{ }} v_{ } \cdot \nabla \psi_{b}^{1} \dd x \dd v \\
&=& \int_{\Omega^c \times \R^3} \sum_{i=1}^{3} \sum_{j=1}^{3} b^{i}  v_{ i}  v_{ j} \sqrt{\mu_{ }} \pt_j \psi_{b}^{1} \dd x \dd v\\
&=& \int_{\Omega^c \times \R^3} (\phi_5+\phi_1)b^{1} \pt_1 \psi_{b}^{1} + \phi_8 b^{1} \pt_2 \psi_{b}^{1} + \phi_{10} b^{1} \pt_3 \psi_{b}^{1} \\
&& + \phi_8 b^{2} \pt_1 \psi_{b}^{1} +(\phi_6+\phi_1)b^{2} \pt_2 \psi_{b}^{1} + \phi_{9} b^{2} \pt_3 \psi_{b}^{1} \\
&& + \phi_{10} b^{3} \pt_1 \psi_{b}^{1} + \phi_{9} b^{3} \pt_2 \psi_{b}^{1} + (\phi_7+\phi_1)b^{3} \pt_3 \psi_{b}^{1} 
\dd x \dd v\\
&=& \int_{\Omega^c} 2 b^{1} \pt_{1}\pt_{1} \varphi_{b}^{1} + b^{1} \pt_{2}\pt_{2} \varphi_{b}^{1} + b^{1} \pt_{3}\pt_{3} \varphi_{b}^{1} \\
&& + b^{2} \pt_{1}\pt_{2} \varphi_{b}^{1} - b^{2} \pt_{2}\pt_{1} \varphi_{b}^{1}
\dd x + 0 \\
&& + b^{3} \pt_{1}\pt_{3} \varphi_{b}^{1} + 0 + b^{3} \pt_{3}\pt_{1} \varphi_{b}^{1} \dd x \\
&= &  \int_{\Omega^c} b^{1}(\Delta + \pt_{1}^2) \varphi_{b}^{1} \dd x.
\end{eqnarray*}

Define $b_{{1}}$ as
\begin{equation} \label{btildedefn}
b_{{1}} := b-b_{{2}}-b_{{3}} = \left\{
\begin{aligned}
b_{{2}}-b_{{3}} &\text{ in } \Omega \\
b-b_{{2}}-b_{{3}} &\text{ in } \Omega^c    
\end{aligned} \right.
\end{equation} 
Then, \eqref{integbypartb1} becomes
\begin{eqnarray} \label{tildebestimate}
&&\int_{\Omega^c \times \R^3} \e \ip \pt_t f  \psi_{b}^{1} \dd x \dd v -\int_{\R^3 \times \R^3} b_{1}^{1} (\Delta+\pt_1^2) \varphi_{b}^{1} \dd x \dd v \\ \nonumber
&& + \e^{-1} \int_{\Omega^c \times \R^3} \psi_{b}^{1} L_{ } f \dd x \dd v\\
\nonumber
&&= \int_{\Omega^c \times \R^3}  \ip  g \psi_{b}^{1} \dd x \dd v.
\end{eqnarray}

By Lemma \ref{65estimate}, there is a unique $\varphi_{b}^{1}$ is a solution of
\begin{align*}
- (\Delta+\pt_1^2) \varphi_{b}^{1} = (b_{1}^{1})^5 \text{ in } \Omega^c , \quad \varphi_c = 0 \text{ on } \pt \Omega,
\end{align*}
and satisfies
\begin{align*}
\normm{\varphi_{b}^{1}}_{\dot{H}^{1}(\Omega^c)} \le \normm{(b_{1}^{1})^5}_{L^{\frac{6}{5}}(\Omega^c)} = \normm{b_{1}^{1}}_{L^{6}(\Omega^c)}^5.
\end{align*}
Since, $C_{c}^{\infty}(\R^3)$ is dense in $\dot{W}^{2,6/5}(\R^3)$ there exist $\varphi_{b,1,N} \in C_{c}^{\infty}(\Omega^c)$ that $$\normm{- (\Delta+\pt_1^2) \varphi_{b,1,N}-(b_{1}^{1})^5}_{L^{\frac{6}{5}}} <\frac{1}{N}\normm{(b_{1}^{1})^5}_{L^{\frac{6}{5}}}.$$
Putting $\varphi_{b,1,N}$ into \eqref{tildebestimate}, we obtain:
\begin{eqnarray*}
-\int_{\R^3 \times \R^3} {b_{1}^{1}} (\Delta+\pt_1^2) \varphi_{b,1,N} \dd x \dd v &=& -\int_{\R^3 \times \R^3} {b_{1}^{1}} (\Delta+\pt_1^2) \varphi_{b,1,N} \dd x \dd v \\
&=& -\int_{\R^3 \times \R^3} {b_{1}^{1}} ((\Delta+\pt_1^2) \varphi_{b,1,N}-(b_{1}^{1})^5) \dd x \dd v + \int_{\R^3 \times \R^3} (b_{1}^{1})^6 \dd x \dd v \\
&>& \normm{b_{1}^{1}}_{L^6}^6 - \frac{1}{N}\normm{b_{1}^{1}}_{L^6} \normm{(b_{1}^{1})^5}_{L^{\frac{6}{5}}} \\
&=&(1-\frac{1}{N}) \normm{b_{1}^{1}}_{L^6}^6.
\end{eqnarray*}
Moreover,
\begin{eqnarray*}
\int_{\Omega^c \times \R^3} \e \pt_t f  \psi_{b,1,N} \dd x \dd v &=& \int_{\Omega^c \times \R^3} \e \ip \pt_t f  \psi_{b,1,N} \dd x \dd v \\
& \le & \e \normm{\psi_{b,1,N}}_{2} \normm{(I-\P )\pt_t f}_{2} \\
& \le & \e \normm{\nabla \varphi_{b,1,N}}_{2} \normm{\ip \pt_t f}_{2} \\
& \lesssim & \e \normm{b_{1}^{1}}_{6}^5 \normm{\ip \pt_t f}_{2}.
\end{eqnarray*}
Similarly,
\begin{eqnarray*}
\e^{-1} \int_{\Omega^c \times \R^3} \psi_{b,1,N} L_{ } f \dd x \dd v &\lesssim&  \e^{-1} \normm{b_{1}^{1}}_{6}^5 \normm{\ip f}_{2},\\
\int_{\Omega^c \times \R^3} \psi_{b,1,N} g \dd x \dd v &\lesssim&  \normm{b_{1}^{1}}_{6}^5 \normm{\ip g}_{2}.
\end{eqnarray*}
Thus, we can get
\begin{align}\label{b1l6estimate}
\normm{b_{1}^{1}}_{L^6} \lesssim \e^{-1} \normm{\ip f}_{\nu} +\e\normm{\pt_t \ip f } + \normm{g \nu^{-1/2}}_{2}.
\end{align}
\newline
\textbf{Step 3)} Splitting of $a$.

To get the bound of $a$ choose a test function $\psi_a$ as:
\begin{align*}
\psi_a =& \left(\frac{\phi_{2}}{\normm{\phi_{2}}^2}+2\frac{\phi_{11}}{\normm{\phi_{11}}^2}\right) (v_{ }) \pt_1 \varphi_a + \left(\frac{\phi_{3}}{\normm{\phi_{3}}^2}+2\frac{\phi_{12}}{\normm{\phi_{12}}^2}\right) (v_{ }) \pt_2 \varphi_a \\
& + \left(\frac{\phi_{4}}{\normm{\phi_{4}}^2}+2\frac{\phi_{13}}{\normm{\phi_{13}}^2}\right) (v_{ }) \pt_3 \varphi_a.
\end{align*}
$\varphi_a$ is to be determined.
$\psi_a \in \ker L_{ }$ because of the definition of $\psi_a$. Therefore,
\begin{eqnarray} \label{integbyparta}
&&\int_{\Omega^c \times \R^3} \e \pt_t f  \psi_a \dd x \dd v + \int_{\g_+} f \psi_a \dd \g \\ \nonumber
&&- \int_{\Omega^c \times \R^3} f v \cdot \nabla \psi_a \dd x \dd v + \e^{-1} \int_{\Omega^c \times \R^3} \psi_a L_{ } f \dd x \dd v \\ \nonumber
&& = \int_{\Omega^c \times \R^3} g \psi_a \dd x \dd v + \int_{\g_{-}} (P_{\g}^{ } f + \e^{\frac{1}{2}}r)\psi_a \dd \g.
\end{eqnarray}
\newline
\textbf{STEP 3-1)} Construction of $a_{2}$ and $L^2$ estimate of ${a_{2}}$.

According to the Lemma \ref{ellipticlemma} there exist $a_{2}$ satisfies
\begin{eqnarray} \label{abardefn}
-C\int_{\Omega^c} a_{2} \Delta \varphi \dd x &=& \int_{\Omega^c \times \R^3} \ip f \sum_{i=1}^{3} \sum_{j=1}^{3} v_j \phi_{1+i} \pt_i \pt_j \varphi \dd x \dd v,
\end{eqnarray}
for all $\varphi \in C_c^{\infty} (\R^3) \subset \dot{H}^{2}(\R^3)$, where $C$ is a constant defined in {STEP 3-3)}, and the following inequality holds
\begin{align}\label{a2l2estimate}
\normm{a_{2}}_{L^2} \lesssim \normm{\ip f}_{L^2}.
\end{align}
\newline
\textbf{STEP 3-2)} Construction of $a_{3}$ and $L^3$ estimate of ${a_{3}}$.

According to the Lemma \ref{ellipticlemma} there exist $a_{3}$ satisfies
\begin{eqnarray} \label{agammadefn}
-C\int_{\Omega^c} a_{3} \Delta \varphi \dd x &=& \int_{\g_{-}} ((P_{\g}^{ }-1) f + \e^{\frac{1}{2}}r)\sum_{i=1}^{3} \sum_{j=1}^{3} v_j \phi_{1+i} \pt_i \pt_j \varphi \dd \g,
\end{eqnarray}
for all $\varphi \in C_c^{\infty} (\R^3) \subset \dot{W}^{2,p}(\R^3)$, where $C$ is a constant defined in {STEP 3-3)}, and the following inequality holds
\begin{align} \label{a3l3estimate}
\normm{a_{2}}_{L^3} \lesssim \norm{(1-P_{\g}^{ })f}_{2,+} + \e^{1/2} \norm{r} .
\end{align}
\newline
\textbf{STEP 3-3)} Construction of $a_{2}$ and $L^6$ estimate of ${a_{1}}$.

By \eqref{abardefn}, \eqref{integbyparta} takes the following form:\begin{eqnarray*}
&&\int_{\Omega^c \times \R^3} \e \pt_t f  \psi_a \dd x \dd v  - \int_{\Omega^c \times \R^3} \P f v \cdot \nabla \psi_a \dd x \dd v + \e^{-1} \int_{\Omega^c \times \R^3} \psi_a L_{ } f \dd x \dd v \\
&& = \int_{\Omega^c \times \R^3} g \psi_a \dd x \dd v -c\int_{\R^3} a_{2} \Delta \varphi_a \dd x -c\int_{\R^3} a_{3} \Delta \varphi_a \dd x,
\end{eqnarray*}
if $\varphi_a \in C_c^{\infty}(\R^3)$.

Note that
\begin{align*}
\int_{\Omega^c \times \R^3} \P f v \cdot \nabla \psi_{a} \dd x \dd v =& \int_{\Omega^c \times \R^3} a \sqrt{\mu_{ }} v_{ } \cdot \nabla \psi_{a} \dd x \dd v \\
&+ \int_{\Omega^c \times \R^3} b \cdot v_{ } \sqrt{\mu_{ }} v_{ } \cdot \nabla \psi_{a} \dd x \dd \\
&+\int_{\Omega^c \times \R^3} c \frac{\norm{v_{ }}^2-3}{2} \sqrt{\mu_{ }} v_{ } \cdot \nabla \psi_{a} \dd x \dd v.
\end{align*}
By the definition of $\psi_a$, corollary \ref{basestructure} and direct computation,
\begin{align*}
\int_{\Omega^c \times \R^3} a \sqrt{\mu_{ }} v_{ } \cdot \nabla \psi_c \dd x \dd v&=C\int_{\Omega^c} a \Delta \varphi_a \dd x\\
\int_{\Omega^c \times \R^3} b \cdot v_{ } \sqrt{\mu_{ }} v_{ } \cdot \nabla \psi_a \dd x \dd v&=0\\
\int_{\Omega^c \times \R^3} c \frac{\norm{v_{ }}^2-3}{2} \sqrt{\mu_{ }} v_{ } \cdot \nabla \psi_a \dd x \dd v&= 0,
\end{align*}
where $C$ is a fixed constant. 

Define $a_{1}$ as
\begin{equation} \label{atildedefn}
a_{1} := a-a_{2}-a_{3} = \left\{
\begin{aligned}
a_{2}-a_{3} &\text{ in } \Omega \\
a-a_{2}-a_{3} &\text{ in } \Omega^c    
\end{aligned} \right.
\end{equation} 
Then, \eqref{integbyparta} become
\begin{eqnarray} \label{tildeaestimate}
&&\int_{\Omega^c \times \R^3} \e  \pt_t f  \psi_a \dd x \dd v  -5\int_{\R^3 \times \R^3} c_{1} \Delta \varphi_a \dd x \dd v + \e^{-1} \int_{\Omega^c \times \R^3} \psi_a L_{ } f \dd x \dd v\\
\nonumber
&&= \int_{\Omega^c \times \R^3}  \ip  g \psi_a \dd x \dd v.
\end{eqnarray}
By Lemma \ref{65estimate}, there exists a unique $\varphi_a$ that is a solution of
\begin{align*}
- \Delta \varphi_a = a_{1}^5 \text{ in } \R^3,
\end{align*}
and satisfies
\begin{eqnarray*}
\normm{\varphi_a}_{\dot{H}^{1}(\R^3)} \le \normm{a_{1}^{5}}_{L^{\frac{6}{5}}(\R^3)} = \normm{a_{1}}_{L^{6}(\R^3)}^5.
\end{eqnarray*}
Since, $C_{c}^{\infty}(\R^3)$ is dense in $\dot{W}^{2,6/5}(\R^3)$ there exist
$\varphi_{a,N} \in C_{c}^{\infty}(\R^3)$ that $$\normm{- \Delta \varphi_{a,N}-a_{1}^5}_{L^{\frac{6}{5}}} <\frac{1}{N}\normm{a_{1}^5}_{L^{\frac{6}{5}}}.$$
Putting $\varphi_{a,N}$ into \eqref{tildeaestimate}, we obtain:
\begin{eqnarray*}
-\int_{\R^3 \times \R^3} a_{1} \Delta \varphi_{a,N} \dd x \dd v &=& -\int_{\R^3 \times \R^3} a_{1} \Delta \varphi_{a,N} \dd x \dd v \\
&=& -\int_{\R^3 \times \R^3} a_{1} (\Delta \varphi_{c,N}-a_{1}^5) \dd x \dd v + \int_{\R^3 \times \R^3} a_{1}^6 \dd x \dd v \\
&>& \normm{a_{1}}_{L^6}^6 - \frac{1}{N}\normm{a_{1}}_{L^6} \normm{a_{1}^5}_{L^{\frac{6}{5}}} \\
&=&(1-\frac{1}{N}) \normm{a_{1}}_{L^6}^6.
\end{eqnarray*}
Moreover,
\begin{eqnarray*}
\int_{\Omega^c \times \R^3} \e  \pt_t f  \psi_{a,N} \dd x \dd v
& \le & \e \normm{\psi_a}_{2} \normm{ \pt_t f}_{2} \\
& \le & \e \normm{\nabla \varphi_{a,N}}_{2} \normm{ \pt_t f}_{2} \\
& \lesssim & \e \normm{a_{1}}_{6}^5 \normm{ \pt_t f}_{2}.
\end{eqnarray*}
Similarly,
\begin{eqnarray*}
\e^{-1} \int_{\Omega^c \times \R^3} \psi_{a,N} L_{ } f \dd x \dd v &\lesssim&  \e^{-1} \normm{a_{1}}_{6}^5 \normm{\ip  f}_{2}\\
\int_{\Omega^c \times \R^3} \psi_{a,N} g \dd x \dd v &\lesssim&  \normm{a_{1}}_{6}^5 \normm{  g}_{2}.
\end{eqnarray*}
Thus, we can get
\begin{align} \label{a1l6estimate}
\normm{a_{1}}_{L^6} \lesssim \e^{-1} \normm{\ip f}_{\nu} +\e\normm{\pt_t f } + \normm{g \nu^{-1/2}}_{2}.
\end{align}

Combining \eqref{a1l6estimate}, \eqref{a2l2estimate}, \eqref{a3l3estimate}, \eqref{b1l6estimate}, \eqref{b2l2estimate}, \eqref{b3l3estimate}, \eqref{c1l6estimate}, \eqref{c2l2estimate} and \eqref{c3l3estimate} we can conclude the theorem \ref{l2l6estimate}.
\end{proof}

\section{Nonlinear Estimate} \label{sec6}
The main purpose of this section is to obtain the bound of the nonlinear collision term  $\G$. Proving the following proposition is the goal of this section.

\begin{proposition} \label{nonlinearestimate}
Let $g_i(t,x,v)$, $i=1,2$ is a smooth functions. Then,
\begin{align*}
\normm{\nu^{-1/2} \G(g_1,g_2)}_{L^2_x L^2_{v}} \lesssim& \normm{\ip g_1}_{\nu} \normm{\omega_{\b} g_2}_{L^\infty}  + \normm{\ip g_2}_{\nu} \normm{\omega_{\b} g_1}_{L^\infty}\\
& +\normm{{\P}_{1}' g_1}_{L^6} \normm{\P g_2}_{L^2}^{2/3} \normm{ \omega_{\b} g_2}_{L^\infty}^{1/3}+\normm{\P_{2} g_1}_{L^2} \normm{\omega_{\b} g_2}_{L^\infty}\\
& +\normm{{\P}_{3} g_1}_{L^3} \normm{\P g_2}_{L^2}^{1/3} \normm{ \omega_{\b} g_2}_{L^\infty}^{2/3},\\
\normm{\nu^{-1/2} \G(g_1,g_2)}_{L^2_x L^2_{v}} \lesssim& \normm{\ip g_1}_{\nu} \normm{\omega_{\b} g_2}_{L^\infty}  + \normm{\ip g_2}_{\nu} \normm{\omega_{\b} g_1}_{L^\infty}\\
& +\normm{{\P}_{1}' g_2}_{L^6} \normm{\P g_1}_{L^2}^{2/3} \normm{ \omega_{\b} g_1}_{L^\infty}^{1/3}+\normm{\P_{2} g_2}_{L^2} \normm{g_1}_{L^\infty} \\
& +\normm{{\P}_{3} g_2}_{L^3} \normm{\P g_1}_{L^2}^{1/3} \normm{ \omega_{\b} g_1}_{L^\infty}^{2/3},\\
\normm{\nu^{-1/2} \G(g_1,g_2)}_{L^\infty} \lesssim & \normm{\omega_{\b} g_1}_{L^\infty} \normm{\omega_{\b} g_2}_{L^\infty},
\end{align*}
where $\omega_{\b}(v) = exp(\b \norm{v}^2) $.
\end{proposition}

\begin{lemma} \label{lemma52}
Let $g_i(t,x,v)$, $i=1,2,3$ is a smooth functions. Then,
\begin{eqnarray*}
&&\int_{\R^3}\G(g_1,g_2)g_3 dv  \\
&&\lesssim \left[\int_{\R^3} \nu g_1^2 dv \right]^{1/2}\left[\int_{\R^3} g_2^2 dv \right]^{1/2} \left[\int_{\R^3} \nu g_3^2 dv \right]^{1/2}\\
&& + \left[\int_{\R^3}  g_1^2 dv \right]^{1/2} \left[\int_{\R^3} \nu g_2^2 dv \right]^{1/2} \left[\int_{\R^3} \nu g_3^2 dv \right]^{1/2}.
\end{eqnarray*}
\end{lemma}
\begin{proof}
This is the Lemma 2.3 in  \cite{Guo2002}.
\end{proof}

\begin{corollary} \label{cor53}
Let $\omega_{\b}(v) = exp(\b \norm{v}^2) $ then,
$$
\normm{\nu^{-1/2} \G(g_1,g_2)}_{2} \lesssim \normm{g_1}_{\nu} \normm{\omega_{\b} g_2}_{\infty}
$$
$$
\normm{\nu^{-1/2} \G(g_1,g_2)}_{2} \lesssim  \normm{\omega_{\b} g_1}_{\infty} \normm{g_2}_{\nu}
$$
for all $\b>0$.
\end{corollary}
\begin{proof}
Let us put $g_3 = \nu^{-1} \G(g_1,g_2)$ to the lemma \ref{lemma52} then we can get,
\begin{eqnarray*}
&&\int_{\R^3}\nu^{-1} \G(g_1,g_2)^2 dv  \\
&&\lesssim \left[\int_{\R^3} \nu g_1^2 dv \right]^{1/2}\left[\int_{\R^3} g_2^2 dv \right]^{1/2} \left[\int_{\R^3}  \nu^{-1} \G(g_1,g_2)^2 dv \right]^{1/2}\\
&& + \left[\int_{\R^3}  g_1^2 dv \right]^{1/2} \left[\int_{\R^3} \nu g_2^2 dv \right]^{1/2} \left[\int_{\R^3} \nu^{-1} \G(g_1,g_2)^2 dv \right]^{1/2}.
\end{eqnarray*}
In addition,
\begin{eqnarray*}
\int_{\R^3} \nu f^2 dv &\le& \int_{\R^3} \omega_{\b}^{-2} \nu (\omega_{\b} f)^2 dv \\
&\le& \norm{\omega_{\b} f}^2 \int_{\R^3} \omega_{\b}^{-2} \nu dv \\
&\le&  C(\b) \norm{\omega_{\b} f}^2.
\end{eqnarray*}
So,
\begin{eqnarray*}
\normm{\nu^{-1/2} \G(g_1,g_2)}_{2}^{2} &=& \iint\nu^{-1} \G(g_1,g_2)^2 dv dx\\
&\le& \int\left[ \int_{\R^3} \nu g_1^2 dv \int_{\R^3} g_2^2 dv + \int_{\R^3}  g_1^2 dv \int_{\R^3} \nu g_2^2 dv \right] dx\\
&\lesssim& \min\left\{ \normm{g_1}_{\nu} \normm{\omega_{\b} g_2}_{\infty},  \normm{\omega_{\b} g_1}_{\infty} \normm{g_2}_{\nu} \right\}^{2}.
\end{eqnarray*}
\end{proof}

\begin{proof}(proof of Proposition \ref{nonlinearestimate})

First,
\begin{align*}
\normm{\nu^{-1/2} \G(g_1,g_2)}_{L^\infty_{x,v}} \lesssim & \normm{\omega_{\b} g_1}_{L^\infty} \normm{\omega_{\b} g_2}_{L^\infty}
\end{align*}
is straightforward by the definition of $\G$.
To get $L^2$ estimate, we can decompose $\G(g_1,g_2)$ into
\begin{align*}
\G(g_1,g_2) =& \G(\P g_1, \P g_2) + \G(\ip g_1, \P g_2) \\
&+ \G(\P g_1, \ip g_2) + \G(\ip g_1, \ip g_2)
.
\end{align*}
According to the corollary \ref{cor53}
\begin{align*}
\normm{\nu^{-1/2} \G(\ip g_1, \P g_2)}_{L^2} &\lesssim \normm{\ip g_1}_{\nu} \normm{\omega_{\b} g_2}_{L^\infty}, \\
\normm{\nu^{-1/2} \G(\P g_1, \ip g_2)}_{L^2} &\lesssim \normm{\ip g_2}_{\nu} \normm{\omega_{\b} g_1}_{L^\infty}, \\
\normm{\nu^{-1/2} \G(\ip g_1, \ip g_2)}_{L^2} &\lesssim \normm{\ip g_2}_{\nu} \normm{\omega_{\b} g_1}_{L^\infty}.
\end{align*}
Thus, we only need to control $\G(\P g_1, \P g_2)$. Let
\begin{align*}
\P g_1 =& \sqrt{\mu_{ }} [a_1 + b_1 \cdot v_{ } + c_1 \frac{1}{2} (\norm{v_{ }}^2-3)]\\
\P g_2 =& \sqrt{\mu_{ }} [a_2 + b_2 \cdot v_{ } + c_2 \frac{1}{2} (\norm{v_{ }}^2-3)].
\end{align*}
\begin{align*}
\G(\sqrt{\mu_{ }} a_1 ,\P g_2) &= \frac{1}{\sqrt{\mu_{  }}} Q({\mu}_{ } a_1,\sqrt{\mu_{ }}\P g_2) \\
&= \iint_{\R^3 \times \S^2} B(\omega,\norm{v-u}) {\mu_{ }}(u+v) a_f \left(\P g_2(u')- \P g_2(u) \right) \dd \omega \dd u \\
&= 0.
\end{align*}
Similarly, $\G(\P g_1, \sqrt{\mu_{ }} a_2) = 0$. Thus,
\begin{eqnarray*}
&&\normm{\nu^{-1/2} \G(\P g_1 ,\P g_2)}_{L^2} \\
&=& \normm{\nu^{-1/2} \G(\sqrt{\mu_{ }} [a_1 + b_1 \cdot v_{ } + c_1 \frac{1}{2} (\norm{v_{ }}^2-3)], \P g_2)}_{L^2} \\
&=& \normm{\nu^{-1/2} \G(\sqrt{\mu_{ }} [b_1 \cdot v_{ } + c_1 \frac{1}{2} (\norm{v_{ }}^2-3)], \P g_2) }_{L^2}\\
&\lesssim& \normm{{\P}_{1}' g_1}_{6} \normm{\P g_2}_{3} +\normm{\P_{2} g_1}_{2} \normm{\P g_2}_{\infty} + \normm{{\P}_{3} g_1}_{3} \normm{\P g_2}_{6}  \\
&\lesssim& \normm{{\P}_{1}' g_1}_{6} \normm{\P g_2}_{2}^{2/3} \normm{ g_2}_{\infty}^{1/3} + \normm{\P_{2} g_1}_{2} \normm{g_2}_{\infty} + \normm{{\P}_{3} g_1}_{3} \normm{\P g_2}_{2}^{1/3} \normm{ g_2}_{\infty}^{2/3}.
\end{eqnarray*}
\end{proof}

\section{Results} \label{sec7}
The aim of this section is to prove the main result (Theorem \ref{mainthm}) of this paper.

\begin{definition}
Define energy and dissipation as
\begin{eqnarray}
\mathscr{E}[f](t) &:=& \sup_{0 \le s \le t}\normm{f}_{L^2_{x,v}}^2 + \sup_{0 \le s \le t}\normm{\pt_t f}_{L^2_{x,v}}^2 \\
\mathscr{D}[f](t) &:=& \frac{1}{\e^2} \int_{0}^{t}\normm{\ip f(s)}_{\nu}^2 \dd s + \frac{1}{\e^2} \int_{0}^{t}\normm{\ip \pt_t  f(s)}_{\nu}^2 \dd s\\
\nonumber && + \frac{1}{\e} \int_{0}^{t} \norm{(1-P_{\g})f(s)}_{2,+}^2 \dd s + \frac{1}{\e} \int_{0}^{t} \norm{(1-P_{\g})\pt_t f(s)}_{2,+}^2 \dd s\\ \nonumber
&& + \int_{0}^{t}\normm{{\P}_{1}' f(s)}_{L^6_{x,v}}^2 \dd s.
\end{eqnarray}
Define new norm $\inn{\inn{f}}$ as 
\begin{eqnarray}
\inn{\inn{f}} &:=& \mathscr{E}[f](\infty)^{1/2} + \mathscr{D}[f](\infty)^{1/2} \\
&& \nonumber + \e^{1/2}\normm{\omega_{\b} f}_{L^{\infty}_{t}L^{\infty}_{x,v}} + \e^{3/2}\normm{\omega_{\b}\pt_t f}_{L^{\infty}_{t}L^{\infty}_{x,v}}.
\end{eqnarray}
\end{definition}

\begin{proof}(proof of Theorem \ref{mainthm})
\newline
Consider ${R}^{l}(t,x,v)$ solving for $l \in \mathbb{N}$,
\begin{align*}
\e \pt_t {R}^{l+1} + v \cdot \nabla_x {R}^{l+1} + \e^{-1}L({R}^{l+1}) =& h + \tilde{L}({R}^{l+1}) + \e^{1/2} \G({R}^{l},{R}^{l}), \\
{R}^{l+1} |_{\gamma_{-}} =& \mathscr{P}^{w}_{\gamma}({R}^{l+1})+  r.
\end{align*}
Here we set, ${R}^{0} : = 0$.
\newline
\textbf{Step 1)}We will prove the following relation:
\begin{align*}
\inn{\inn{{R}^{l+1}}} &\lesssim
o(1) + o(1)\inn{\inn{{R}^{l+1}}} + \inn{\inn{{R}^{l}}}^2.
\end{align*}
According to the Proposition \ref{energyestimate} we can get,
\begin{align*}
&\normm{R^{l+1}}^2_{L^2_{x,v}}(t) - \normm{R^{l+1}}^2_{L^2_{x,v}}(0) \\
&+ \normm{\e^{-1}\ip  R^{l+1}}_{L^2_t {\nu}}^2 + \norm{\e^{-1/2}(1-P_{\gamma})R^{l+1}}_{L^{2}_{t}2,+}^2 \\
& \lesssim 
\norm{\e^{-1/2}r}_{L^{2}_{t}2,+}^2 + \e^{-1}\int_0^t \iint_{\Omega^c \times \R^3} \left(h(s) +  \tilde{L}R^{l+1}(s) +  \e^{1/2} \G(R^l(s),R^l(s)) \right) R^{l+1}(s) \dd x \dd v \dd s,
\end{align*}
and
\begin{align*}
&\normm{\pt_t R^{l+1}}^2_{L^2_{x,v}}(t) - \normm{\pt_t R^{l+1}}^2_{L^2_{x,v}}(0) \\
&+ \normm{\e^{-1}\ip  \pt_t R^{l+1}}_{L^2_t{\nu}}^2 + \norm{\e^{-1/2}(1-P_{\gamma})\pt_t R^{l+1}}_{L^{2}_{t}2,+}^2 \\
& \lesssim 
\norm{\e^{-1/2}\pt_t r}_{L^{2}_{t}2,+}^2 + \e^{-1}\int_0^t \iint_{\Omega^c \times \R^3} (\pt_t h(s) +  \pt_t \tilde{L}R^{l+1}(s) \\
& +  \e^{1/2} \G(\pt_t R^l(s),R^l(s)) + \e^{1/2} \G(R^l(s), \pt_t R^l(s)) ) \pt_t R^{l+1}(s) \dd x \dd v \dd s.
\end{align*}
By Proposition \ref{hdecompose},
\begin{align*}
&\e^{-1} \int_{0}^t \iint_{\Omega^c \times \R^3} h(s) R^{l+1}(s) \dd x \dd v \dd s \\
&= \e^{-1} \int_{0}^t \iint_{\Omega^c \times \R^3} (\e^{1/2} h_{1/2} + \e^{3/2} (h_{3/2}+v \cdot \pt_t u_2 \sqrt{\mu})) \ip R^{l+1}(s) \dd x \dd v \dd s \\
&- \e^{-1} \int_{0}^t \iint_{\Omega^c \times \R^3} \e^{3/2}v \cdot \pt_t u_2 \sqrt{\mu} R^{l+1}(s) \dd x \dd v \dd s \\
& \lesssim \int_{0}^t \normm{\e^{1/2} h_{1/2} + \e^{3/2} (h_{3/2}+ v \cdot \pt_t u_2 \sqrt{\mu})}_{L^2_{x,v}} \normm{\e^{-1}\ip  R^{l+1}}_{{\nu}} \dd s\\
&+ \e^{-1} \int_{0}^t \iint_{\Omega^c \times \R^3} \norm{\e^{3/2}v \cdot \pt_t u_2 \sqrt{\mu} (\P_{1}'R^{l+1} + \P_{2} R^{l+1} + \P_{3} R^{l+1})} \dd x \dd v \dd s \\
& \lesssim \normm{\e^{1/2} h_{1/2} + \e^{3/2} (h_{3/2}+ v \cdot \pt_t u_2 \sqrt{\mu})}^2_{L^2_t L^2_{x,v}} + o(1) \normm{\e^{-1}\ip  R^{l+1}}_{L^2_t{\nu}} ^2\\
&+ \e^{1/2}\normm{v \cdot \pt_t u_2 \sqrt{\mu}}_{L^{2}_{t}L^{3}_{x,v}} \normm{\P_{1}' R^{l+1}}_{L^{2}_{t}L^{6}_{x,v}} + \e^{1/2}\normm{v \cdot \pt_t u_2 \sqrt{\mu}}_{L^{2}_{t}L^{\infty}_{x,v}} \normm{\P_{2} R^{l+1}}_{L^{2}_{t}L^{2}_{x,v}} \\
&+ \e^{1/2}\normm{v \cdot \pt_t u_2 \sqrt{\mu}}_{L^{2}_{t}L^{6}_{x,v}} \normm{\P_{3} R^{l+1}}_{L^{2}_{t}L^{3}_{x,v}}\\
&\lesssim o(1) + o(1) \inn{\inn{R^{l+1}}}^2,
\end{align*}
and
\begin{align*}
&\e^{-1} \int_{0}^t \iint_{\Omega^c \times \R^3} \pt_t h(s) \pt_t R^{l+1}(s) \dd x \dd v \dd s \\
&= \e^{-1} \int_{0}^t \iint_{\Omega^c \times \R^3} (\e^{1/2} \pt_t h_{1/2} + \e^{3/2} \pt_t (h_{3/2}+v \cdot \pt_t u_2 \sqrt{\mu})) \ip \pt_t R^{l+1}(s) \dd x \dd v \dd s \\
&- \e^{-1} \int_{0}^t \iint_{\Omega^c \times \R^3} \e^{3/2}v \cdot \pt_t^2 u_2 \sqrt{\mu} \pt_t R^{l+1}(s) \dd x \dd v \dd s \\
&= \e^{-1} \int_{0}^t \iint_{\Omega^c \times \R^3} (\e^{1/2} \pt_t h_{1/2} + \e^{3/2} \pt_t (h_{3/2}+v \cdot \pt_t u_2 \sqrt{\mu})) \ip \pt_t R^{l+1}(s) \dd x \dd v \dd s \\
&+ \e^{-1} \int_{0}^t \iint_{\Omega^c \times \R^3} \e^{3/2}v \cdot \pt_t^3 u_2 \sqrt{\mu} R^{l+1}(s) \dd x \dd v \dd s \\
&- \e^{-1} \iint_{\Omega^c \times \R^3} \e^{3/2}v \cdot \pt_t^2 u_2 \sqrt{\mu} R^{l+1}(s) \dd x \dd v |_{0}^{t} \\
& \lesssim \normm{\e^{1/2} \pt_t h_{1/2} + \e^{3/2} \pt_t h_{3/2}}_{L^2_t L^2_{x,v}} + o(1) \normm{\e^{-1}\ip  \pt_t R^{l+1}}_{L^2_t {\nu}} ^2\\
&+ \e^{1/2}\normm{v \cdot \pt_t^3 u_2 \sqrt{\mu}}_{L^{2}_{t}L^{3}_{x,v}} \normm{\P_{1}' R^{l+1}}_{L^{2}_{t}L^{6}_{x,v}} + \e^{1/2}\normm{v \cdot \pt_t^3 u_2 \sqrt{\mu}}_{L^{2}_{t}L^{\infty}_{x,v}} \normm{\P_{2} R^{l+1}}_{L^{2}_{t}L^{2}_{x,v}} \\
&+ \e^{1/2}\normm{v \cdot \pt_t^3 u_2 \sqrt{\mu}}_{L^{2}_{t}L^{6}_{x,v}} \normm{\P_{3} R^{l+1}}_{L^{2}_{t}L^{3}_{x,v}}\\
& + \e^{1/2}(\normm{\pt_t^2 u_{2}(t)}_{L^{\infty}_{x}}\normm{R^{l+1}(t)}_{L^{2}_{x,v}} +\normm{\pt_t^2 u_{2}(0)}_{L^{\infty}_{x}}\normm{R^{l+1}(0)}_{L^{2}_{x,v}}) \\
&\lesssim o(1) + o(1) \inn{\inn{R^{l+1}}}^2.
\end{align*}
By the definition of $\tilde{L}$ and corollary \ref{cor53}
\begin{align*}
&\e^{-1} \int_{0}^t \iint_{\Omega^c \times \R^3} \tilde{L}R^{l+1}(s) R^{l+1}(s) \dd x \dd v \dd s \\
&= \e^{-1} \int_{0}^t \iint_{\Omega^c \times \R^3} \left(\G(f_1 + \e f_2, R^{l+1}) + \G( R^{l+1}, f_1 + \e f_2) \right)\ip R^{l+1}(s) \dd x \dd v \dd s \\
& \lesssim \normm{f_1 + \e f_2}_{L^2_t L^\infty_{x}}^2 \normm{R^{l+1}}_{L^\infty_t L^2_{x,v}}^2 + o(1) \normm{\e^{-1}\ip  R^{l+1}}_{L^2_t {\nu}}^2 \\
&\lesssim o(1) \inn{\inn{R^{l+1}}}^2,
\end{align*}
and
\begin{align*}
&\e^{-1} \int_{0}^t \iint_{\Omega^c \times \R^3} \pt_t \tilde{L}R^{l+1}(s) \pt_tR^{l+1}(s) \dd x \dd v \dd s \\
&= \e^{-1} \int_{0}^t \iint_{\Omega^c \times \R^3} \left(\G(\pt_t f_1 + \e \pt_t f_2, R^{l+1}) + \G( R^{l+1}, \pt_t f_1 + \e \pt_t f_2) \right)\ip \pt_t R^{l+1}(s) \dd x \dd v \dd s \\
& + \e^{-1} \int_{0}^t \iint_{\Omega^c \times \R^3} \left(\G(f_1 + \e  f_2, \pt_t R^{l+1}) + \G(\pt_t R^{l+1}, f_1 + \e f_2) \right)\ip \pt_t R^{l+1}(s) \dd x \dd v \dd s \\
& \lesssim \normm{\pt_t f_1 + \e \pt_t f_2}_{L^2_t L^\infty_{x}}^2 \normm{R^{l+1}}_{L^\infty_t L^2_{x,v}}^2 + \normm{f_1 + \e f_2}_{L^2_t L^\infty_{x}}^2 \normm{\pt_t R^{l+1}}_{L^\infty_t L^2_{x,v}}^2 \\
&+ o(1) \normm{\e^{-1}\ip  \pt_t R^{l+1}}_{L^2_t {\nu}}^2 \\
&\lesssim o(1) \inn{\inn{R^{l+1}}}^2.
\end{align*}
By Proposition \ref{nonlinearestimate} 
\begin{align*}
&\e^{-1/2} \int_0^t \iint_{\Omega^c \times \R^3} \G(R^l(s),R^l(s))  R^{l+1}(s) \dd x \dd v \dd s \\
&\lesssim \e^{1/2} \normm{\G(R^l(s),R^l(s))}_{L^2_{t,x,v}} \normm{\e^{-1}\ip  R^{l+1}}_{L^2_t{\nu}}\\
&\lesssim \e \normm{\G(R^l(s),R^l(s))}_{L^2_{t,x,v}}^2 +
o(1) \normm{\e^{-1}\ip  R^{l+1}}_{L^2_t{\nu}}^2\\
& \lesssim  \e^{2} \normm{\e^{-1}\ip R^l}_{L^2_t L^2_\nu}^2 \normm{\e^{1/2} \omega_{\b} R^l}_{L^\infty_{t,x,v}}^2\\
&+ \e^{2/3} \normm{{\P}_{1}' R^l}_{L^2_{t} L^6_{x,v}}^2 \normm{\P R^l}_{L^\infty_t L^2_x L^2_v}^{4/3} \normm{ \e^{1/2} R^l}_{L^\infty_{t,x,v}}^{2/3} \\
& + \e^{4/3} \normm{\e^{-1/2}{\P}_{3} R^l}_{L^2_{t} L^3_{x,v}}^2 \normm{\P R^l}_{L^\infty_t L^2_x L^2_v}^{2/3} \normm{ \e^{1/2} R^l}_{L^\infty_{t,x,v}}^{4/3} \\
& + \e^{2} \normm{\e^{-1} \P_{2} R^l}_{L^2_{t,x,v}}^2 \normm{\e^{1/2}R^l}_{L^\infty_{t,x,v}}^2 + o(1)\normm{\e^{-1}\ip  R^{l+1}}_{L^2_t{\nu}}^2 \\
& \lesssim o(1) \inn{\inn{R^{l+1}}}^2 + \inn{\inn{R^{l}}}^4
\end{align*}
and
\begin{align*}
&\e^{-1/2} \int_0^t \iint_{\Omega^c \times \R^3} \pt_t \G(R^l(s),R^l(s))  \pt_t R^{l+1}(s) \dd x \dd v \dd s \\
&\lesssim \e^{1/2} \normm{\pt_t \G(R^l(s),R^l(s))}_{L^2_{t,x,v}} \normm{\e^{-1}\ip  \pt_t R^{l+1}}_{L^2_t{\nu}}\\
&\lesssim \e \normm{\G(\pt_t R^l(s),R^l(s))}_{L^2_{t,x,v}}^2 + \e \normm{\G(R^l(s),\pt_t R^l(s))}_{L^2_{t,x,v}}^2 +
o(1) \normm{\e^{-1}\ip  R^{l+1}}_{L^2_t{\nu}}^2\\
& \lesssim  \e^2 \normm{\e^{-1}\ip \pt_t R^l}_{L^2_t L^2_\nu}^2 \normm{\e^{1/2} \omega_{\b} R^l}_{L^\infty_{t,x,v}}^2 + \normm{\e^{-1}\ip R^l}_{L^2_t L^2_\nu}^2 \normm{\e^{3/2} \omega_{\b} \pt_t R^l}_{L^\infty_{t,x,v}}^2\\
&+ \normm{{\P}_{1}' R^l}_{L^2_{t} L^6_{x,v}}^2 \normm{\P \pt_t R^l}_{L^\infty_t L^2_{x,v}}^{4/3} \normm{ \e^{3/2} \pt_t R^l}_{L^\infty_{t,x,v}}^{2/3} + \normm{\e^{-1} \P_{2} R^l}_{L^2_{t,x,v}}^2 \normm{\e^{3/2}\pt_t R^l}_{L^\infty_{t,x,v}}^2 \\
&+ \normm{\e^{-1/2}{\P}_{3} R^l}_{L^2_{t} L^3_{x,v}}^2 \normm{\P \pt_t R^l}_{L^\infty_t L^2_{x,v}}^{2/3} \normm{ \e^{3/2} \pt_t R^l}_{L^\infty_{t,x,v}}^{4/3} + o(1)\normm{\e^{-1}\ip  \pt_t R^{l+1}}_{L^2_t{\nu}}^2 \\
& \lesssim o(1) \inn{\inn{R^{l+1}}}^2 + \inn{\inn{R^{l}}}^4.
\end{align*}
Thus,
\begin{align*}
\mathscr{E}[R^{l+1}](\infty) + \mathscr{D}[R^{l+1}](\infty)  \lesssim o(1) + o(1)\inn{\inn{R^{l+1}}}^2 + \inn{\inn{R^{l}}}^4
\end{align*}

According to the Proposition \ref{linftyestimate},
\begin{align*}
\normm{\e^{1/2} \omega_\b R^{l+1}}_{L^{\infty}_{t},L^{\infty}_{x,v}} \lesssim&  \normm{\e^{1/2} \omega_\b R_0}_{L^2_{x,v}} + \e^{1/2}\sup_{0 \le s \le t}\normm{\omega_\b r(s)} + \e^{1/2} \inn{\inn{R^l}}^2\\
& +\sup_{0\le s\le t } \normm{ \P_{1} R^{l+1}(s)}_{L^6_{x,v}}  +  \e^{-1}\sup_{0\le s\le t } \normm{ \P_{2} R^{l+1}(s)}_{L^{2}_{x,v}}\\
&+\e^{-1/2}\sup_{0\le s\le t } \normm{ {\P}_{3} R^{l+1}(s)}_{L^3_{x,v}} +  \e^{-1}\sup_{0\le s\le t } \normm{ \ip R^{l+1}(s)}_{L^{2}_{x,v}} ,
\end{align*}
and
\begin{align*}
\normm{\e^{3/2} \omega_\b \pt_t R^{l+1}}_{L^{\infty}_{t},L^{\infty}_{x,v}} \lesssim&  \normm{\e^{3/2} \omega_\b \pt_t R_0}_{L^2_{x,v}} + \e^{3/2}\sup_{0 \le s \le t}\normm{ \omega_\b \pt_t r(s)} + \e^{3/2} \inn{\inn{R^{l}}}^2\\
& + \sup_{0\le s\le t } \normm{ \pt_t R^{l+1}(s)}_{L^{2}_{x,v}}.
\end{align*}
According to the Theorem \ref{l2l6estimate},
$\P_{1} f$, $\P_{2} f$ and $\P_{3} f$ can be bounded by
\begin{align*}
\normm{ \P_{1} R^{l+1}(s)}_{L^6_{x,v}} &\lesssim \e^{-1} \normm{\ip R^{l+1} }_{\nu} +\e\normm{\pt_t R^{l+1} }_{2} + \normm{\G (R^l,R^l) \nu^{-1/2}}_{2}, \\
\normm{\P_{2} R^{l+1}(s)}_{L^2_{x,v}} &\lesssim  \normm{\ip R^{l+1}}_{\nu}
\\
\normm{{\P}_{3} R^{l+1}(s)}_{L^3_{x,v}} &\lesssim  \norm{(1-P_{\g})R^{l+1}}_{2,+} + \e^{1/2} \norm{r}. 
\end{align*}
Since,
\begin{eqnarray*}
\sup_{0 \le s \le t} \e^{-2} \normm{\ip R^{l+1}}_{\nu}^2 &\le& \e^{-2} \normm{\ip R_0}^2_{\nu} + 2 \e^{-2} \int_{0}^{t} \normm{\ip R^{l+1}}_{\nu} \normm{\ip \pt_t R^{l+1}}_{\nu} \dd s\\
&\le& \e^{-2} \normm{\ip R_0}^2_{\nu} + \mathscr{D}[R^{l+1}](t),
\end{eqnarray*}
and similarly
\begin{eqnarray*}
\sup_{0 \le s \le t} \e^{-1} \norm{(1-P_{\g})) R^{l+1}}_{2,+}^2 &\le& \e^{-1} \norm{(1-P_{\g}) R_0}^2_{2,+} + 2 \int_{0}^{t} \norm{(1-P_{\g}) R^{l+1}}_{2,+} \norm{(1-P_{\g}) \pt_t R^{l+1}}_{2,+} \dd s\\
&\le& \e^{-1} \norm{(1-P_{\g}) R_0}^2_{2,+} + \mathscr{D}[R^{l+1}](t),
\end{eqnarray*}
and
\begin{eqnarray*}
\sup_{0 \le s \le t}  \normm{\G(R^l,R^l)}_{\nu}^2 &\le&  \normm{\G(R_0,R_0)}^2_{\nu} + 2  \int_{0}^{t} \normm{\G(R^l,R^l)}_{\nu} \normm{\pt_t \G(R,R)}_{\nu} \dd s\\
&\le& \normm{\G(R_0,R_0)_0}^2_{\nu} + \inn{\inn{R^l}}^4.
\end{eqnarray*}

Thus,
\begin{eqnarray*}
\inn{\inn{{R}^{l+1}}} &\lesssim& 
o(1) + o(1)\inn{\inn{{R}^{l+1}}} + \inn{\inn{{R}^{l}}}^2.
\end{eqnarray*}
\newline
\textbf{STEP 2)} We need to prove that the map $\Phi : R^{l} \to R^{l+1}$ is a contraction map via norm $\inn{\inn{\cdot}}$.
We know the following relation holds:
\begin{align*}
&\e \pt_t ({R}^{l+1} - {R}^{l}) + v \cdot \nabla_x ({R}^{l+1} - {R}^{l}) + \e^{-1}L({R}^{l+1} - {R}^{l})\\
&= \tilde{L}({R}^{l+1} - {R}^{l}) + \e^{1/2} (\G({R}^{l+1},{R}^{l+1}) - \G({R}^{l},{R}^{l})).
\end{align*}
\begin{align*}
({R}^{l+1} - {R}^{l}) |_{\gamma_{-}} = \mathscr{P}_{\gamma}({R}^{l+1} - {R}^{l})
\end{align*}
If we use same energy estimate as \textbf{STEP 1)}, then we can get the following bounds:
\begin{align*}
\inn{\inn{{R}^{l+1} - {R}^{l}}} \lesssim& 
\left(\normm{f_1+\e f_2 }_{L^{2}_{t}L^{\infty}_{x,v}} + \normm{\pt_t f_1+\e \pt_t f_2 }_{L^{2}_{t}L^{\infty}_{x,v}}\right) \inn{\inn{{R}^{l+1} - {R}^{l}}} \\
&+ (\inn{\inn{{R}^{l+1}}} +\inn{\inn{{R}^{l}}})(\inn{\inn{{R}^{l+1} - {R}^{l}}}).
\end{align*}
Which means the map $\Phi$ is a contraction map and it has a fixed point.
\end{proof}

\section*{Acknowledgement}
The author thanks Yan
Guo for his comments and suggestions. This research is funded in part
by a Kwanjeong foundation and NSF Grant DMS-2106650


\begin{thebibliography}{20}


\bibitem{Asano1983}
Asano, K., Ukai, S.: The Euler limit and the initial layer of the nonlinear Boltzmann equation. Hokkaido Math. J. 12, 303–324 (1983)


\bibitem{Bardos1991}
Bardos, C., Ukai, S.: The classical incompressible Navier–Stokes limit of the Boltzmann equation. Math. Mod. Meth. Appl. Sci. 1, 235 (1991)


\bibitem{Bardos1991-1}
Bardos, C., Golse, F., Levermore, D.: Fluid Dynamic Limits of Kinetic Equations I: Formal Derivations. J. Statistical Physics  63, 323–344 (1991)

\bibitem{Bardos1993}
Bardos, C., Golse, F., Levermore, D.: Fluid dynamical limits of kinetic equations, II: convergence proofs for the Boltzmann equation. Commun. Pure Appl. Math. 46, 667–753 (1993)



\bibitem{Borchers1991}
Borchers, W., Miyakawa, T.: Algebraic $L^2$ decay for Navier–Stokes flows in exterior domains, II.
Hiroshima Math. J. 21, 621–640 (1991)


\bibitem{Caflisch1980}
Caflisch, R.E.: The fluid dynamical limit of the nonlinear Boltzmann equation. Comm. Pure Appl. Math. 33, 651–666 (1980)

\bibitem{Cao2023}
Cao, Y., Jang, J., Kim, C.: Passage from the Boltzmann equation with diffuse boundary to the incompressible Euler equation with heat convection. J. Differential Equations 366, 565–644. (2023)

\bibitem{Cercignani1994}
Cercignani, C., Illner, R., Pulvirenti, M.: The Mathematical Theory of Dilute Gases. Springer, New
York (1994)

\bibitem{Diperna1989}
DiPerna, R.J., Lions, P.L.: On the Cauchy problem for Boltzmann equations: global existence and weak stability. Ann. Math. 130, 321–366 (1989)

\bibitem{Duan2008}
Duan, J.: On the Cauchy problem for the Boltzmann equation in the whole space: Global existence and uniform stability in $L^2_{\xi} (H^N_x)$. J. Differential Equations 244, 3204-3234 (2008)

\bibitem{Duan2012}
Duan, J., Yang, T., Zhao, H.
: The Vlasov-Poisson-Boltzmann system in the whole space: the hard potential case.
J. Differential Equations 252, 6356–6386
 (2012)

\bibitem{Esposito2013}
Esposito, R., Guo, Y., Kim, C., Marra, R.: Non-isothermal boundary in the Boltzmann theory and Fourier law. Commun. Math. Phys. 323, 177–239 (2013)


\bibitem{Esposito2018-apde}
Esposito, R., Guo, Y., Kim, C., Marra, R. : Stationary Solutions to the Boltzmann Equation in the Hydrodynamic Limit.
Ann. PDE (2018)

\bibitem{Esposito2018-cmp}
Esposito, R., Guo, Y., Marra, R. :
Hydrodynamic Limit of a Kinetic Gas Flow Past an Obstacle. Commun. Math. Phys. 364, 765–823 (2018)


\bibitem{Esposito2020}
Esposito, R., Guo, Y., Kim, C., Marra, R.: Diffusive limits of the Boltzmann equation in bounded domain. Ann. Appl. Math. 36, no. 2, 111–185. (2020)

\bibitem{Evans1998}
Evans, L.C.: Partial Differential Equations, Graduate Studies in Mathematics, vol. 19. American Mathematical Society, Providence (1998)


\bibitem{Galdi2011}
Galdi, G.P.:
An Introduction to the Mathematical Theory of the Navier-Stokes Equations Steady-State Problems. Springer (2011)




\bibitem{Glassey1996}
Glassey, R.: The Cauchy Problem in Kinetic Theory. SIAM (1996)


\bibitem{Gilbarg2001}
Gilbarg, D., Trudinger, N.: Elliptic Partial Differential Equations of Second Order, Springer (2001)


\bibitem{Golse2004}
Golse, F., Saint-Raymond, L.:The Navier–Stokes limit of theBoltzmann equation for bounded collision kernels. Invent. Math. 155, 81–161 (2004)

\bibitem{Guo2002}
Guo, Y.: The Vlasov-Poisson-Boltzmann system near Maxwellians, Comm. Pure Appl.Math. 55
, 1104–1135 (2002)

\bibitem{Guo2006}
Guo, Y. : Boltzmann Diffusive Limit Beyond the Navier-Stokes Approximation. Comm. Pure Appl. Math., 59, 0626–0687 (2006)


\bibitem{Guo2010-arma}
Guo, Y.: Decay and continuity of the Boltzmann equation in bounded domains. Arch. Ration. Mech. Anal. 197, 713–809 (2010)

\bibitem{Guo2010}
Guo, Y., Jang, J.: Global Hilbert expansion for the Vlasov–Poisson–Boltzmann system. Comm. Math. Phys. 299, 469–501 (2010)



\bibitem{Guo2010_2}
Guo, Y., Jang, J., Jiang, N.: Acoustic limit for the Boltzmann equation in optimal scaling. Comm. Pure Appl. Math. 63, 337–361 (2010)


\bibitem{Hilbert1916}
Hilbert, D.: Begründung der kinetischen Gastheorie. Math. Ann. 72, 331–407 (1916/17)

\bibitem{Hilbert_}
Hilbert, D.: Grundzügeiner allgemeinen Theorie der linearen Integralgleichungen. Chelsea, New York

\bibitem{Hilbert1900}
Hilbert, D.: Mathematical Problems. Göttinges Nachrichten, pp. 253–297 (1900)


\bibitem{Huang2010}
Huang, F., Wang, Y., Yang, T.: Hydrodynamic limit of the Boltzmann equation with contact discontinuities. Comm. Math. Phys. 295, 293–326 (2010)

\bibitem{Jang2009}
Jang, J.: Vlasov-Maxwell-Boltzmann diffusive limit. Arch. Ration. Mech. Anal. 194, no. 2, 531–584. (2009)

\bibitem{Jang2021}
Jang, J., Kim, C.: Incompressible Euler limit from Boltzmann equation with diffuse boundary condition for analytic data. Ann. PDE 7, no. 2, Paper No. 22, 103 pp. (2021)


\bibitem{Ishige2009}
Ishige, K.: Gradient estimates for the heat equation in the exterior domains under the Neumann boundary condition. Differential and Integral Equations Volume 22, Numbers 5-6, , 401-410 (2009)


\bibitem{Kato1984}
Kato, T. : Strong LP-Solutions of the Navier-Stokes Equation in R m, with Applications to Weak Solutions. Math. Z. 187, 471-480 (1984)


\bibitem{Kozono2013}
Kozono, H., Yanagisawa, T.: Generalized Lax-Milgram theorem in Banach spaces and its application to the elliptic system of boundary value problems, Manuscripta Mathematica, Vol. 141, pages 637–662, (2013)

\bibitem{Lachowicz1987}
Lachowicz, M.: On the initial layer and the existence theorem for the nonlinear Boltzmann equation. Math. Methods Appl. Sci. 9(3), 27–70 (1987)


\bibitem{Lions2001}
Lions, P.-L., Masmoudi, N.: From the Boltzmann equations to the equations of incompressible fluid mechanics. I, II. Arch. Ration. Mech. Anal. 158(3), 173–193, 195–211 (2001)

\bibitem{Leoni2009}
Leoni, G.: A First Course in Sobolev Spaces, Graduate Studies in Mathematics, vol. 105. American Mathematical Society, Providence (2009)

\bibitem{Masi1989}
De Masi, A., Esposito, R., Lebowitz, J.L.: Incompressible Navier–Stokes and Euler Limits of the Boltzmann Equation. Comm. Pure Appl. Math. 42, 1189–1214 (1989)


\bibitem{Masmoudi2003}
Masmoudi, N., Saint-Raymond, L.: From the Boltzmann equation to the Stokes–Fourier system in a bounded domain. Comm. Pure Appl. Math. 56(9), 1263–1293 (2003)



\bibitem{Nishida1978}
Nishida, T.: Fluid dynamical limit of the nonlinear Boltzmann equation to the level of the compressible Euler equation. Comm. Math. Phys. 61, 119–148 (1978)


\bibitem{Saint-Raymond2009}
Saint-Raymond, L.: Hydrodynamic limits of the Boltzmann equation. Lecture Notes in Mathematics, no. 1971. Springer, Berlin (2009)





\bibitem{Sone1971}
Sone, Y.: Asymptotic Theory of Flow of a Rarefied Gas over a Smooth Boundary II. In: Rarefied Gas Dynamics. Vol. II, pp. 737–749, ed. by D. Dini. Pisa: Editrice Tecnico Scientifica (1971)


\bibitem{Sone1987}
Sone, Y., Aoki K.: Steady Gas Flows Past Bodies at Small Knudsen Numbers - Boltzmann and Hydrodynamics Systems, Trans. Theory Stat. Phys., 16, 189-199 (1987)


\bibitem{Strain2006}
Strain, R. M.
: The Vlasov-Maxwell-Boltzmann system in the whole space.
Comm. Math. Phys. 268, 543–567 (2006)


\bibitem{Ukai1983}
Ukai, S., Asano, K.: Steady solutions of the Boltzmann equation for a gas flow past an obstacle I Existence. Arch. Ration. Mech. Anal. 84, 249–291 (1983)


\bibitem{Ukai1986}
Ukai, S., Asano, K.: Steady solutions of the Boltzmann equation for a gas flow past an obstacle. II. Stability. Publ. Res. Inst. Math. Sci. 22, 1035–1062 (1986)


\bibitem{Ukai2009}
Ukai, S., Yang, T., Zhao, H.: Stationary solutions to the exterior problems for the Boltzmann equation. I. Existence. Discrete Contin. Dyn. Syst. 23, 495–520 (2009)


\bibitem{Varnhorn2022}
Varnhorn, W. On the Poisson equation in exterior domains. Constructive Mathematical Analysis 5, No. 3, pp. 134-140 (2022)


\bibitem{Wigner2000}
Wiegner, M.: Decay estimates for strong solutions of the Navier–Stokes equations in exterior domains. Ann. Univ. Ferrara Sez. VII.
Sc. Mat. 46, 61–79 (2000)


\bibitem{Xia2017}
Xia J.: Exterior problem for the Boltzmann equation with temperature difference: asymptotic stability of steady solutions. J. Differential Equations. 262, 3642–3688 (2017)

\bibitem{Yu2005}
Yu, S.-H.: Hydrodynamic limits with shock waves of the Boltzmann equation. Comm. Pure Appl. Math. 58, 409–443 (2005)













\end{thebibliography}
\end{document}